\documentclass{article}
\usepackage[utf8]{inputenc}
\usepackage[margin = 1in]{geometry}
\usepackage{xparse}
\usepackage{amsmath}
\usepackage{amssymb}
\usepackage{amsfonts}
\usepackage{graphicx}
\usepackage{diagbox}
\usepackage{color}
\usepackage{multicol}
\usepackage{epstopdf}
\usepackage{bm}
\usepackage{booktabs}
\usepackage[sort&compress]{natbib}

\usepackage[labelformat=simple]{subcaption}

\graphicspath{{figures/}}
\usepackage[hidelinks]{hyperref}
\usepackage{amsthm}

\usepackage[format=plain,
            labelfont={it},
            textfont=it]{caption}

\newtheorem{proposition}{Proposition}

\newtheorem{lemma}{Lemma}

\usepackage[title]{appendix}

\renewcommand\theequation{\thesection.\arabic{equation}}

\renewcommand{\theequation}{\arabic{section}.\arabic{equation}}

\def\pw{0.13\textwidth}

\newcommand{\bsub}{\begin{subequations}}
\newcommand{\esub}{\end{subequations}$\!$}

\newcommand{\eps}{\varepsilon}
\newcommand{\Om}{\Omega}
\newcommand{\Gam}{\Gamma}
\newcommand{\x}{\mathbf{x}}
\newcommand{\K}{\mathcal{K}}
\newcommand{\y}{\mathbf{y}}
\newcommand{\X}{\mathbf{X}}
\newcommand{\T}{\mathbf{T}}
\newcommand{\MI}{\mathbf{M}_1}
\newcommand{\MII}{\mathbf{M}_2}

\newcommand{\dd}{\,\mathrm{d}}
\newcommand{\GN}{G}
\newcommand{\bxi}{\bm{\xi}}

\DeclareMathOperator{\dist}{dist}

\newcommand{\stacksum}{\sum_{\stackrel{i=1}{i\neq j}}^{N}}

\def\B{{\mathcal B}}
\def\pa{\partial{\mathcal B}}
\def\R{\mathbb{R}}
\def\C{\mathbb{C}}
\def\D{{\mathcal D}}

\title{ The Mean First-Reaction Time on a Thin Filament in a 3-D Domain}

\numberwithin{equation}{section}

\begin{document}

\author{Denis S. Grebenkov \thanks{Laboratoire de Physique de la Mati\'{e}re
    Condens\'{e}e, CNRS -- Ecole Polytechnique, Institut Polytechnique
    de Paris, 91120 Palaiseau, France. Email:
    {\tt denis.grebenkov@polytechnique.edu}} , \and
    Theodore Kolokolnikov \thanks{Dept. of Mathematics, Dalhousie University,
    Halifax, Nova Scotia, B3H 4R2, Canada. Email: {\tt
      tkolokol@mathstat.dal.ca}} , \and
    Alan E. Lindsay \thanks{NSF-Simons National Institute for Theory and Mathematics in Biology, Chicago IL.}  
    \thanks{Dept. of Applied and Computational Mathematics
    and Statistics, University of Notre Dame, Notre Dame, IN, 46556,
    USA.  Email: {\tt a.lindsay@nd.edu}} , \and
 Michael J. Ward \thanks{Department of Mathematics,  University of British Columbia, 
 Vancouver, British Columbia, V6T 1Z2, Canada. Email: {\tt ward@math.ubc.ca}
 (corresponding author)}} \date{\today}

\maketitle

\begin{abstract}
We study the mean first-reaction time (MFRT) of a Brownian particle in
a bounded three-dimensional domain $\Omega$ with a reflecting boundary
$\partial\Omega$, that contains a thin reactive filament. The filament
is a tube of cross-sectional diameter $\eps\ll 1$ around a smooth
closed curve $\Gamma$, and its surface is partially reactive with
reactivity $b$. As $\eps\to 0$ the filament reduces to a line source
in the outer problem. Its cross-section enters only through a
logarithmic reactive capacitance $d_c(b)$, so that the MFRT is
$\mathcal{O}(-\log[\eps d_c])$. This contrasts with the
$\mathcal{O}(\eps^{-1})$ scaling for spatially localized traps of
radius $\eps$.  Using matched asymptotic expansions, we derive a
Fredholm integral equation on $\Gamma$ for the line density of the
reactive flux. It sums all powers of $\nu\equiv -1/\log[\eps d_c]$ and
involves two geometric quantities: the regular part $R$ of the Neumann
Green's function, which encodes the interaction of $\Gamma$ with
$\partial\Omega$, and an intrinsic self-energy of $\Gamma$ that
depends only on its shape. The framework extends to multiple
filaments, splitting probabilities, a spatially heterogeneous
diffusivity, and the principal eigenvalue of the Laplacian. For the
unit sphere we solve the integral equation with a simple trapezoidal
Nystr\"om scheme, and we validate the predictions for rings, linked
rings and torus knots against full 3-D finite-element solutions. For
an equatorial ring we obtain the MFRT in closed form and determine the
optimal ring radius. We compare it with the optimal radius for $N$
small targets equally spaced on a ring. Finally, we show that
fragmenting a filament into $N$ small microtraps is asymptotically
equivalent to a continuous filament with an effective radius
determined by a periodic capacitance problem. This result connects the
logarithmic scaling of filaments to the algebraic scaling of
point-like traps.
\end{abstract}


\baselineskip=16pt 

\section{Introduction} \label{section:1}

\color{black} The first time a diffusing particle or agent hits a
reactive target is a fundamental quantity with applications in ecology, cell biology and physics
\cite{CHOU2014,redner2001guide,Grebenkov2024TargetSearch}. In
a cellular setting, the diffusing agent may be a signaling molecule
that is transported by diffusion to a reactive site (e.g. receptor) to
complete or initiate a biological event \cite{Fok2009, Lawley2023,
  BL2018,BressloffSchumm2022,Bressloff2025}. In an ecological
setting, first passage time theory is an essential tool for
characterizing the time needed for individuals to find a mate, food or
refuge.

In many of these examples the reactive sites are small regions centered on fixed separated points. Examples include
receptors on the cell membrane \cite{Miles2020,Coombs2009} or
trafficking by nuclear pore complexes, the small apertures on the
nucleus whose role is to transport cargo between the cytoplasm and
nucleoplasm \cite{BL2025,Holcman2004a,Holcman2014}. 
 When the target regions are small balls centered around fixed
points, an extensive literature determining the arrival
statistics of diffusing molecules exists
\cite{Bressloff2024_JMBIO,Pillay2010,LTK2015,KTW,Kurella2015,Lawley2023,BL2018}.

In this work we develop an extension of this theory to a class of
important biological structures characterized by elongated, thin
targets. Prominent examples include microtubule filaments,
actin fibers and strands of DNA
\cite{Mlynarczyk2019,Ciocanel2017,NewbyBressloff2009,NewbyBressloff2013,Maxian2024}. Related studies of ecological-scale transport have examined how
arrival statistics are influenced by linear features such as roads,
utility lines or fire breaks \cite{McKenzie2012,HDL2025,DML2026}. In 3-D problems in cell
biology, these filamentous structures provide essential functionality by
modulating search and signaling processes
\cite{Chatterjee2023,Kim2022PRE}. The goal of this paper is to present
an asymptotic analysis of the 3-D mean first-reaction time of a
diffusing molecule to thin closed filaments in the limit of vanishing
width. Before stating our main results, we describe the geometric
scenario under consideration.

We let $\Om\subset\mathbb{R}^3$ be a bounded domain with smooth
reflecting boundary $\partial\Om$, and we let $\Gam\subset\Om$ be a
smooth simple closed curve parameterized in terms of arclength $s$ as
$\X(s)$, where $s\in[0,L)$ and $L=|\Gam|$ is the length of the curve
(see Fig.~\ref{fig:intro_schem}).
\begin{figure}
    \centering
    \includegraphics[width = 0.95\textwidth]{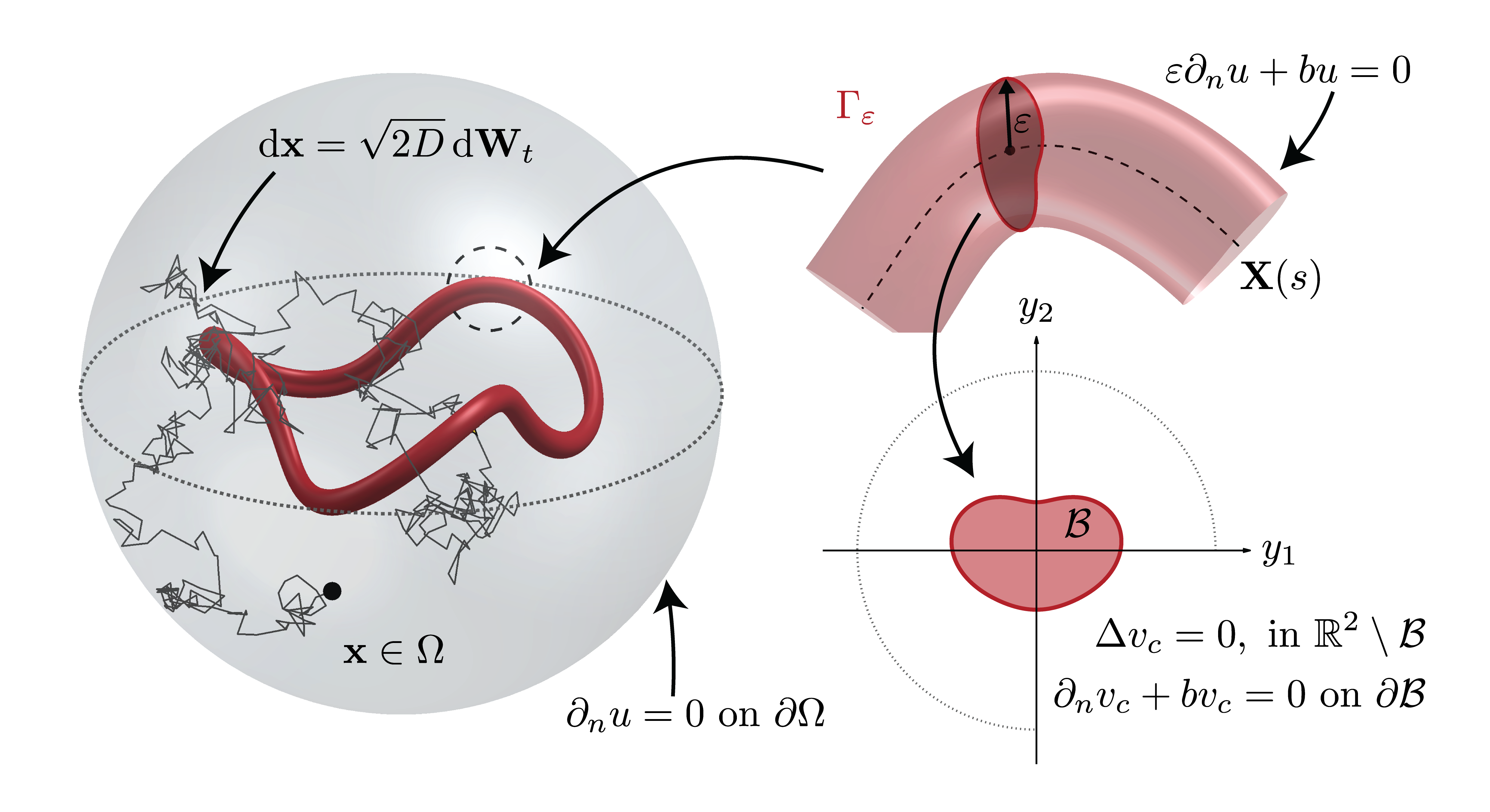}
    \caption{ Illustration of the considered first-reaction problem
      of a 3-D Brownian particle on a
      thin closed filament $\Gamma_{\eps}$. The enclosing region
      $\Omega\subset\mathbb{R}^3$ has reflecting boundary conditions
      on $\partial\Omega$. The centerline of the filament is given by
       a simple curve $\textbf{X}(s)$ and the fixed cross-sectional rescaled profile
      is given by $\mathcal{B}$.\label{fig:intro_schem}}
\end{figure}
The curve is assumed to lie completely within $\Om$ in the sense that
$\dist(\Gamma,\partial\Omega)={\mathcal O}(1)$. Around $\Gam$ we place
a reactive thin filament of diameter ${\mathcal O}(\eps)$ with a
specified cross-sectional shape ${\mathcal B}_\eps$. We assume the filament is well separated from itself, in the sense that if two points $\mathbf{X}(s)$ and $\mathbf{X}(s')$ are separated by an $\mathcal{O}(1)$ arclength along $\Gamma$, then $|\mathbf{X}(s)-\mathbf{X}(s')| = \mathcal{O}(1)$. To characterize
the tubular geometry around $\Gam$ we introduce the so-called Bishop
frame \cite{bishop1975,hester2023} which is used extensively in slender body theory \cite{ohm2025hierarchybloodvesselmodels,Ohm2024,MOS2019,Koens_Lauga_2018}. We let $\T(s)=\X^{\prime}(s)$
be the unit tangent vector of the centerline and we define two unit
vectors $\MI(s)$ and $\MII(s)$ in the plane orthogonal to $\T(s)$, so that
$\lbrace{\T(s),\MI(s),\MII(s)\rbrace}$ form an orthonormal basis. To
ensure that the reference frame along the curve does not twist, we
enforce that $\MI^{\prime}(s){\bf \cdot} \MII(s)=0$ and
$\MII^{\prime}(s){\bf \cdot} \MI(s)=0$. The reactive filament is
locally defined by
\begin{equation}
  \Gam_\eps= \left\{ \X(s) + \eta_1\MI(s) + \eta_2\MII(s) \in
  \mathbb{R}^3 \;\middle\vert{}\; s \in [0, L), \; (\eta_1, \eta_2) \in
  {\mathcal B}_\eps \right\}\,. \label{intro:tube}
\end{equation}
The mean first-reaction time (MFRT) $u(\x)$ of a particle, initially
at position $\x\in\Omega\setminus\Gamma_{\eps}$, undergoing Brownian
motion with diffusivity $D$, with a reactive filament $\Gamma_{\eps}$
satisfies
\begin{equation}\label{intro:mfrt}
\Delta_{\x} u=-\frac{1}{D} \,, \quad \x\in \Om\setminus\Gam_\eps\,, \qquad
\eps \partial_n u + b u =0 \,, \quad \x \in \partial\Gam_\eps\,,\qquad
\partial_n u=0\, ,\quad \x \in \partial\Om \,.
\end{equation}
Here $b>0$ is the reactivity of the filament. The strong reaction case
$b\gg1$ results in an instantaneous reaction when the particle reaches
$\Gamma_{\eps}$ while the weak reaction limit $b\ll1$ requires a
particle to spend a long time near $\Gamma_{\eps}$ for a reaction to
occur. The quantity $\partial_n$ denotes the outward normal
derivative.  The global MFRT (GMFRT) is defined as the spatially
averaged MFRT over a uniform distribution of the initial point,
\begin{equation}\label{intro:mfrt_ave}
  \bar{u} = \frac{1}{|\Om\setminus\Gam_\eps|} \int_{\Om\setminus\Gam_\eps}
  u(\x)  \dd \x \,.
\end{equation}

We will use the method of matched asymptotic expansions to analyze
\eqref{intro:mfrt} in the limit $\eps\to 0$ of small filament diameter and
determine $\bar{u}$. Our asymptotic analysis leads to a Fredholm
integral equation of the second kind for the strength of a line
singularity, which we refer to as the line density supported on
$\Gamma$, and a coupling condition to determine $\bar{u}$. The
derivation of our integral equation for the line density involves the
Neumann Green's function $\GN(\x;\bxi)$ in $\Om$, which for
$\bxi\in \Omega$ is given as the unique solution to
\bsub \label{intro:Gdef}
\begin{equation} 
    \Delta_{\x} \GN =
    \frac{1}{|\Om|}-\delta(\x-\bxi)\,, \quad \x\in \Omega \,;
    \qquad \partial_n\GN =0\,,\quad \x\in \partial\Omega \,;
    \qquad \int_\Om \GN(\x;\bxi) \dd\x =0 \,. \label{intro:Gdef1}
\end{equation}
It can be decomposed into a singular part and a smooth regular part
$R(\x;\bxi)$ as
\begin{equation}  \label{intro:Gdef2}
  \GN(\x;\bxi) =\frac{1}{4\pi|\x-\bxi|}+R(\x;\bxi)\,.
\end{equation}
\esub 
The Neumann Green's function is known analytically when $\Omega$ is a
sphere \cite{ChevWard2010} and can be computed numerically for general
shapes
\cite{lindsay20263D_GFun,chakraborty2025fastintegralmethodsneumann}. In
terms of this Neumann Green's function, in Sec.~\ref{sec2:asy} we use
inner-outer asymptotic matching to show that the line density and the
GMFRT $\bar{u}$ are determined in terms of the solution to a Fredholm
integral equation defined over the centerline $\Gam$ of the filament. 

As an example of our results, we derive in Sec.~\ref{sec:twoterm} the
behavior of the GMFRT for a single reactive filament to be
\begin{equation}\label{intro:prop:2term_chi} 
    \bar{u} \sim \frac{|\Om|}{2\pi D |\Gam|} \left(-\log[\eps d_c(b)] +
    \frac{2\pi}{|\Gam|}  \int_{\Gam} \int_{\Gam}
    R\left(\X(s); \X(s^{\prime}) \right) \dd s \dd s^{\prime} +
    \frac{1}{2|\Gam|} \int_{\Gam} {\mathcal E}(s) \dd s +
    {\mathcal O}(\nu)  \right) \,, \qquad \nu = \frac{-1}{\log[ \eps d_c]},
\end{equation}
as $\eps\to0$. Here  $d_c$ is the logarithmic reactive
capacitance, which is equal to $e^{-1/b}$ for a circular
cross-section. For general
cross-sectional profiles $\mathcal{B}$ and reactivities $b$, we provide in Appendix \ref{app:inner} an algorithm to calculate $d_c$. A notable
feature of \eqref{intro:prop:2term_chi} is the scaling law
$\bar{u}= \mathcal{O}(\log \eps)$ as $\eps\to0$ , which is the typical 2-D scaling law for point absorbers \cite{KTW}, and not
the $\bar{u}= \mathcal{O}(\eps^{-1})$ behavior associated with absorbers
centered around discrete points in 3-D domains \cite{ChevWard2010}. In
\eqref{intro:prop:2term_chi} the term $\int_{\Gam} \int_{\Gam}
R\left(\X(s); \X(s^{\prime}) \right) \dd s \dd s^{\prime}$ describes
the interaction of the filament with the domain $\Omega$ while
$\int_{\Gam} {\mathcal E}(s) \dd s$ describes the interaction of the
filament with itself through an \emph{intrinsic self-energy} term
${\mathcal E}(s)$ that we define in \eqref{sec2:intrinsic} and which
depends only on $\Gamma$.

Our approach parallels slender-body theory for potential and Stokes
flow \cite{HandelsmanKeller1967,Batchelor1970,KellerRubinow1976}, in
which a thin body is replaced by a line distribution of singularities
whose strength satisfies an integral equation with a local logarithmic
term and a nonlocal interaction along the centerline; see
\cite{MOS2019,Koens_Lauga_2018,Ohm2024} for its rigorous justification
and connection to boundary integral methods. We remark that in
\cite{Gjerde2019} and \cite{Gjerde2021} a finite element method was
used to numerically compute solutions to Poisson-type problems with
pre-specified line densities after analytically extracting an
appropriate singular behavior from the solution to be computed.

In Sec.~\ref{sec5:multiple} we extend the analysis to $N$ disjoint
filaments $\Gamma_{\eps,j}$ with centerlines $\Gamma_j$ of length
$L_j$, cross-sections $\mathcal{B}_j$ and reactivities $b_j$. Our main
result (Proposition~\ref{mprop:int_eq}) is that $\bar{u}\sim\chi$ and 
away
from the filaments, the solution of \eqref{intro:mfrt} is given by
\[u(\mathbf{x})\sim\chi-2\pi\sum_{j=1}^N\int_{\Gamma_j}A_j(s)\,
G(\mathbf{x};\mathbf{X}_j(s))\,\dd s, \qquad \dist(\x, \Gamma_j) \gg \eps, \quad j = 1,\ldots, N.\]
The line densities $A_j(s)$ on $\Gamma_j$ and the constant
$\chi$ solve the coupled system
\begin{equation}\label{intro:opsys}
  \mathcal{K}_j A_j + \sum_{\substack{i=1\\i\neq j} }^N\mathcal{G}_{ji}A_i = \chi\,,
  \quad j=1,\ldots,N\,;\qquad
  \sum_{j=1}^{N}\int_{\Gamma_j}A_j(s)\,\dd s=\frac{|\Omega|}{2\pi D}\,.
\end{equation}
The self-interaction operator $\mathcal{K}_j$ on $\Gamma_j$ and the
coupling operator $\mathcal{G}_{ji}$ between filaments $\Gamma_i$ and
$\Gamma_j$ are
\begin{subequations}\label{intro:ops}
\begin{align}
  (\mathcal{K}_jA)(s) &\equiv \Big(\frac{1}{\nu_j}+\frac{\mathcal{E}_j(s)}{2}\Big)A(s)
   -\frac12\int_{\Gamma_j}\frac{A(s)-A(s')}{|\mathbf{X}_j(s)-\mathbf{X}_j(s')|}\,\dd s'
   +2\pi\int_{\Gamma_j}A(s')\,R\big(\mathbf{X}_j(s);\mathbf{X}_j(s')\big)\,\dd s'\,,\\
  (\mathcal{G}_{ji}A)(s) &\equiv 2\pi\int_{\Gamma_i}A(s')\,G\big(\mathbf{X}_j(s);\mathbf{X}_i(s')\big)\,\dd s'\,,
\end{align}
\end{subequations}
with $\nu_j=-1/\log[\eps d_c(b_j)]$ and $\mathcal{E}_j$ the intrinsic
self-energy of $\Gamma_j$. The local term $1/\nu_j$ carries
information on the cross-sectional profile $\mathcal{B}_j$ and
reactivity $b_j$ and the self-energy terms are determined by the shape
of $\Gamma_j$. The terms involving $R$ and $G$ encode information on a
filament's interaction with the domain and with each other,
respectively. The line density $A_j(s)$ represents the local
flux arriving at the cross-section of the filament $j$ at $s$.  

The methods we develop here can be applied to several related problems
in first passage time theory. For example, the so called 
\emph{splitting probability} $\phi_k(\x)$ is defined as the
probability that a searcher starting at $\x$ is absorbed by filament
$k$ before any other and it has been used extensively in source
detection and signaling problems \cite{Miles2020,DOBRAMYSL2018,Lindsay2023a}. The
 splitting probabilities satisfy the elliptic problem,
\bsub\label{eq:intro_split-pde}
\begin{gather}
  \label{eq:intro_split-pde_a}  \Delta\phi_k=0\,, \quad
  \x\in\Omega\setminus \cup_{j=1}^{N} \Gamma_{\eps,j}\,;
  \qquad \partial_n\phi_k=0\,, \quad \x\in\partial\Omega\,,\\[5pt]
  \label{eq:intro_split-pde_b}  \eps\partial_n\phi_k+b_j\,(\phi_k-\delta_{jk})=0\,,
  \qquad
  \x\in\partial \Gamma_{\eps,j}\,, \quad j\in \lbrace{1,\dots,N\rbrace} \,,
\end{gather}
\esub 
with $\phi_k=\delta_{jk}$ on $\partial \Gamma_{\eps,j}$ when
$b_j=\infty$ and where $\delta_{jk}$ is the Kronecker symbol. In
Sec.~\ref{sec5:split}, we derive an asymptotic solution (Proposition \ref{prop:splitting}) to
\eqref{eq:intro_split-pde} in the limit as $\eps\to0$ in which the
solution is expressed in terms of densities supported over the
centerlines $\Gamma_j$ for $j = 1,\ldots, N$ which solve a system
related \eqref{intro:opsys} but with a modified right-hand side.

In Sec.~\ref{sec2:conc_ring} we analyze scenarios where the bounding
domain is a sphere and the filaments are 
arranged circular rings of radius $0<r_0<1$. In the case of a single
ring centered at the origin, or two rings arranged symmetrically on
the planes $z = \pm z_0$, we are able to solve the integral system in
closed form. In doing so, we identify GMFRT optimizing configurations
at finite $r_0$ values where $r_0$ is the radius of the rings.

In Sec.~\ref{sec4:numerics}, we describe an implementation of the
trapezoidal rule to discretize and solve the coupled integral equation
system \eqref{intro:opsys}. The accuracy of the expansions is
confirmed by full finite element solutions of the $\eps>0$ problems
\eqref{intro:mfrt} and \eqref{eq:intro_split-pde}. We show results on
several non-ring filaments and identify shielding behavior, in which distal
sections of highly undulating filaments exhibit locally reduced
capture rates.

In Sec.~\ref{sec5:extensions}, we demonstrate the potential for our
approach to extend to related problems in capture theory. One such
problem concerns the determination of the MFRT $u(\x)$ when the
environment has a spatially varying diffusivity $D(\x)>0$ (see
\cite{LawleyBMB2026}). In Sec.~\ref{sec5:varD}, we consider the
problem (assuming an It\^{o} interpretation)
\begin{equation}\label{intro:varD:mfrt}
\Delta u=-\frac{1}{D(\x)} \,, \quad \x\in \Om\setminus\Gam_\eps\,;\qquad 
\eps \partial_n u + b u =0 \,, \quad \x \in \partial\Gam_\eps\,;\qquad
\partial_n u=0\, ,  \quad \x \in \partial\Om \,.
\end{equation}
We determine in the limit as $\eps\to0$ a corresponding reduction of
\eqref{intro:varD:mfrt} to an integral equation supported on the
filament centerline. A second problem considered is the determination
of the principal eigenvalue $\lambda_0$ and eigenfunction $u_0(\x)$ of
the Laplacian for a domain with a reactive filament. These quantities
solve
\begin{equation}\label{Intro:eig:main}
\Delta u_0 + \lambda_0 u_0 =0 \,, \quad \x\in
\Om\setminus\Gam_\eps\,; \qquad\eps \partial_n u_0 + b u_0 =0 \,, \quad \x
\in \partial\Gam_\eps\,; \qquad \partial_n u_0=0 \,, \quad \x\in \partial\Om\,.
\end{equation}
The principal eigenvalue is simple (multiplicity $1$) and the
corresponding eigenfunction can be chosen positive. The survival
probability of the Brownian particle is proportional to
$e^{-D\lambda_0 t}$ and so $D\lambda_0 \propto 1/\bar{u}$. In
Sec.~\ref{sec5:eig} we derive a log-summed expansion of solutions to
\eqref{Intro:eig:main} which is accurate to all orders in $\nu =
-1/\log[\eps d_c]$. This result is a hybrid numerical asymptotic
solution formulated in terms of a nonlinear problem for $\lambda_0$
involving the associated Helmholtz Green's function. A two-term
expansion of $\lambda_0$ and the associated eigenfunction in terms of
$\nu$ is developed that gives an explicit, but 
slower converging expression involving the Neumann Green's
function \eqref{intro:Gdef}. 

In Sec.~\ref{sec6:ball-ring} we
investigate the fragmentation of a single absorbing filament into
multiple absorbing microtraps. We derive and solve a new periodic
capacitance problem, which captures the transition between the $\bar{u}
= \mathcal{O}(\log\eps)$ scaling behavior of the MFRT with an
absorbing filament and a $\bar{u}= \mathcal{O}(\sigma^{-1})$ scaling
regime of the MFRT with multiple point absorbers of radius
$\mathcal{O}(\sigma)$. Finally, in Sec.~\ref{sec7:discussion}, we
summarize results and discuss a few avenues for future study arising from
this work.

\section{Asymptotic reduction to an integral equation}\label{sec2:asy}

\subsection{Inner solution}

We first consider the  inner (or local) solution for
\eqref{intro:mfrt} near the filament. In terms of the tubular
coordinate system \eqref{intro:tube}, the Laplacian of an arbitrary
function $\Phi(\x)$ transforms to (see Appendix
\ref{app:inner} and
\cite{hester2023})
\begin{equation}\label{sec1:lapl1}
  \Delta_{\x} \Phi = \left(\partial_{\eta_1\eta_1} +
  \partial_{\eta_2\eta_2} \right) \Phi -\frac{\kappa_1}{h_{s}}
  \partial_{\eta_1}\Phi - \frac{\kappa_2}{h_{s}} \partial_{\eta_2}
  \Phi + \frac{1}{h_s^2} \partial_{ss} \Phi
  +\frac{1}{h_s^3} \left(\eta_1 \kappa_1^{\prime} + \eta_2
  \kappa_2^{\prime}\right) \partial_{s} \Phi \,,
\end{equation}
where the scale factor is $h_s=1-\kappa_1\eta_1-\kappa_2\eta_2$. Here
$\kappa_1=\kappa_1(s)$ and $\kappa_2=\kappa_2(s)$ are the curvatures
along each respective normal direction, defined by
$\MI^{\prime}(s)=-\kappa_1(s)\T(s)$ and
$\MII^{\prime}(s)=-\kappa_2(s)\T(s)$, where $\T(s)=\X^{\prime}(s)$ is
the unit tangent along the filament centerline. A sufficient condition for this local description to be  valid all along the filament is that
$\sqrt{\eta_1^2 + \eta_2^2} < {1/\kappa(s)}$ for $0\le s\le L$, where 
$\kappa(s)\equiv \sqrt{\kappa_1^2+\kappa_2^2} = |\X^{\prime\prime}(s)|$ is the curvature of the filament centerline.

Upon re-scaling the normal plane by $\eps$ by defining
$y_1={\eta_1/\eps}$ and $y_2={\eta_2/\eps}$, and labeling the inner
variable $\y=(y_1,y_2)^{T}$ and $|\y|=\rho\equiv
\sqrt{y_1^2+y_2^{2}}$, we use
\eqref{sec1:lapl1} to obtain that the MFRT in \eqref{intro:mfrt}
near the filament, written as $v(y_1,y_2,s)$ in the inner
variables, satisfies
\begin{equation}\label{sec1:lapl2}
  \Delta_{\x} u = \frac{1}{\eps^2} \Delta_{\y} v - \frac{\kappa_{1}}{\eps}
  v_{y_1} - \frac{\kappa_2}{\eps} v_{y_2} + {\mathcal O}(1) 
      = - \frac{1}{D} \,.
\end{equation} 
As a result, by expanding $v=v_0+\eps v_1+\ldots$, we obtain that the
local solution near the filament is a 2-D harmonic solution
$\Delta_{\y} v_0=0$.  In terms of the local variable $\y$, we denote
the  rescaled filament cross-section as
${\mathcal B}\equiv {{\mathcal B}_{\eps}/\eps}$.

The canonical inner problem associated with the filament cross-section
is, at each point $s$ on $\Gam$, the following exterior problem in
${\mathbb R}^2$ for $v_{c}(\y)$, which is made unique by
prescribing a unit logarithmic growth at infinity:
\begin{equation}\label{sec2:in_prob}
  \begin{split}
\Delta_\y v_c &=0\,, \quad \y\in {\mathbb R}^2\setminus {\mathcal B} \,,\\
\partial_{n} v_{c} + b v_{c} =0 \,, \quad \y\in \partial {\mathcal B}\,; &
\qquad v_{c} \sim \log|\y| - \log\left[ d_{c}(b)\right] + o(1) \,, \quad
\mbox{as} \quad |\y|\to \infty \,,
\end{split}
\end{equation}
where 
$\partial_n$ points into ${\mathcal B}$.  We refer to the constant
$d_{c}(b)$ as the {\em logarithmic reactive capacitance} of ${\mathcal B}$.

We remark that since the filament has a uniform cross-section, the
rescaled shape ${\mathcal B}$ and hence $d_{c}(b)$ is
independent of the position $s$ along the filament centerline. In
addition, since $d_{c}(b)$ is invariant under rotations of ${\mathcal
B}$, we can allow for the filament centerline $\X(s)$ to exhibit
twist, without influencing $d_{c}(b)$.

Only for a few special cases is $d_{c}(b)$ known analytically.
When ${\mathcal B}$ is the unit disk we have
$d_{c}(b)=e^{-{1/b}}$.  When $b=\infty$, $d_{c}(\infty)$ can
be determined in terms of the conformal map of ${\mathcal B}$ to the
unit disk. In Table 1 of \cite{Kurella2015} explicit formulas for
$d_{c}(\infty)$ are given for various simple choices of ${\mathcal
  B}$. However, no such results are typically available for finite
reactivities for non-circular shapes.  In Appendix \S
\ref{app:steklov} we provide a  methodology to numerically determine
$d_{c}(b)$ from a Steklov eigenfunction representation, as well as its
behavior for $b\ll 1$.

In terms of the canonical solution $v_c$, the leading-order inner
solution is taken as
\begin{equation}\label{sec2:v0_1}
  v_{0} = A(s) v_{c} \,,
\end{equation}
for some strength $A(s)$ that depends on the position along the
filament, and which must be determined by matching to the outer
solution.  In this way, the far-field behavior of the inner solution,
written in the outer variable $\y=\boldsymbol{\eta}{/\eps}$, is
\begin{equation}\label{sec2:inn_match}
  v = A(s) \left(
  \log\left(|{\boldsymbol{\eta}|/\eps}\right)-\log\left[d_{c}(b)\right] \right) + o(1) =
   A(s) \log|\boldsymbol{\eta}| + {A(s)/\nu} + o(1) \,, \quad \mbox{where} \quad
  \nu\equiv {-1/\log[\eps d_{c}(b)]} \,.
\end{equation}
By applying the 2-D divergence theorem on \eqref{sec2:in_prob} we
identify that
\begin{equation}\label{sec2:inn_A}
\quad  A(s) = -\frac{1}{2\pi}
\int_{\partial{\mathcal B}} \partial_n v_{0} \dd \ell  \,,
\end{equation}
where $\dd \ell$ represents arclength along $\partial{\mathcal B}$.
The far-field behavior \eqref{sec2:inn_match} will provide the
matching condition for the outer solution constructed in \S
\ref{sec2:outer}.

To interpret the role of $A(s)$, we estimate the diffusive flux $J$ into
the filament through its boundary. In terms of the tubular coordinates
\eqref{intro:tube}, and where $\ell_\eps$ denotes arclength along
$\partial{\mathcal B}_\eps$, we have
\begin{equation}\label{sec2:flux_1}
  J  = \int_{\partial \Gamma_{\eps}} (-D\partial_n u) \dd S
  = - D \int_{\Gam} \left( \int_{\partial {\mathcal B}_{\eps}}
  \left(1-\kappa_1(s)\eta_1 - \kappa_2(s) \eta_2\right)\partial_{\eta}
  u  \dd \ell_{\eps}
  \right) \dd s\,,
\end{equation}
where the curvature terms arise from the Jacobian between tubular and
Cartesian coordinates. In terms of the inner variables
$y_1={\eta_1/\eps}$, $y_2={\eta_2/\eps}$, $d \ell_{\eps}=\eps d\ell$,
$v\sim v_0$,  $\partial_\eta=\eps^{-1}\partial_n$, and using
\eqref{sec2:inn_A} in \eqref{sec2:flux_1}, we obtain to
leading-order, after neglecting the ${\mathcal O}(\eps)$ curvature
terms in \eqref{sec2:flux_1}, that
\begin{equation}\label{sec2:flux}
  J \sim - D \int_{\Gam} \left( \int_{\partial {\mathcal B}} \partial_n
    v_0 \dd \ell \right) \dd s = 2\pi D \int_{\Gam} A(s) \, \dd s \,.
\end{equation}
This relates the total flux across the boundary of the filament to a line
integral of the unknown $A(s)$ along the filament centerline.

Finally, we remark that if the filament has a variable cross-section
and non-uniform reactivity along the filament centerline, then we need
only replace $b=b(s)$ and $\partial {\mathcal B}=\partial {\mathcal
  B}(s)$ in \eqref{sec2:in_prob} to obtain a parameterized 2-D inner
problem, for which $d_c=d_c(s)$ can be readily computed numerically at
each fixed $s$ using the Steklov methodology in Appendix
\ref{app:steklov}. We will not pursue this simple generalization of
our theory in the analysis below.

\subsection{Outer solution}\label{sec2:outer}

In the limit $\eps\to 0$, the filament shrinks to its centerline:
$\Gamma_\eps\to \Gamma$. As a result, the outer solution exhibits a
line singularity of strength $A(s)$ to be determined
self-consistently.  In other words, when $\x$ approaches any
point of $\Gamma$, the solution becomes singular.  A common way to
incorporate such a line singularity on the ``degenerate boundary''
$\Gamma$ consists in introducing Dirac $\delta$ sources along
$\Gamma$:
\bsub \label{sec2:out_1}
\begin{equation}
    \Delta_{\x} u = -\frac{1}{D} + 2\pi \int_{\Gam} A(s^{\prime})
    \delta(\x-\X(s^{\prime})) \, \dd s^{\prime} \,, \quad \x\in \Omega\setminus
    \Gam \,; \qquad \partial_n u=0 \,, \quad \x\in
    \partial\Omega \,, \label{sec2:out_1a} 
\end{equation}
where we refer to $A(s)$ as the {\em line density}.  In this way, upon
matching the far-field behavior \eqref{sec2:inn_match} of the inner
solution near the filament to the outer solution, defined at
${\mathcal O}(1)$ distances from the filament, we obtain 
\begin{equation}
    u - A(s) \log|\x-\X(s)| \to \frac{A(s)}{\nu} \quad \mbox{as}
    \quad \x\to \X(s) \,, \quad s\in [0,L] \,. \label{sec2:out_1b}
\end{equation}
\esub 

The specification of the regular part of the singularity behavior in
\eqref{sec2:out_1b} is a constraint that will yield an integral
equation $A(s)$ that is accurate to all orders in the logarithmic
gauge $\nu={-1/\log\left[\eps d_c\right]}$.

Upon applying the divergence theorem to \eqref{sec2:out_1a} we derive
that
\begin{equation*}
  -\frac{|\Om|}{D} + 2\pi \int_{\Gam} A(s^{\prime}) \left(\int_{\Om}
  \delta\big(\x-\X(s^{\prime})\big) \dd \x\right) \dd s^{\prime} =0 \,,
\end{equation*}
which provides the constraint
\begin{equation}\label{sec2:cons}
  \int_{\Gam} A(s) \dd s = \frac{|\Om|}{2\pi D} \,.
\end{equation}

Assuming that \eqref{sec2:cons} holds, we now represent the solution
to \eqref{sec2:out_1} in terms of the Neumann Green's function $\GN$
of \eqref{intro:Gdef}. By using Green's second identity we derive that
\begin{align*}
  0=\int_{\Om} \left( \GN \Delta_{\x^{\prime}}u - u \Delta_{\x^{\prime}} \GN
  \right) \, \dd \x^{\prime}  = &- \frac{1}{D} \int_{\Om} \GN(\x^{\prime};\x)
  \dd \x^{\prime} + 2\pi \int_{\Gam} A(s^{\prime}) \left(
  \int_{\Om} \GN(\x^{\prime};\x) \delta(\x^{\prime}-\X(s^{\prime})) \dd
  \x^{\prime} \right) \, \dd s^{\prime} \\
  &  -\frac{1}{|\Om|} \int_{\Om} u(\x^{\prime}) \dd \x^{\prime}
  + u(\x) \,.
\end{align*}
Since $\GN(\x^{\prime};\x)=\GN(\x;\x^{\prime})$ and
$\int_{\Om} \GN(\x;\x^{\prime})\, \dd \x=0$, we obtain in terms of an unknown
constant $\chi$ that 
\begin{equation}\label{sec2:urep}
  u(\x) = \chi - 2\pi \int_{\Gam} A(s^{\prime}) \GN(\x;\X(s^{\prime}))\dd
  s^{\prime} \,, \quad \mbox{for} \quad \x \notin \Gam \,.
\end{equation}
Since the volume of the filament $\Gam_\eps$ is ${\mathcal O}(\eps^2)$, and
$\int_{\Om} \GN(\x;\x^{\prime})\, \dd \x=0$, we conclude from 
\eqref{intro:mfrt_ave} and \eqref{sec2:urep} that the {\em global MFRT} is
\begin{equation}\label{sec2:chi_ave}
   \bar{u} = \chi + o(1). \,
\end{equation}
The order of the error estimate in \eqref{sec2:chi_ave} arises from
neglecting the effect of the ${\mathcal O}(\eps)$ curvature correction
in the inner region.

Next, we will derive an integral equation formulation for $A(s)$ and
the constant $\chi$ by ensuring that the regular part of the singularity
structure in \eqref{sec2:out_1b} is satisfied.

To do so, we use \eqref{intro:Gdef2} to isolate the singular part of
$\GN$. Upon substituting \eqref{intro:Gdef2} into \eqref{sec2:urep}
we obtain that
\begin{equation}\label{sec2:uint_1}
  u(\x) = \chi - \frac{1}{2}\int_{\Gam}  \frac{A(s^{\prime})}{|\x-\X(s^{\prime})|}
  \dd s^{\prime} -2\pi \int_{\Gam} A(s^{\prime}) R(\x; \X(s^{\prime})) \dd
  s^{\prime} \,, \quad \mbox{for} \quad \x \notin \Gam \,.
\end{equation}
We now let $\X(s)$ be any target point on $\Gam$ and we let
$\x=\X_{\eta}(s)$  be any point at a distance $\eta$ from $\X(s)$ in the normal plane to $\Gamma$ at $\X(s)$.
In the limit $\eta\to 0$, we observe that the second integral in
\eqref{sec2:uint_1} has a well-defined limit, while the first integral
is divergent.
As such, we write $A(s^{\prime}) = (A(s^{\prime})-A(s)) + A(s)$ in
\eqref{sec2:uint_1} in order to obtain an integrand with a removable
singularity. Therefore, as $\x=\X_{\eta}(s) \to \X(s)$ we have
\begin{equation}\label{sec2:unit_2}
  u(\x) \sim \chi - \frac{A(s)}{2} \left( \int_{\Gam} \frac{1}{
    | \x - \X(s^{\prime})|} \dd s^{\prime} \right) + \frac{1}{2}
  \int_{\Gam} \frac{ A(s)-A(s^{\prime})}{|\X(s)-\X(s^{\prime})|} \dd s^{\prime}
  -2\pi \int_{\Gam} A(s^{\prime}) R(\X(s); \X(s^{\prime})) \dd
  s^{\prime} \,.
\end{equation}
In Appendix \ref{app:curve} we establish the leading-order singular behavior
\begin{equation}\label{sec2:csing}
  \int_{\Gam} \frac{\dd s^{\prime} }{|\X_{\eta}(s)-\X(s^{\prime})|} \sim
  -2\log\eta + {\mathcal O}(1) \,, \quad \mbox{as} \quad \eta\to 0\,.
\end{equation}
As a result, we conclude that \eqref{sec2:unit_2} has the following
asymptotic behavior as $\X_{\eta}(s)\to \X(s)$:
\begin{equation}\label{sec2:unit_3}
  u \sim  A(s) \log\eta + \chi - \frac{1}{2} A(s) {\mathcal E}(s) +
  \frac{1}{2}
  \int_{\Gam} \frac{ A(s)-A(s^{\prime})}{|\X(s)-\X(s^{\prime})|} \dd s^{\prime}
  -2\pi \int_{\Gam} A(s^{\prime}) R(\X(s); \X(s^{\prime})) \dd
  s^{\prime} \,.
\end{equation}  
Here ${\mathcal E}(s)$, referred to as the {\em intrinsic self-energy
  of the curve $\Gam$}, is defined by
\begin{equation}\label{sec2:intrinsic}
  {\mathcal E}(s) \equiv \lim_{\eta\to 0} \left( \int_{\Gam}
  \frac{\dd s^{\prime}}{
    |\X_{\eta}(s)-\X(s^{\prime})|}  + 2\log\eta \right) \,.
\end{equation}
In Appendix \ref{app:curve} we calculate ${\mathcal E}(s)$ when $\Gam$
is a planar loop and we provide a numerical algorithm to compute it
for an arbitrary simple closed curve. Finally, to obtain our integral
equation for $A(s)$ we enforce from \eqref{sec2:out_1b} that the
limiting behavior in \eqref{sec2:unit_3} must satisfy $u\sim
A(s)\log\eta + {A(s)/\nu}$ as $|\X_{\eta}(s)-\X(s)|=\eta\to 0$.
We summarize our result in the following proposition:

\begin{proposition}\label{prop:int_eq} 
Let $\Gam$ be any smooth simple closed curve in $\Omega$ parameterized
by arclength as $\X(s)$, which has an ${\mathcal O}(\eps)$
cross-sectional diameter. In the limit $\eps\to 0$, the asymptotic
solution to \eqref{intro:mfrt} in the outer region at ${\mathcal
O}(1)$ distances from $\Gam$ is
\begin{equation}\label{prop:uout}
  u(\x) \sim\chi - 2\pi \int_{\Gam} A(s) \GN(\x;\X(s))\dd
  s \,, \quad \mbox{for} \quad \x \notin \Gam \,,
\end{equation}
where the  line density $A(s)$ and the asymptotic approximation
$\chi$ to the volume-averaged GMFRT $\bar{u}$ are solutions to the
integral equation, defined for all $s\in [0,L]$ with $L=|\Gam|$, by
\begin{equation}\label{prop:inteq}
    (\mathcal{K}A)(s) \equiv \frac{A(s)}{\nu} + \frac{1}{2} A(s) {\mathcal E}(s) -
      \frac{1}{2}
  \int_{\Gam} \frac{ A(s)-A(s^{\prime})}{|\X(s)-\X(s^{\prime})|} \dd s^{\prime}
  +2\pi \int_{\Gam} A(s^{\prime}) R(\X(s); \X(s^{\prime})) \dd
  s^{\prime}=\chi \,, 
\end{equation}
subject to the constraint $\int_{\Gam} A(s) \dd s = {|\Om|/(2\pi D)}$.
In \eqref{prop:uout} and \eqref{prop:inteq}, $\GN$ is the Neumann
Green's function with regular part $R$ satisfying \eqref{intro:Gdef},
while ${\mathcal E}(s)$ is the intrinsic self-energy of $\Gam$ defined
by \eqref{sec2:intrinsic}. In addition, $\nu=-{1/\log\left[\eps
d_c\right]}$, where $d_c=d_c(b)$ is the logarithmic reactive
capacitance of the rescaled cross-sectional shape ${\mathcal
B}$ of the filament $\Gam_{\eps}$, as defined by the exterior 2-D
problem \eqref{sec2:in_prob}.
\end{proposition}

We remark that \eqref{prop:inteq} is a Fredholm integral equation of
the second kind on $\Gam$. By solving \eqref{prop:inteq} numerically
as $A=A(s;\nu)$ and $\chi=\chi(\nu)$ it effectively sums all powers of
the logarithmic gauge function $\nu$. A related approach to sum
logarithmic expansions in 2-D domains with spatially localized defects
was initially undertaken in \cite{WHK}.  Through the logarithmic
reactive capacitance $d_c(b)$, this gauge function encodes the
reactivity $b$ of the boundary of the filament and its cross-sectional
shape ${\mathcal B}$. Recall that in Appendix \ref{app:steklov} it is
shown how to numerically compute $d_c(b)$ by using a Steklov
representation. Moreover, in \eqref{prop:inteq} we emphasize that
${\mathcal E}(s)$ is independent of the shape of the confining domain
$\Omega$ and describes only the interaction of the curve $\Gamma$ with
itself. The effect of $\Omega$, and the relative position of $\Gam$
within $\Omega$, is encoded in \eqref{prop:inteq} by the regular part
$R$ of the Neumann Green's function.

\subsection{A two-term approximation}\label{sec:twoterm}

We now determine a more explicit two-term asymptotic approximation for
$\chi$ in powers of $\nu$. To do so, in \eqref{prop:inteq} and
\eqref{sec2:cons} we expand $A(s;\nu)$ and $\chi(\nu)$ as
\begin{equation}\label{sec2a:asy}
  A(s;\nu) = A_{0}(s) + \nu A_{1}(s) + \ldots \,, \qquad
  \chi(\nu)= \frac{\chi_0}{\nu} + \chi_{1} + \ldots \,.
\end{equation}
From the leading terms in \eqref{prop:inteq}, we obtain that
$A_0=\chi_0$ for all $s$. From the constraint \eqref{sec2:cons}, we get
\begin{equation}\label{sec2a:asy_0}
  A_0 = \chi_0 = \frac{|\Omega|}{2\pi D |\Gam|} \,.
\end{equation}
At next order, we obtain from \eqref{prop:inteq} and \eqref{sec2:cons}
that
\begin{equation}\label{sec2a:asy_1}
  \chi_1 = 2\pi A_0 \int_{\Gam} R\left(\X(s); \X(s^{\prime}) \right) 
  \dd s^{\prime} + A_1(s) + \frac{1}{2} A_0 {\mathcal E}(s) \,, \qquad
  \int_{\Gam} A_1  \dd s = 0 \,.
\end{equation}
To isolate the constant $\chi_1$, we integrate the first equation in
\eqref{sec2a:asy_1} over $\Gam$ to conclude that
\begin{equation}\label{sec2a:chi_1}
  \chi_{1} = \frac{A_0}{|\Gam|} \left( 2\pi \int_{\Gam} \int_{\Gam}
    R\left(\X(s); \X(s^{\prime}) \right) \dd s \dd s^{\prime} +
    \frac{1}{2} \int_{\Gam} {\mathcal E}(s) \, \dd s\right) \,.
\end{equation}
Upon substituting $\chi_1$ back into \eqref{sec2a:asy_1}, we determine
$A_1=A_{1}(s)$ as
\begin{equation}\label{sec2a:a1}
  A_1 = -\frac{A_0}{2} \left( {\mathcal E}(s) - \frac{1}{|\Gam|}
  \int_{\Gam} {\mathcal E}(s) \dd s \right) - 2\pi A_0 \left(
  \int_{\Gam} R\left(\X(s);\X(s^{\prime})\right)\dd s^{\prime} -
  \frac{1}{|\Gam|} \int_{\Gam} \int_{\Gam}
  R\left(\X(s); \X(s^{\prime}) \right) \dd s \dd s^{\prime} \right) \,,
\end{equation}
where $A_0$ is given in \eqref{sec2a:asy_0}. We summarize the results
of our two-term expansion for the volume-averaged MFRT $\chi$ and the
outer solution in \eqref{prop:uout} as follows:

\begin{proposition}\label{prop:2term} 
In the limit $\nu\to 0$,
we have the explicit two-term approximation:
\bsub\label{prop:2term_full}
\begin{equation}\label{prop:2term_chi}
    \chi \sim \frac{|\Om|}{2\pi D |\Gam|\nu} \left( 1 +
    \frac{2\pi \nu}{|\Gam|}  \int_{\Gam} \int_{\Gam}
    R\left(\X(s); \X(s^{\prime}) \right) \dd s \dd s^{\prime} +
    \frac{\nu}{2|\Gam|} \int_{\Gam} {\mathcal E}(s) \dd s +
    {\mathcal O}(\nu^2)  \right) \,,
\end{equation}
where $|\Gam|=L$. The corresponding outer solution for $\x\notin \Gam$
is
\begin{equation}\label{prop:2term_uout}
    u(\x) \sim \frac{|\Om|}{2\pi D |\Gam|\nu} \left( 1 +
     \frac{2\pi\nu}{|\Gam|}  \int_{\Gam} \int_{\Gam}
    R\left(\X(s); \X(s^{\prime}) \right) \dd s \dd s^{\prime} +
    \frac{\nu}{2|\Gam|} \int_{\Gam} {\mathcal E}(s) \dd s - 2\pi \nu
    \int_{\Gamma} G\left(\x;\X(s)\right) \dd s + {\mathcal O}(\nu^2) \right) \,.
\end{equation}
\esub
\end{proposition}
From \eqref{prop:2term_chi} we observe that this two-term expansion
depends on the average of both the intrinsic self-energy of $\Gam$ and
of the regular part of the Neumann Green's function.

\subsection{Multiple non-intersecting thin filaments}\label{sec5:multiple}

We now extend our asymptotic theory 
to derive a coupled set of integral equations for the case of multiple
non-intersecting simple closed curves.

We let $\Gam_{i}$ for $i\in \lbrace{1,\ldots,N\rbrace}$ be $N$ smooth,
simple, and closed non-intersecting curves with $\Gam_{i}\subset
\Omega$ and $\mbox{dist}(\Gam_i,\Gam_j)>0$ for $i\neq j$. Each curve
$\Gam_i$ is parameterized by $\X_{i}(s_i)$ where $s_i$ with $0\leq
s_i\leq L_i$, is arclength along $\Gam_{i}$, and $L_i$ is the length
of $\Gam_i$.  Each thin filament $\Gam_{\eps,i}$ around the centerline
$\Gam_{i}$ is assumed to have a uniform cross-section $\eps
{\mathcal B}_i$ at each $s_i\in [0,L_i]$, but its shape ${\mathcal
B}_i$ may be different for each filament. We assume that the boundary
of each filament has a uniform reactivity $b_i$. With the local
geometry defined as similar to \eqref{intro:tube}, we obtain $N$
canonical inner problems for each filament, with each given by the
solution to (see \eqref{sec2:in_prob})
\begin{equation}\label{sec5:in_prob}
  \begin{split}
    \Delta_\y v_{ci} &=0\,, \quad \y \in {\mathbb R}^{2} \setminus
          {\mathcal B}_{i} \,,\\
          \partial_{n} v_{ci} + b_i v_{ci} =0
          \,, \quad \y\in \partial {\mathcal B}_{i}\,; & \qquad v_{ci}
          \sim \log|\y| - \log\left[ d_{ci}(b_i)\right] + o(1) \,,
          \quad \mbox{as} \quad |\y|\to \infty \,.
\end{split}
\end{equation}
For $\Gam_i$, the far-field of $v_{ci}$ provides a matching condition
written in terms of outer variables and a line density $A_{i}(s_i)$ as
(see \eqref{sec2:inn_match})
\begin{equation}\label{sec5:inn_match}
  v_i \sim A_i(s_i) \left(
  \log\left({\eta/\eps}\right)-\log\left[d_{ci}(b_i)\right] \right) \sim
  A_{i}(s_i) \log\eta + {A_{i}(s_i)/\nu_i} \,, \qquad \forall s_i\in [0,L_i] \,,
\end{equation}
where $\nu_i\equiv {-1/\log[\eps d_{ci}(b_i)]}$ and $\eta=|\x-\X_i(s_i)|$ is the perpendicular distance from any given target point $\X_i(s_i)$ on the $i$-th filament.  Upon matching the inner
solution near each filament to the outer solution, we obtain in place
of \eqref{sec2:out_1} that the outer solution satisfies
\bsub \label{sec5:out_1}
\begin{align}
  \Delta_{\x} u &= -\frac{1}{D} + 2\pi \sum_{i=1}^{N} \int_{\Gam_i}
  A_{i}(s_{i}^{\prime})
\delta(\x-\X_{i}(s_{i}^{\prime})) \, \dd s_{i}^{\prime} \,,
\quad \x\in \Omega\setminus\cup_{i=1}^{N}\Gam_i \,;
\quad \partial_n u=0 \,, \quad \x\in
    \partial\Omega \,, \label{sec5:out_1a} \\
    u &- A_{i}(s_i) \log|\x-\X_{i}(s_i)| \to
    \frac{A_{i}(s_i)}{\nu_i} \quad \mbox{as}
    \quad \x\to \X_{i}(s_i) \,, \quad s_i\in [0,L_i] \,,
    \quad i\in \lbrace{1,\ldots,N\rbrace} \,. \label{sec5:out_1b}
\end{align}
\esub
Our goal is to derive a coupled set of integral equations for the line
densities $A_i(s_i)$ from enforcing the $N$ constraints obtained from
the regular parts of the singularity structures in
\eqref{sec5:out_1b}. Upon applying the divergence theorem to
\eqref{sec5:out_1a} we readily derive that
\begin{equation}\label{sec5:cons}
 \sum_{i=1}^{N} \int_{\Gam_i} A_{i}(s_i) \dd s_i = \frac{|\Om|}{2\pi D} \,.
\end{equation}
In place of \eqref{sec2:urep} the outer solution is now represented in
terms of the Neumann Green's function $\GN$ and an unknown constant
$\chi$ by
\begin{equation}\label{sec5:urep}
  u(\x) = \chi - 2\pi \sum_{i=1}^{N} \int_{\Gam_i} A_{i}(s_{i}^{\prime})
  \GN(\x;\X_{i}(s_{i}^{\prime}))\dd s_{i}^{\prime} \,,
  \quad \mbox{for} \quad \x \notin \cup_{i=1}^{N} \Gam_i \,,
\end{equation}
where the global MFRT is related to $\chi$ by
\begin{equation}\label{sec:chi_ave}
   \bar{u} = \chi + o(1) \,.
\end{equation}

For each fixed $j\in \lbrace{1,\ldots,N\rbrace}$, we decompose $\GN$ into
its singular part and regular part as
\begin{equation}\label{sec5:decomp}
  G\big(\x;\X_{j}(s_{j}^{\prime})\big) =
  \frac{1}{4\pi|\x-\X_{j}(s_j^{\prime})|} + R\big(\x;\X_{j}(s_j^{\prime})\big)
  \,,
\end{equation} 
in order to reduce \eqref{sec5:urep} to
\begin{equation}\label{sec5:uint_1}
  \begin{split}
  u(\x) &= \chi - \frac{1}{2}\int_{\Gam_j}  \frac{A_j(s_{j}^{\prime})}
  {|\x-\X_j(s_{j}^{\prime})|} \dd s_{j}^{\prime} -
  2\pi \int_{\Gam_j} A_j(s_{j}^{\prime}) R(\x; \X_j(s_{j}^{\prime})) \dd
  s_{j}^{\prime} \\
  & \qquad -2 \pi \stacksum \int_{\Gam_i} A_{i}(s_{i}^{\prime}) 
  G\big(\x;\X_{i}(s_{i}^{\prime})\big) \dd s_{i}^{\prime}  
  \,, \quad \mbox{for} \quad \x \notin \cup_{i=1}^{N} \Gam_i \,.
  \end{split}
\end{equation}

We now fix a target point $\X_{j}(s_j)$ on $\Gam_j$ and we let
$\x=\X_{\eta,j}(s_j)$ denote any point at a distance $\eta$ from
$\X_j(s_j)$ in the normal plane to $\Gamma_j$ at $\X_j(s_j)$.
As similar to the analysis in
\eqref{sec2:unit_2}-\eqref{sec2:unit_3} for a single filament, we
obtain in the limit $\eta\to 0$ that in place of \eqref{sec2:unit_3}
we now have
\begin{equation}\label{sec5:unit_3}
  \begin{split}
  u &\sim A_j(s_j) \log\eta + \chi - \frac{1}{2} A_{j}(s_j) {\mathcal
    E}_{j}(s_j) + \frac{1}{2} \int_{\Gam_j} \frac{
    A_{j}(s_j)-A_{j}(s_{j}^{\prime})}
  {|\X_{j}(s_j)-\X_{j}(s_{j}^{\prime})|} \dd s_{j}^{\prime} -2\pi
  \int_{\Gam_j} A_{j}(s_{j}^{\prime}) R(\X_{j}(s_j);
  \X_{j}(s_j^{\prime})) \dd s_{j}^{\prime} \\
   & \qquad  -2 \pi \stacksum \int_{\Gam_i} A_{i}(s_{i}^{\prime}) 
  G\big(\X_{j}(s_j);\X_{i}(s_{i}^{\prime})\big) \dd s_{i}^{\prime}  
  \,, \quad \mbox{as} \quad \X_{\eta,j}(s_j)\to \X_{j}(s_j)\,.
  \end{split}
\end{equation}  
Here ${\mathcal E}_{j}(s_j)$, is {\em intrinsic self-energy
  of $\Gam_j$} defined by
\begin{equation}\label{sec5:intrinsic}
  {\mathcal E}_{j}(s_j) \equiv \lim_{\eta\to 0^{+}} \left( \int_{\Gam_j}
  \frac{\dd s_{j}^{\prime}}{
    |\X_{\eta,j}(s_j)-\X_{j}(s_{j}^{\prime})|}  +
  2\log\eta \right) \,.
\end{equation}

Finally, to obtain our coupled set of integral equations for the line
densities $A_{j}(s_j)$ for $j\in \lbrace{1,\ldots,N\rbrace}$ and
$\chi$, we enforce from \eqref{sec5:out_1b} that the limiting behavior
in \eqref{sec5:unit_3} must satisfy $u\sim A_{j}(s_j)\log\eta +
{A(s_j)/\nu_j}$ as $|\X_{\eta,j}(s_j)-\X_{j}(s_j)|=\eta\to 0$.  We
summarize our result in the following proposition:

\begin{proposition}\label{mprop:int_eq} 
Let $\cup_{j=1}^{N}\Gam_j$ be a collection of smooth, simple,
non-intersecting closed curves in $\Omega$ parameterized by arclength
as $\X_{j}(s_j)$ for $j\in \lbrace{1,\ldots,N\rbrace}$ with
$\mbox{dist}(\Gam_i,\Gam_j)={\mathcal O}(1)$ for $i\neq j$. Assume
that each filament centered around $\Gam_j$ has an ${\mathcal
O}(\eps)$ cross-sectional diameter and a uniform shape $\eps
{\mathcal B}_j$.  In the limit $\eps\to 0$, the asymptotic solution
to \eqref{intro:mfrt} in the outer region at ${\mathcal O}(1)$
distances from $\cup_{j=1}^{N}
\Gam_j$ is
\begin{equation}\label{mprop:uout}
  u(\x) = \chi - 2\pi \sum_{j=1}^{N} \int_{\Gam_j}
  A_{j}(s_{j}) \GN(\x;\X_{j}(s_{j}))\dd
  s_{j} \,, \quad \mbox{for} \quad \x \notin \cup_{j=1}^{N}
  \Gam_j \,,
\end{equation}
where the line densities $A_{j}(s_j)$ for $j\in \lbrace{1,\ldots,N
\rbrace}$ and the asymptotic approximation $\chi$ to the
GMFRT $\bar{u}$ are solutions to the coupled 1-D Fredholm integral
equations of the second kind, defined for all $s_j\in [0,L_j]$ by
\begin{equation}\label{mprop:inteq}
\mathcal{K}_j A_j + \sum_{\substack{i=1\\i\neq j}}^N \mathcal{G}_{ji} A_i = \chi, \quad j = 1,\ldots,N; \qquad \sum_{j=1}^{N}\int_{\Gam_j} A_{j}(s_j) \dd
s_j = \frac{|\Om|}{2\pi D},
\end{equation}
where the integral operators $\mathcal{K}_j$ and $\mathcal{G}_{ji}$ are defined in \eqref{intro:ops}.  In \eqref{mprop:uout} $\GN$ is the Neumann Green's function with
regular part $R$ satisfying \eqref{intro:Gdef}.
\end{proposition}

\subsection{Splitting probabilities}\label{sec5:split}

In this subsection we describe an extension of our analysis to
determine the splitting probabilities for the case of $N$ filaments
$\Gam_{\eps,j}$ for $j\in \lbrace{1,\ldots,N\rbrace}$. We again assume
that the centerlines $\Gamma_j$ for $j\in \lbrace{1,\ldots,N\rbrace}$
are a collection of non-intersecting, smooth and simple closed
curves. Let $\phi_k(\x)$ be the probability that a searcher starting
at $\x$ is absorbed by filament $k$ before any other. Then,
$\phi_k(\x)$ satisfies 
\bsub\label{eq:split-pde}
\begin{gather}\label{eq:split-pde_a}
  \Delta_{\x}\phi_k=0\,, \quad \x\in\Omega\setminus \cup_{j=1}^{N} \Gamma_{\eps,j}\,;
  \qquad \partial_n\phi_k=0\,, \quad \x\in\partial\Omega\,,\\[5pt]
\label{eq:split-pde_b} \eps \partial_n\phi_k+b_j\,(\phi_k-\delta_{jk})=0\,, \qquad
  \x\in\partial \Gamma_{\eps,j}\,, \quad j\in \lbrace{1,\dots,N\rbrace} \,,
\end{gather}
\esub 
with $\phi_k=\delta_{jk}$ on $\partial \Gamma_{\eps,j}$ when
$b_j=\infty$ and where $\delta_{jk}$ is the Kronecker symbol. The analysis outlined in the previous section can be applied in an
identical fashion to solve \eqref{eq:split-pde}. The solvability
condition is altered by the homogeneous nature of the right-hand side
of \eqref{eq:split-pde_a} and an extra constant term is introduced by
the boundary condition \eqref{eq:split-pde_b}. The result is
summarized as follows:
\begin{proposition}\label{prop:splitting}
Let $\cup_{j=1}^{N}\Gam_j$ be a collection of smooth, simple,
non-intersecting closed curves in $\Omega$ parameterized by arclength
as $\X_{j}(s_j)$ for $j\in \lbrace{1,\ldots,N\rbrace}$ with
$\mbox{dist}(\Gam_i,\Gam_j)={\mathcal O}(1)$ for $i\neq j$. Assume
that each filament centered around $\Gam_j$ has an ${\mathcal
O}(\eps)$ cross-sectional diameter and a uniform shape $\eps
{\mathcal B}_j$.  In the limit $\eps\to 0$, the asymptotic solution
to \eqref{eq:split-pde} in the outer region at ${\mathcal O}(1)$
distances from $\cup_{j=1}^{N}
\Gam_j$ takes the form
\begin{equation}\label{eq:split-outer}
  \phi_k(\x)\sim\bar{\phi}_k-2\pi\sum_{j=1}^{N}\int_{\Gam_j}
  B_{kj}(s_{j})\,
  G\big(\x;  \X_j(s_{j})\big)\dd s_{j} \,, \quad \x\in
  \Omega\setminus\cup_{j=1}^{N}\Gam_j \,,
\end{equation}
where $B_{kj}(s_j)$ for $j\in{1,\ldots,N}$ are the line densities of on
$\Gamma_j$ for the $k$-th splitting probability. The quantity
$\bar{\phi}_k$ is the leading-order volume average of $\phi_k$ over
all starting locations. The line densities satisfy a system of
coupled 1-D Fredholm integral equations with a constraint, defined for all $s_j\in
[0,L_j]$ by
\begin{equation}\label{eq:split-IE}
\mathcal{K}_j B_{kj} + \sum_{\substack{i=1\\i\neq j}}^N \mathcal{G}_{ji} B_{ki} = \bar{\phi}_k -\delta_{jk}, \quad j = 1,\ldots,N; \qquad \sum_{j=1}^{N}\int_{\Gam_j} B_{kj}(s_j) \dd
s_j = 0,
\end{equation}
where the integral operators $\mathcal{K}_j$ and $\mathcal{G}_{ji}$ are defined in \eqref{intro:ops}.
\end{proposition}

\section{Closed-form solution: A concentric circular filament in a sphere}\label{sec2:conc_ring}

To illustrate our analysis, we first consider a simple example where
$\Om$ is the unit sphere, with $|\Om|={4\pi/3}$, and where $\Gam$ is a
simple closed circular ring of radius $r_0<1$ in the equatorial plane
centered at the origin, so that $L=|\Gam|=2\pi r_0$ (see
Fig.~\ref{fig:sphere_tube} for a schematic). In this case,
it is shown in Appendix \ref{appa:der_loop} that its intrinsic
self-energy has the constant value ${\mathcal E}=2\log(8r_0)$.

\begin{figure}[h!tbp]
  \centering
      \begin{subfigure}[b]{0.48\textwidth}  
        \includegraphics[width=\textwidth, trim={90bp 90bp 90bp
            90bp}, clip]{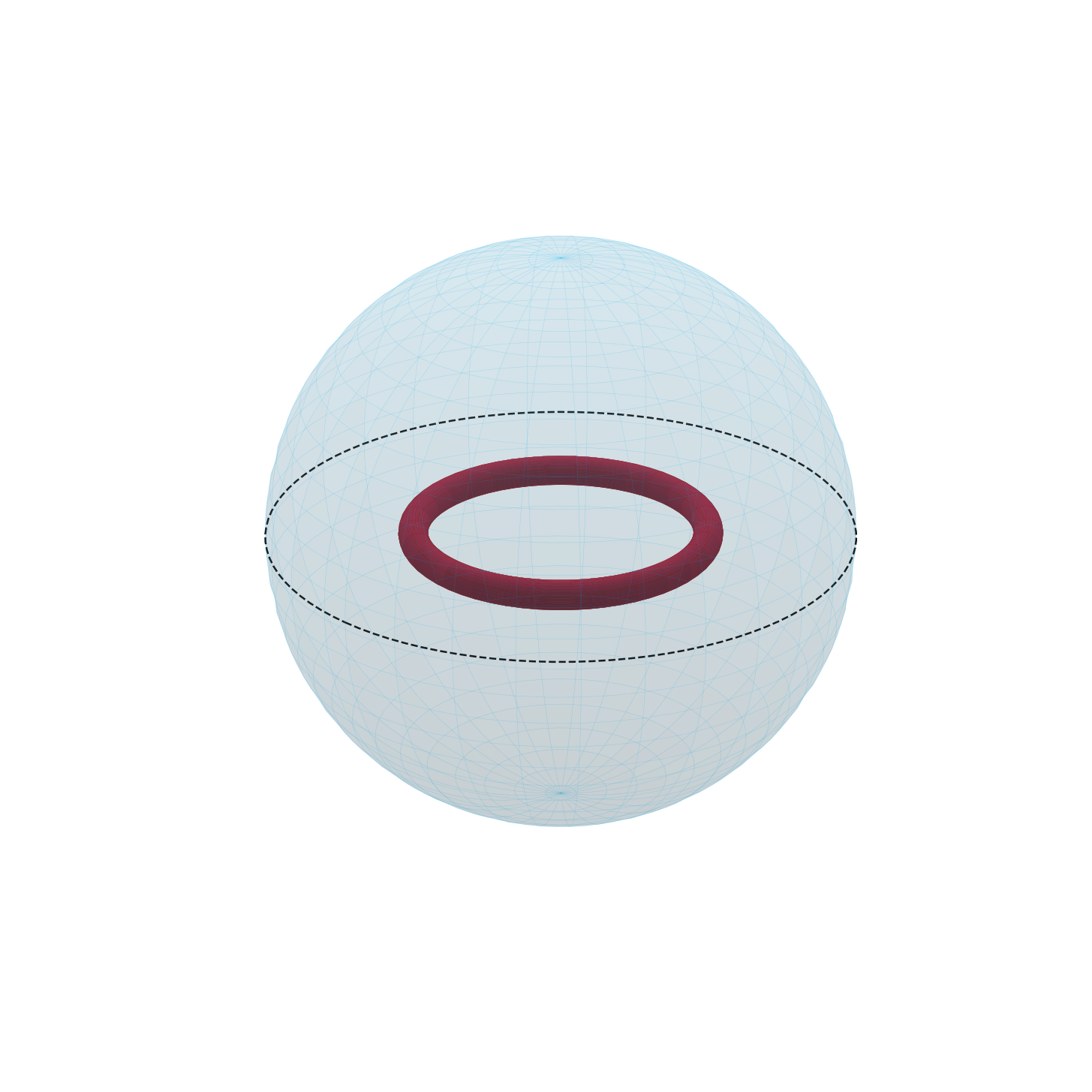}
        \caption{Unit sphere with filament}
        \label{fig:sphere_tube}
    \end{subfigure}
      \begin{subfigure}[b]{0.435\textwidth}
            \includegraphics[width=\textwidth, trim={90bp 90bp 90bp
            90bp}, clip]{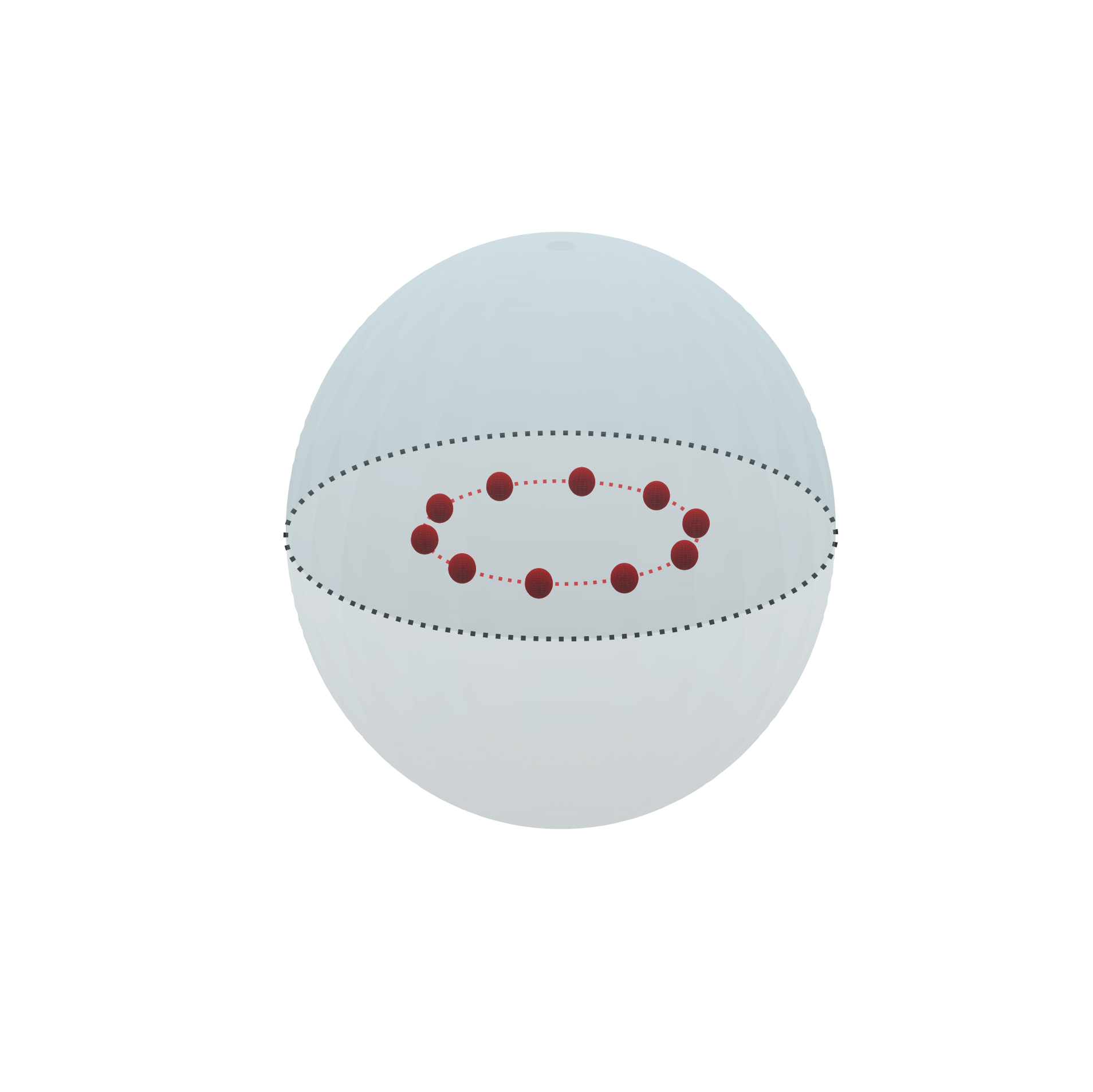}
      \caption{Unit sphere with balls}
      \label{fig:sphere_balls}
    \end{subfigure}
    \caption{Left: The unit sphere with a thin filament (torus) of
      circular cross-section with radius $\varepsilon=0.05$ centered
      concentrically on the equatorial plane. The distance from the
      centerline of the filament to the origin is $r_0=0.5$. Right: $N=10$
      small spheres of a common radius $\sigma=0.05$
      equidistantly spaced on a concentric ring of radius $r_0=0.5$ within
      the unit sphere.}
    \label{fig:sphere_tube_balls}
\end{figure}

For the unit sphere, the regular part $R(\x;\xi)$ of the Neumann
Green's function $G(\x;\bxi)$ satisfying \eqref{intro:Gdef}, as
derived in Appendix A of \cite{ChevWard2010}, is
\begin{equation}\label{gr:sphere}
  R(\x;\bxi) = \frac{1}{4\pi |{\x/|\x|}-|\x| \bxi| }
  + \frac{1}{4\pi} \log\left(\frac{2}{1 - \x {\bf \cdot} \bxi +
    |{\x/|\x|}- |\x|\bxi| } \right)  + \frac{1}{8\pi} \left(
 |\x|^2 + |\bxi|^2 \right)  -\frac{7}{10\pi} \,.
\end{equation}

Since the ring lies concentrically on the equatorial plane, it follows
by symmetry that the line density $A(s)$ is independent of position
along the ring. As a result, since $A(s)=A$ for all $s\in [0,L]$, we
obtain from the constraint \eqref{sec2:cons} that
\begin{equation}\label{sec3:aring}
  A = \frac{|\Om|}{2\pi D |\Gam|} = \frac{1}{3\pi D r_0} \,.
\end{equation}
As a result, we obtain from \eqref{prop:inteq} upon using
${\mathcal E}=2\log(8 r_0)$ that
\begin{equation}\label{sec3:chi_ring}
  \bar{u}\sim  \chi = \frac{1}{3\pi D}
  \left[ \frac{1}{r_0} \left(-\log(\eps d_c)\right)
    + \frac{\log(8r_0)}{r_0} + {\mathcal J} \right] \,,
  \quad \mbox{where} \quad {\mathcal J} \equiv \frac{2\pi}{r_0} \int_{\Gam}
  R\left( \X(s);\X(s^{\prime})\right) \dd s^{\prime}\,,
\end{equation}
where $\X(s)=r_0\left(\cos({s/r_0}),\sin({s/r_0}),0\right)^{T}$ is the
parameterization of the circular ring in terms of arclength $s$. By
symmetry it follows that the integral ${\mathcal J}$ is independent of $s$.

To evaluate ${\mathcal J}$ it is convenient to introduce $\phi$ and
$\phi^{\prime}$ by $s=r_0\phi$ and $s^{\prime}=r_0\phi^{\prime}$, so that
\begin{equation}\label{sec3:I}
   {\mathcal J} = 2\pi \int_{0}^{2\pi}
  R\left( \X(r_0\phi);\X(r_0\phi^{\prime})\right) \dd \phi^{\prime} \,.
\end{equation}

To evaluate $R\left( \X(r_0\phi);\X(r_0\phi^{\prime})\right)$
from
\eqref{gr:sphere} we first calculate that
\begin{equation*}
  \X(r_0\phi) {\bf \cdot} \X(r_0\phi^{\prime}) = r_0^{2}
  \cos\left(\phi-\phi^{\prime}\right) \,, \qquad
  \left|\frac{\X(r_0\phi)}{|\X(r_0\phi)|} - \X(r_0\phi^{\prime}) \,|\X(r_0\phi)| \right|
  =  \left( r_0^{4} + 1 - 2 r_0^2 \cos\left(\phi-\phi^{\prime}\right)\right)^{1/2}\,,
\end{equation*}
so that from \eqref{gr:sphere} we get
\begin{equation}\label{sec3:r1}
  \begin{split}
  2\pi  R\left( \X(r_0\phi);\X(r_0\phi^{\prime})\right) &=
  \frac{1}{2 \left(r_{0}^4 + 1 - 2 r_0^2 \cos\left(\phi-\phi^{\prime}\right)
    \right)^{1/2}} + \frac{1}{2}\left(r_0^2 - \frac{14}{5} + \log{2}\right) \\
     & \qquad -\frac{1}{2}
  \log\left[ 1 - r_0^2 \cos\left(\phi-\phi^{\prime}\right) +
    \left(r_{0}^4 + 1 - 2 r_0^2 \cos\left(\phi-\phi^{\prime}\right)
    \right)^{1/2}\right] \,.
  \end{split}
\end{equation}
Upon substituting \eqref{sec3:r1} into \eqref{sec3:I} we get
\begin{equation}\label{sec3:r2}
  {\mathcal J} = \frac{1}{2} I(r_0) + \pi \left(r_0^2 -\frac{14}{5} +
  \log{2}\right) - \frac{1}{2} J(r_0)\,,
\end{equation}
where we have defined
\begin{equation}\label{sec3:I0J0}
  I(r_0) \equiv \int_{0}^{2\pi} \frac{\dd\psi}{\left(r_0^4 + 1 - 2r_0^2\cos\psi
    \right)^{1/2}}  \,, \quad
  J(r_0) \equiv \int_{0}^{2\pi}  \log\left[ 1 - r_0^2 \cos\psi +
    \left(r_{0}^4 + 1 - 2 r_0^2 \cos\psi\right)^{1/2}\right] \dd\psi\,.
\end{equation} 

Next, we evaluate $I(r_0)$ in terms of the complete elliptic
integral of the first kind of modulus $m$: $K(m)\equiv
\int_{0}^{\pi/2} {d\phi/\sqrt{1-m^2\sin^2\phi}}$.  To do so, we let
$\cos\psi=2\cos^2\left({\psi/2}\right)-1$ and then set
$\phi={\psi/2}$. We obtain that
\begin{equation}\label{sec3:I01}
  I(r_0) = 4 \int_{0}^{\pi/2} \frac{\dd\phi}{\left[ (r_0^2+1)^2 -
      4r_0^2 \cos^2\phi \right]^{1/2}} = \frac{4}{1+r_0^2} K\left(
  \frac{2 r_0}{1+r_0^2} \right) = 4 K\left( r_0^2\right) \,,
\end{equation}
where we used the Landen transformation to obtain the last equality in
\eqref{sec3:I01}.

Upon using \eqref{sec3:I01} in \eqref{sec3:r2} and recalling
\eqref{sec3:chi_ring}, we conclude that
\bsub \label{sec3:ubar}
\begin{equation}
  \bar{u}\sim \chi = \frac{1}{3\pi D} {\mathcal F}_{\eps}(r_0) \,,
  \label{sec3:ubar_1}
\end{equation}
where ${\mathcal F}_{\eps}(r_0)$ is defined by
\begin{equation}  \label{sec3:ubar_2}
  {\mathcal F}_{\eps}(r_0) \equiv
  \frac{\log(a_{\eps} r_0)}{r_0} + 2
  K\left(r_0^2\right) + \pi \left(r_0^2 + \log{2} -\frac{14}{5}
  \right) - \frac{1}{2} J(r_0) \,, \quad \mbox{with} \quad a_{\eps}\equiv
  \frac{8}{\eps d_c} \,.
\end{equation}
\esub Here $J(r_0)$ as defined in \eqref{sec3:I0J0}. From an
evaluation of $J(r_0)$ by the trapezoidal rule, in
Fig.~\ref{fig:ring_a}A we plot ${\mathcal F}_{\eps}(r_0)$ on ${\mathcal
  O}(\eps)\ll r_0 \ll 1-{\mathcal O}(\eps)$ for several values of
$\eps d_c$. This numerical evidence suggests that there is an optimum
ring radius $r_{0}=r_{0\min}$ at which ${\mathcal F}_{\eps}(r_0)$ is
minimized. From \eqref{sec3:ubar_1}, this optimum value determines the
optimal location of a concentric equatorial ring that minimizes the
global MFRT $\bar{u}$ in the unit sphere to all logarithmic
orders.

\begin{figure}[h!tbp]
      \centering
      \includegraphics[width = 0.95\textwidth]{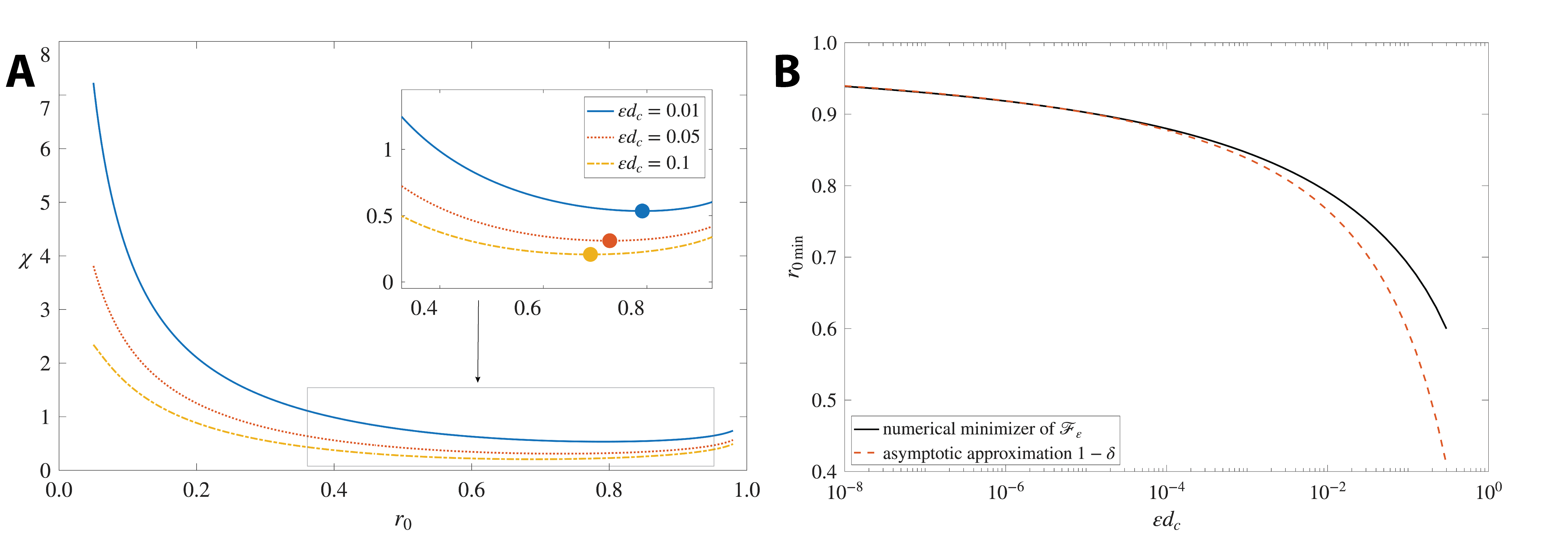}
\caption{
The exactly solvable one-ring case. Panel A: The MFRT
$\bar{u}\sim \chi$, defined in \eqref{sec3:ubar}, versus the ring
radius $r_0$ for three values of $\varepsilon d_c$ when $D=1$. The
minimizing radii $r_{0\min}$ is highlighted for each case by a solid
dot . Panel B: The minimizing ring radius $r_{0\min}$ against
$\varepsilon d_c$ from the numerical minimization of ${\mathcal
F}_{\eps}$ (solid line) and the asymptotic approximation $r_{0\min}
\sim 1 - \delta$ (dashed line) given in Proposition
\ref{prop:opt_ring}.} \label{fig:ring_a}
\end{figure}

In Appendix \ref{appr:optim} we  prove that ${\mathcal F}_{\eps}(r_0)$ 
has a unique minimum $r_{0\min}$ on 
${\mathcal O}(\eps)\ll r_0 \ll 1-{\mathcal O}(\eps)$ and determine 
an approximation to $r_{0\min}$ for $\eps d_c\ll 1$. The result
derived in Appendix \ref{appr:optim} is summarized as follows:

\begin{proposition}\label{prop:opt_ring} In the limit $\eps\to 0$, there
  is a unique centerline radius $r_0=r_{0\min}$ of a concentric
  equatorial filament that minimizes the global MFRT inside the unit
  sphere. The optimal MFRT is
  \begin{equation}\label{prop:opt_ring_mfpt}
    \bar{u}_{\min}\sim \chi_{\min} = \frac{1}{3\pi D} {\mathcal
      F}_{\eps}(r_{0\min}) \quad \mbox{with} \quad {\mathcal
      F}_{\eps}(r_0) \equiv \frac{\log(a_{\eps} r_0)}{r_0} + 2
    K\left(r_0^2\right) + \pi \left(r_0^2 + \log{2} -\frac{14}{5}
    \right) - \frac{1}{2} J(r_0) \,,
  \end{equation}
  where $r_{0min}$ is the unique root on ${\mathcal
    O}\left({1/a_{\eps}}\right)\ll r_0\ll 1-{\mathcal
    O}\left({1/a_{\eps}}\right)$ of ${\mathcal C}_{\eps}(r_0)=0$ given
  in \eqref{sec3:fp}.  Here $J(r_0)$ is defined in \eqref{sec3:I0J0}
  and $a_{\eps}={8/(\eps d_c)}$. For $a_{\eps}\gg 1$, we have the
  explicit approximation
  \begin{equation}\label{prop:opt_r0}
    r_{0\min} \sim 1 - \frac{1}{\log{a_{\eps}}} -
    \frac{\left[\log(4 \log a_\eps) +{1/2}\right]}{(\log{a_{\eps}})^2}
    \,.
  \end{equation}
\end{proposition}

Although the optimal ring radius tends to the boundary of the unit
sphere as $\eps$ decreases, the requirement that $r_{0\min}\ll 1
-{\mathcal O}\left({1/a_{\eps}}\right)$ is never violated.  In
addition, \eqref{prop:opt_r0} shows that the optimal ring radius
depends on the local geometry and reactivity of the filament
cross-section through its logarithmic reactive capacitance
$d_{c}=d_{c}(b)$. Finally, we remark that as $r_0$ increases the
surface area of the filament $4\pi^2\eps r_0$ increases, but is still
${\mathcal O}(\eps)$ for any $r_0={\mathcal O}(1)$.

In Table \ref{tab:mfpt_optim} we show that the approximate root
\eqref{prop:opt_r0} provides an increasingly adequate approximation to
the numerically computed root of ${\mathcal C}_{\eps}(r_0)=0$ as $\eps
d_c$ decreases. The prediction for the minimum of the MFRT as obtained
from \eqref{prop:opt_r0} and the numerically computed root are also
favorably compared in Table \ref{tab:mfpt_optim}. See also the agreement shown in Fig.~\ref{fig:ring_a}B.
             
\begin{table}[htbp]
    \centering
        \begin{tabular}{rcccccc}
        \hline
  $\eps d_c$ & $r_{0\min}$ (num) & $r_{0\min}$ (asy) & $\%$ error for $r_{0\min}$ & $\chi_{\min}$ (num) & $\chi_{\min}$ (asy) \\ \hline
  0.05  & 0.7290  & 0.6667  & 8.55$\%$  &  0.3109  & 0.3187  \\ 
  0.02  & 0.7680  & 0.7307  & 4.86$\%$  &  0.4406  & 0.4439  \\ 
  0.01  & 0.7914  & 0.7657  & 3.25$\%$  &  0.5349  & 0.5367  \\ 
 0.005  & 0.8110  & 0.7931  & 2.21$\%$  &  0.6267 &  0.6277   \\
 0.0025 & 0.8276  & 0.8151  & 1.52$\%$  &  0.7164 &  0.7170  \\ \hline
    \end{tabular}
    \caption{Comparison of the asymptotic result \eqref{prop:opt_r0} for the
      optimal ring radius with the numerically computed root of
      ${\mathcal C}_{\eps}(r_0)=0$, given in \eqref{sec3:fp}, for several
      values of $\eps d_c$. The last two columns show a favorable comparison
      for the optimal MFRT given in \eqref{prop:opt_ring_mfpt} with either
      the approximate root or the numerically computed root.}
    \label{tab:mfpt_optim}
\end{table}

\subsection{Optimization of a ring pattern of small spheres}\label{sec3:ball}

We now compare results from our previous analysis for the MFRT for a
thin equatorial filament in the unit sphere to a related problem where
there are $N$ small identical perfectly absorbing targets
equidistantly spaced on a concentric ring of radius $r_0$ within the
unit sphere. A schematic of the geometry is shown in
Fig.~\ref{fig:sphere_balls} when the targets are small spheres.  In a
3-D domain punctured by $N$ such small targets, the MFRT $u(\x)$
satisfies 
\bsub
\begin{gather}
  \Delta u = -\frac{1}{D} \,,  \quad \x \in \Omega\setminus \Omega_a 
   \,, \label{ball:ssp_1}\\
  \partial_n u  = 0 \,,  \quad \x\in \partial \Omega \,; \qquad
        u = 0\,, \quad \x \in \partial \Omega_a \equiv \cup_{j=1}^{N} 
 \partial \Omega_{\sigma,j} \,, \label{ball:ssp_2}
\end{gather}
\esub 
where $\Omega_{\sigma,j}$ is a small 3-D target of diameter ${\mathcal
O}(\sigma)$ centered at $\x_j\in \Omega$, and $\Omega_a =
\cup_{j=1}^{N} \Omega_{\sigma,j}$.  We assume that the
targets are well-separated in the sense that $\dist(\x_j,\x_k)\gg
{\mathcal O}(\sigma)\,$ $\forall j\ne k$ and
$\dist(\x_j,\partial\Omega)\gg {\mathcal O}(\sigma)$.  The volume
averaged MFRT, denoted by $\bar{u}$, is defined by
\begin{equation} 
  \bar{u} \equiv \frac{1}{|\Omega\setminus \Omega_a|} \int_{\Omega\setminus
  \Omega_a} u(\x) \,
\dd \x \,. \label{ball:barv}
\end{equation}

For the small target limit $\sigma\to 0$, by using the method of matched
asymptotic expansions it was derived in Principal Result 3.1 of
\cite{ChevWard2010} that
\begin{equation}\label{ball:mfpt}
  \bar{u} \sim \chi_{\text{ball}} =
  \frac{|\Omega|}{4\pi D} \left[ \frac{1}{\sigma N C} +
    p(\x_1,\ldots,\x_N) + o(1) \right] \,; \quad
  p(\x_1,\ldots,\x_N) \equiv \frac{4\pi}{N^2} \sum_{j=1}^{N} \left(
  R(\x_j;\x_j) + \stacksum G(\x_i;\x_j) \right)\,,
\end{equation}
where $G(\x;\bxi)$ is the Neumann Green's function with regular part
$R(\bxi;\bxi)$ (see \eqref{intro:Gdef}). Here $C$ is the
capacitance of the rescaled target $\Omega_0 =
\Omega_{\sigma,j}/\sigma$, defined in terms of the far-field behavior
of the canonical local problem
\begin{equation}\label{ball:capac}
   \Delta_\y w_c =0 \,, \quad \y \in {\mathbb R}^{3}\setminus\Omega_0 \,; \qquad
   w_c = 1 \,, \quad \y \in \partial\Omega_0 \,; \qquad w_c =
   \frac{C}{|\y|} + {\mathcal O}(|\y|^{-2}) \quad \mbox{as} \quad
   |\y|\to \infty \,,
\end{equation}
 When $\Omega_0$ is the unit ball we have $C=1$, while $C={2/\pi}$ when
$\Omega_0$ is a flat disk of unit radius.

Now suppose that the $N$ identical traps are equidistantly spaced on a
concentric ring of radius $r_0$ within the unit sphere and are
centered at $\x_j=r_0\left( \cos\theta_j,\sin\theta_j,0\right)^{T}$
with $\theta_j={2\pi (j-1)/N}$ for $j\in \lbrace{1,\ldots,N\rbrace}$.
Since this pattern has a cyclic symmetry in the sphere,
\eqref{ball:mfpt} reduces to
\begin{equation}\label{ball:pnew}
  p(\x_1,\ldots,\x_N) = \frac{4\pi}{N} \left( R(\x_1;\x_1) +
  \sum_{j=2}^{N} G(\x_1;\x_j)\right) \,.
\end{equation}
From \eqref{gr:sphere} we calculate that
\begin{equation}\label{ball:R}
  4\pi R(\x_1,\x_1) = \frac{1}{(1-r_0^2)} -\log(1-r_0^2) + r_0^2-\frac{14}{5}
  \,,
\end{equation}
while for $j\in \lbrace{2,\ldots,N\rbrace}$, we have
\begin{equation}\label{ball:G}
  4\pi G(\x_1;\x_j) = \frac{1}{|\x_1-\x_j|} + \frac{1}{r_{c,j}}
  + \log\left( \frac{2}{1-\x_1 {\bf \cdot} \x_j + r_{c,j}}\right) +
  \frac{1}{2}\left(|\x_1|^2 + |\x_j|^2\right) - \frac{14}{5} \quad
  \mbox{where} \quad r_{c,j} \equiv \big{\vert} \frac{\x_1}{|\x_1|} -
  \x_j |\x_1| \big{\vert} \,.
\end{equation}
Since $\theta_1=0$, we calculate that
\begin{equation}\label{ball:G_1}
  \x_1 {\bf \cdot} \x_j=r_{0}^2\cos(\theta_j) \,, \qquad
  |\x_1-\x_j|=2 r_0\sin\left( \frac{\theta_j}{2}\right) \,, \qquad
  r_{c,j} =\sqrt{r_0^4 - 2 r_0^2 \cos\theta_j + 1} \,.
\end{equation}
In order that the small traps are well-separated on the ring, we require
that $2\sigma N\ll 2\pi r_0$, so that $\sigma N\ll \pi r_0$.

Upon substituting \eqref{ball:G_1} into \eqref{ball:G}, and absorbing
the regular part \eqref{ball:R} into the $j=1$ term, \eqref{ball:pnew}
becomes solely a function of $N$ and the ring radius $r_0$. We
summarize our result as follows:

\begin{proposition}\label{ball:prop} 
For $N$ identical small targets equidistantly spaced on a concentric
ring of radius $r_0$ inside the unit sphere, the volume-averaged MFRT
is given for $\sigma\to 0$ by
\bsub \label{ball:finalN}
\begin{equation}\label{ball:finalN1}
     \bar{u} \sim \chi_{\text{ball}} =
  \frac{|\Omega|}{4\pi D} \left[ \frac{1}{\sigma N C} +
    P_{N}(r_0) + o(1) \right] \,,
\end{equation}
where $|\Omega|={4\pi/3}$ and
\begin{equation}\label{ball:pfinal1}
   P_{N}(r_0) \equiv  r_0^2 -\frac{14}{5} + \log{2}
  + \frac{{\mathcal S}_{N}}{2r_0 N} + \frac{1}{N} \left( \sum_{j=1}^{N}
  \frac{1}{r_{c,j}} - \log\left(1-r_0^2 \cos\theta_j + r_{c,j} \right)
  \right)\,.
\end{equation} 
Here $r_{c,j}$ is given in \eqref{ball:G_1}, $\theta_j=2\pi{(j-1)/N}$
for $j\in \lbrace{1,\ldots,N\rbrace}$, and ${\mathcal S}_{N}$ is the
discrete sum
\begin{equation}\label{ball:sn}
  {\mathcal S}_{N} \equiv \sum_{j=1}^{N-1} \frac{1}{\sin\left({\pi j/N}
    \right)} \,.
\end{equation}
\esub 
\end{proposition}

Since there are no analytical formulas for the discrete sums in
\eqref{ball:pfinal1}, the optimal ring radius $r_{opt}$ for a given
$N$ is obtained by a numerical line search to determine the minimum of
$P_{N}(r_0)$ on $0<r_0<1$, with the results shown in Table
\ref{tab:ropt_ball}.  The numerical results in the second column of
Table \ref{tab:ropt_ball} suggest that $r_{opt}$ tends very slowly to
the boundary as $N$ increases. In order to estimate the optimal ring
radius for $N\gg 1$, we need the following lemma, as proved in
Appendix \ref{appa:sn}, for the behavior of $S_N$ for large $N$.

\begin{lemma}\label{lemma:sn} For $N\gg 1$, we have
  \begin{equation}\label{ball:sn_asy}
  {\mathcal S}_{N} = {\mathcal S}_{NA} -\frac{\pi}{36 N} + {\mathcal O}(N^{-3}) \,, 
   \qquad \mbox{where} \quad
   {\mathcal S}_{NA} \equiv   \frac{2N}{\pi} \left[ \log\left(\frac{2N}{\pi}
      \right) + \gamma_e \right]\,.
  \end{equation}
  Here $\gamma_e\approx 0.5772$ is Euler's constant.
\end{lemma}

In Table \ref{tab:sn} we show that the simple approximation
${\mathcal S}_{NA}$ in \eqref{ball:sn_asy} compares very favorably
with the discrete sum ${\mathcal S}_N$ even for $N=4$.

\begin{table}[ht!]
\centering
\begin{tabular}{rccc}
   \hline
  $N$ & ${\mathcal S}_N$ &  ${\mathcal S}_{NA}$ & Percent error (\%) \\ \hline
4 & 3.828 & 3.850 & 0.566\% \\
6 & 7.309 & 7.324 & 0.198\% \\
8 & 11.219 & 11.230 & 0.097\% \\
10 & 15.450 &15.459 & 0.056\% \\
20 & 39.738 & 39.742 & 0.011\% \\
40 & 97.134 & 97.136 & 0.002\% \\
\hline 
\end{tabular}
\caption{Comparison between exact sum ${\mathcal S}_N$ and its
  large-$N$ asymptotics ${\mathcal S}_{NA}$ in \eqref{ball:sn_asy}. We
  observe that even for $N=4$, the error in the approximation is less
  than $1\%$.}
\label{tab:sn}
\end{table}

To analyze the large $N$ limit of \eqref{ball:finalN}, but under the
well-separated assumption that $\sigma N\ll \pi r_0$, we simply
approximate the discrete sums in \eqref{ball:pfinal1} by their
corresponding Riemann sums and we use ${\mathcal S}_{NA}$ for
${\mathcal S}_N$. The Riemann sums are identified by recalling the two
integrals $I(r_0)$ and $J_{0}(r_0)$ in \eqref{sec3:I0J0} and
\eqref{sec3:I01}. In this way, for large $N$ we conclude that
\bsub \label{ball:ucont}
\begin{equation}\label{ball:ularge}
  \bar{u}\sim \chi_{\text{ball}} = \frac{1}{3\pi D}
  \left( \frac{\pi}{\sigma N C} + {\mathcal F}_{N}(r_0) \right) \,,
\end{equation}
where ${\mathcal F}_{N}(r_0)$ is defined by
\begin{equation}\label{ball:fn}
 {\mathcal F}_{N}(r_0) \equiv \frac{\log(a_{N})}{r_0} + 2
    K\left(r_0^2\right) + \pi \left(r_0^2 + \log{2} -\frac{14}{5}
    \right) - \frac{1}{2} J(r_0) \,, \qquad a_{N}\equiv
    \frac{2N e^{\gamma_e}}{\pi} \,.
\end{equation} 
\esub

\begin{table}[ht]
\centering
\begin{tabular}{ccc}
\hline
 $N$  & $r_{\text{opt}}$ \eqref{ball:pfinal1}  & $r_{\text{opt}}$ \eqref{ball:fn} \\
\hline
2    & 0.4289     & 0.4798 \\
4    & 0.5613     & 0.5740 \\
6    & 0.6095     & 0.6133 \\
8    & 0.6357     & 0.6369 \\
10   & 0.6529     & 0.6537 \\
20   & 0.6967     & 0.6969 \\
40   & 0.7315     & 0.7320 \\
60   & 0.7488     & 0.7492  \\
80   & 0.7600     & 0.7603 \\
100  & 0.7681     & 0.7685 \\
\hline
\end{tabular}
\caption{Optimal ring radius for $N$ small targets
  equidistantly-spaced on a concentric ring obtained from minimizing
  either the discrete sum $P_{N}(r_0)$ in \eqref{ball:pfinal1} or its
  large-$N$ approximation ${\mathcal F}_{N}(r_0)$ in
  \eqref{ball:fn}. Remarkably, we observe that the large-$N$ approximation
  \eqref{ball:fn} agrees with the discrete optimum in \eqref{ball:pfinal1}
  to within $5 \times 10^{-4}$ for $N \ge 10$ (and to within $0.013$
  already when $N = 4$).}\label{tab:ropt_ball}
\end{table}

We remark that ${\mathcal F}_{N}(r_0)$ is identical to
\eqref{prop:opt_ring_mfpt} for the filament problem
if $a_N$ is replaced by $a_{\eps}r_0$.
As such, we can repeat an almost identical analysis as done for an 
equatorial ring filament to establish that there is an optimal 
ring radius for $N\gg 1$ equidistantly spaced small targets that is 
close to the boundary of the sphere. In the third column of Table
\ref{tab:ropt_ball} we show that the minimizer of ${\mathcal F}_{N}(r_0)$
provides a close approximation to that obtained for the
discrete sum $P_{N}(r_0)$ even when $N$ is only moderately large.

For $N\gg 1$, we can derive an analytical scaling law for the optimal
radius by repeating a similar analysis as for the thin filament. We
find for large $N$ that
\begin{equation}\label{ball:optim}
  r_{opt} = 1 - \frac{1}{\log{a_{N}}} -
    \frac{\left[\log(\log a_N) + \log{4}-{1/2}\right]}{(\log{a_{N}})^2}
    \,.
\end{equation}
In particular if we set $N={\mathcal O}(\sigma^{-1})$, we observe that
the optimal ring radius is ${\mathcal O}({-1/\log\sigma})$ close to the
boundary; a result that is qualitatively the same as that derived
in \eqref{prop:opt_r0} for the thin filament. However, since $\log(a_{N})$
is not very large unless $N$ is very large, \eqref{ball:optim} does not
compare very favorably at moderate values of $N$ with the results shown
in Table \ref{tab:ropt_ball}.

Finally, we compare our MFRT result for $N\gg 1$ small perfectly
absorbing targets, each of diameter ${\mathcal O}(\sigma)$, that are
uniformly spaced on a ring with that obtained earlier for an
equatorial filament of circular cross-section of radius $\eps$
with a perfectly absorbing boundary.  For this filament, we set
$d_c=1$, $L=2\pi r_0$, and $(3\pi D r_0)^{-1}= {|\Omega|/(2\pi D L)}$
in
\eqref{sec3:ubar} to obtain
\begin{equation}\label{sec3:fin_film}
\chi_{\text{filam}} = \frac{|\Om|}{2\pi D L} \left( \underbrace{- \log\eps + \log\left(
\frac{4L}{\pi}\right)} + u_{b} \right), \quad
u_{b}\equiv 2r_0 K(r_0^2) + \pi r_0 (r_0^2+\log{2} - {14/5} ) -\frac{r_0}{2}
J(r_0) \,.
\end{equation}
The corresponding result for $N\gg 1$ small targets on the ring,
as obtained from \eqref{ball:ucont}, can be written as
\begin{equation}\label{sec3:fin_ball}
  \chi_{\text{ball}} = \frac{|\Om|}{2\pi D L} \left(
  \underbrace{\frac{L}{2\sigma N C} +
  \log\left(\frac{2N}{\pi}\right) + \gamma_e} + u_{b} \right) .
\end{equation}
Upon equating the underlined terms in \eqref{sec3:fin_film} and
\eqref{sec3:fin_ball} we observe that $N\gg 1$ identical small targets
of diameter ${\mathcal O}(\sigma)$ will yield asymptotically the same GMFRT
as that for a filament of cross-sectional radius $\eps$, when
\begin{equation}\label{sec3:eps_to_sigma}
  \eps = \frac{2L}{N} \exp\left( - \frac{L}{2\sigma N C} - \gamma_e
  \right) \,.
\end{equation}

In \S \ref{sec6:ball-ring} we perform a more detailed and general
analysis of the fragmentation of a filament of length $L$ into $N$
localized uniformly spaced small targets. We will show that the
relation \eqref{sec3:eps_to_sigma} can be derived from the analysis of
a new inner problem in the {\em transition} scaling regime ${\mathcal
O}(\sigma)\ll {L/N} \ll {\mathcal O}(1)$ that exists between
well-separated and closely-spaced targets.

\section{Trapezoidal Nystr\"{o}m discretization of the boundary integral equation}\label{sec4:numerics}

In this section we describe a boundary integral method for the
solution of the system \eqref{mprop:inteq} based on the trapezoidal
rule. The first step is to decompose the system into
filament self-interacting terms and filament coupling terms. For a
density $A(s)$ on a filament $\Gamma_j$, with the
self-filament integral operator $\K_j$ by
\begin{equation}\label{eqn:filament_self}
    (\K_j A)(s)\equiv\frac{A(s)}{\nu_j}+\frac12A(s)\,\mathcal{E}_j(s)
  -\frac12\int_{\Gamma_j}\frac{A(s)-A(s^{\prime})}{|\mathbf{X}_j(s)-
    \mathbf{X}_j(s^{\prime})|}\,\dd s^{\prime}
  +2\pi\int_{\Gamma_j}A(s^{\prime})\,R\big(\mathbf{X}_j(s);
  \mathbf{X}_j(s^{\prime})\big)\,\dd s^{\prime}\,,
\end{equation}
for $s\in [0,L_j]$ and $L_j$ is the length of the $j^{th}$
filament. For a line density $A(s)$,  the coupling terms
between filaments $\Gamma_i$ and $\Gamma_j$ for $i\neq j$ is
\begin{equation}\label{eqn:filament_coupled}
  \big(\mathcal{G}_{ji}A\big)(s)\equiv2\pi\int_{\Gamma_i}A(s^{\prime})\,
  G\big(\mathbf{X}_j(s);\mathbf{X}_i(s^{\prime})\big)\dd s^{\prime} .
\end{equation}

We now describe a trapezoidal rule for the discretization of the
integral operators
\eqref{eqn:filament_self}-\eqref{eqn:filament_coupled} and the numerical
solution of \eqref{mprop:inteq}. Our method is adapted
to the case of simple closed curves. We subdivide the $j^{th}$ filament
$\Gamma_j$ into $M_j$ equally spaced nodes in terms of arc-length,
by writing
\begin{equation}
    h_j = \frac{L_j}{M_j}\,, \quad s_{j,p} = p\, h_j\,, \quad \X_{j,p} =
    \X_j(s_{j,p})\,, \qquad p = 1,\ldots, M_j \,.
\end{equation}
In this setting, the trapezoid rule applied to a periodic function
$f(s)$ on the filament $\Gamma_j$ is 
\begin{equation}\label{eqn:trap_gam}
\int_{\Gamma_j} f(s)\dd s = h_j\sum_{p = 1}^{M_j} f(s_{j,p})\,.
\end{equation}
The unknowns are $A_{j,p} = A_j(s_{j,p})$ and
$\chi$.  Along each point on the curve, we employ the following
notation
\begin{equation}
r^{ji}_{pq}=|\X_{j,p}-\X_{i,q}|\,,\qquad
R^{ji}_{pq}=R(\X_{j,p};\X_{i,q})\,,\qquad
G^{ji}_{pq}=G(\X_{j,p};\X_{i,q})\,.
\end{equation}

We now focus on the discretization of the operator $\K_j$ for $j =
1,\ldots,N$ as defined in \eqref{eqn:filament_self}. We first recall
the reduced form of the self-energy term $\mathcal{E}_j(s)$ as derived
in \eqref{eq:E} of Appendix \ref{appa:num_intrinsic}
\begin{equation}
  \mathcal E_j(s)=2\log\left(\frac{4L_j}{\pi}\right) +\int_{\Gamma_j}
  \left[\frac{1}{|\X_j(s)-\X_j(s^{\prime})|}-
    \frac{\pi/L_j}{\big|\sin\big(\pi(s-s^{\prime})/L_j\big)\big|}\right]
  \dd s^{\prime}\,.
\end{equation}
Let ${ d_{pq}^j} = (L_j/\pi)\,|\sin(\pi(s_{j,p}-s_{j,q})/L_j)|$ be the
chord between nodes $p$ and $q$ {on $\Gamma_j$.}  Then applying
the trapezoidal rule to the first two terms of
\eqref{eqn:filament_self} yields
\begin{align}
  \frac12A_{j,p}\mathcal{E}_{j,p}-\frac{h_j}{2}\sum_{\substack{q=1\\ q\neq p}}^{M_j}
  \frac{A_{j,p}-A_{j,q}}{r^{jj}_{pq}}
&=\frac12A_{j,p}\left[2\log\frac{4L_j}{\pi}+h_j\sum_{\substack{q=1\\ q\neq p}}^{M_j}
    \left(\frac{1}{r^{jj}_{pq}}-\frac{1}{d_{pq}^j}\right)\right]-
  \frac{h_j}{2}\sum_{\substack{q=1\\ q\neq p}}^{M_j}
  \frac{A_{j,p}-A_{j,q}}{r^{jj}_{pq}}\notag\\
  &=\beta_{j}\,A_{j,p}+\frac{h_j}{2}\sum_{\substack{q=1\\ q\neq p}}^{M_j}
  \frac{A_{j,q}}{r^{jj}_{pq}} \,,\label{trap:self}
\end{align}
where after canceling the terms $A_{j,p}/r^{jj}_{pq}$, the remaining
diagonal term is of form
\begin{equation}\label{trap:beta}
  \beta_{j}\equiv\log\left(\frac{4L_j}{\pi}\right) - \frac{h_j}{2}
  \sum_{\substack{q=1\\ q\neq p}}^{M_j} \frac{1}{d_{pq}^j} \,,
\end{equation}
and can be evaluated once per filament. The remaining smooth terms
featuring the regular part $R$ in $\K_j$ and the coupling terms
$\mathcal{G}_{ji}$ are smooth and treated by a simple application of
\eqref{eqn:trap_gam}. The discretization of \eqref{mprop:inteq} for
each $j=1,\dots,N$ and $p=1,\dots,M_j$ is then,
\begin{equation}\label{trap:disc}
\Big(\frac{1}{\nu_j}+\beta_{j}\Big)A_{j,p}
+\frac{h_j}{2}\sum_{\substack{q=1\\ q\neq p}}^{M_j}\frac{A_{j,q}}{r^{jj}_{pq}}
+2\pi h_j\sum_{q=1}^{M_j}R^{jj}_{pq}A_{j,q}
+2\pi\sum_{\substack{i=1\\ i\neq j}}^{N}h_i\sum_{q=1}^{M_i}G^{ji}_{pq}A_{i,q}=\chi \,,
\end{equation}
together with the discrete form of the constraint,
\begin{equation}\label{trap:constraint}
\sum_{j=1}^{N}h_j\sum_{p=1}^{M_j}A_{j,p}=\frac{|\Omega|}{2\pi D}\,.
\end{equation}
This system of linear equations can be compactly represented in block
form as,
\begin{equation}\label{trap:system}
\begin{pmatrix}
\mathsf{M}_{11}&\cdots&\mathsf{M}_{1N}&-\mathbf{e}\\
\vdots&\ddots&\vdots&\phantom{-}\vdots\\
\mathsf{M}_{N1}&\cdots&\mathsf{M}_{NN}&-\mathbf{e}\\
h_1\mathbf{e}^{T}&\cdots&h_N\mathbf{e}^{T}&\phantom{-}0
\end{pmatrix}
\begin{pmatrix}\mathbf{A}_1\\ \vdots\\ \mathbf{A}_N\\ \chi\end{pmatrix}
=
\begin{pmatrix}\mathbf{0}\\ \vdots\\ \mathbf{0}\\ |\Omega|/(2\pi D)
\end{pmatrix}\,,
\end{equation}
where $\mathbf{A}_j=(A_{j,1},\dots,A_{j,M_j})^{T}$ and $\mathbf{e} =
(1,\ldots,1)^T$.  
 Each block matrix $\mathsf{M}_{ji}$ (of size $M_j \times M_i$)
is given elementwise as
\begin{equation}\label{trap:blocks}
\big(\mathsf{M}_{ji}\big)_{pq} = \left\{\begin{array}{cl}
\Big(\frac{1}{\nu_j}+\beta_{j}\Big)\delta_{pq}+
h_j\left[\frac{1-\delta_{pq}}{2r^{jj}_{pq}}+2\pi R^{jj}_{pq}\right],& i = j,\\
h_i\left[\frac{1}{2r^{ji}_{pq}}+2\pi R^{ji}_{pq}\right],& i\neq j \end{array}\right.
\qquad { \left( \genfrac{}{}{0pt}{}{p = 1,\ldots,M_j}{q = 1,\ldots,M_i}\right).}
\end{equation}
This system of dense equations is solved directly for the discretized
line strengths $\{\mathbf{A}_j\}_{j=1}^N$ and constant $\chi$. The
outer solution \eqref{mprop:uout} away from the filaments can then be
evaluated with the same weights:
\begin{equation}\label{trap:outer}
  u(\mathbf{x})\approx\chi-2\pi\sum_{j=1}^{N}h_j\sum_{q=1}^{M_i}A_{j,q}
  \,G(\mathbf{x};\mathbf{X}_{j,q})\,,\qquad
  \operatorname{dist}(\mathbf{x},\Gamma_j)\gg {\mathcal O}(\varepsilon)\, \quad j = 1,\ldots, N.
\end{equation}
The equivalent discretization scheme for the splitting probabilities $\phi_k(\x)$ described in
Proposition \ref{prop:splitting} is
\begin{equation}\label{trap:system_split}
\begin{pmatrix}
\mathsf{M}_{11}&\cdots&\mathsf{M}_{1N}&-\mathbf{e}\\
\vdots&\ddots&\vdots&\phantom{-}\vdots\\
\mathsf{M}_{N1}&\cdots&\mathsf{M}_{NN}&-\mathbf{e}\\
h_1\mathbf{e}^{T}&\cdots&h_N\mathbf{e}^{T}&\phantom{-}0
\end{pmatrix}
\begin{pmatrix}\mathbf{B}_1\\ \vdots\\ \mathbf{B}_N\\ \bar{\phi}_k\end{pmatrix}
=-
\begin{pmatrix}\mathbf{\delta}_{k1}\mathbf{e} \\ \vdots\\ \mathbf{\delta}_{kN}\mathbf{e}\\ 0
\end{pmatrix}\,,
\end{equation}
where $\mathbf{B}_j=(B_{j,1},\dots,B_{j,M_j})^{T}$ are the
corresponding splitting probability line densities and the constant
$\bar{\phi}_k$ is the spatially averaged splitting probability. After
solving the system \eqref{trap:system_split}, the splitting
probabilities are
\begin{equation}
  \phi_k(\mathbf{x})\approx\bar{\phi}_k-2\pi\sum_{j=1}^{N}h_j
  \sum_{q=1}^{M_i}B_{j,q}
  \,G(\mathbf{x};\mathbf{X}_{j,q})\,,\qquad
  \operatorname{dist}(\mathbf{x},\Gamma_j)\gg {\mathcal O}(\varepsilon)\, \quad j = 1,\ldots,N.
\end{equation}

\paragraph{Errors in the numerical method.}
The trapezoidal rule is spectrally accurate for smooth periodic
integrands \cite{DahlquistBjorck1974}. After the subtraction of the
chord kernel, the self-interaction integrands are continuous but have
a local behavior of the form $|s-s^{\prime}|$ at $s^{\prime}=s$, arising from
the curvature of $\Gamma_j$ and the variations in the line
densities, which results in a local discontinuity in the 
derivative (i.e. a corner). The corner produces
an $\mathcal{O}(h^2)$ quadrature error, resulting in a scheme that
converges as $M^{-2}$ for $M\to\infty$. All other integrals are smooth
and are computed with spectral accuracy. Higher-order accuracy can be
obtained by using Richardson extrapolation or locally corrected
trapezoidal rules
\cite{KapurRokhlin1997,Alpert1999}. For the purposes of validating the
asymptotic expansions, we find that this second-order method is
sufficient.

\paragraph{Non-arclength parameterizations. }
In the derivation, above we assumed that the filament centerline
is parameterized by arc-length $\X = \X(s)$. To treat
simple closed filament centerlines $\tilde{\X}(t)$ that are not
parameterized by arc-length, but instead on $0\leq t<2\pi$, our
approach is to numerically provide a map between $t$ and $s$ that can
be interpolated with high precision. In particular, by computing
$s=\int_{0}^{t} |\tilde{\X}^{\prime}(\tau)| \dd\tau$ and then
interpolating the map to obtain equal spacing in arclength, the
numerical discretization of the integral equations described above can
be used. The map is sufficiently resolved numerically so that we can 
neglect any interpolation errors in comparison to the errors associated 
with the trapezoidal rule.

\subsection{Examples: Ring filaments}\label{sec4:arc}

\begin{figure}[htbp]
    \centering
    \includegraphics[width=0.95\textwidth]{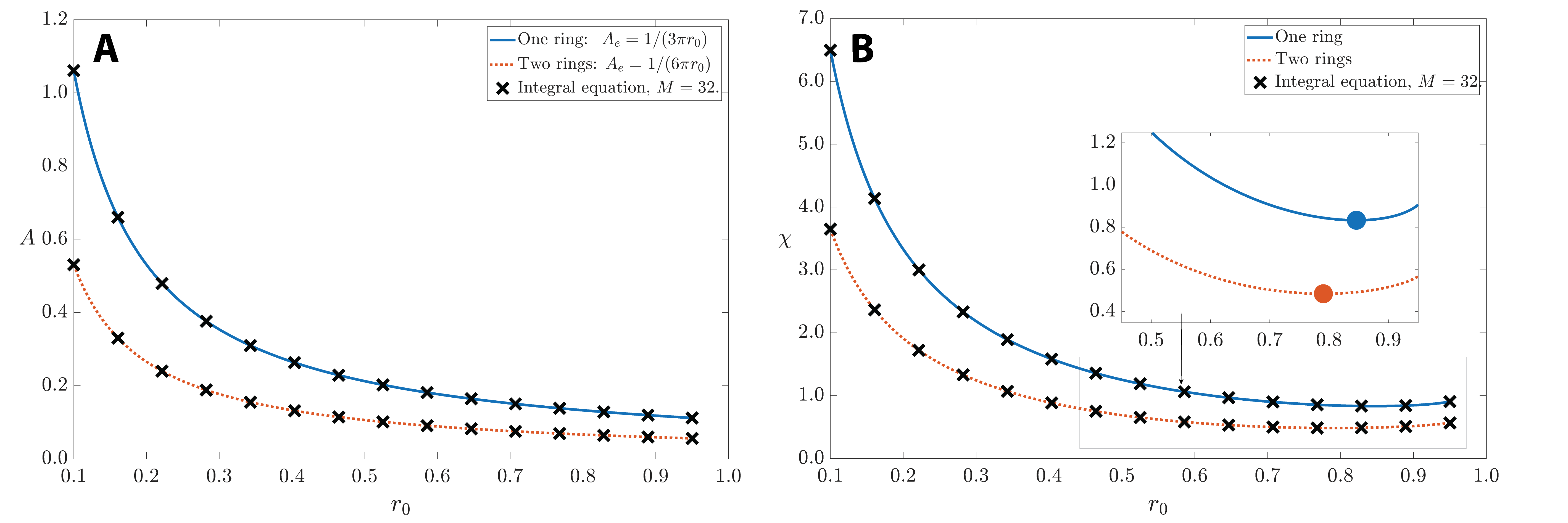}
\caption{
Convergence of the integral method in two examples: a single
concentric ring of radius $r_0$ centered at the origin, and two
concentric rings of radius $r_0$ on planes $z = \pm 0.1$, with $D=1$
and $\eps=10^{-3}$.  The numerical results by the integral equation
solver (crosses) compare favorably with the exact solutions (lines) 
for $A$ (Panel A) and $\chi$ (Panel B).
The inset shows the local minimum of the MFRT in $r_0$.}
\label{fig:ie_conv} 
\end{figure}

In this subsection we validate our numerical method for a few
configurations of ring or loop-type filaments, for which the intrinsic
self-energy of the filament is known analytically.

\begin{figure}[htbp]
    \centering
    \includegraphics[width = 0.8\textwidth]{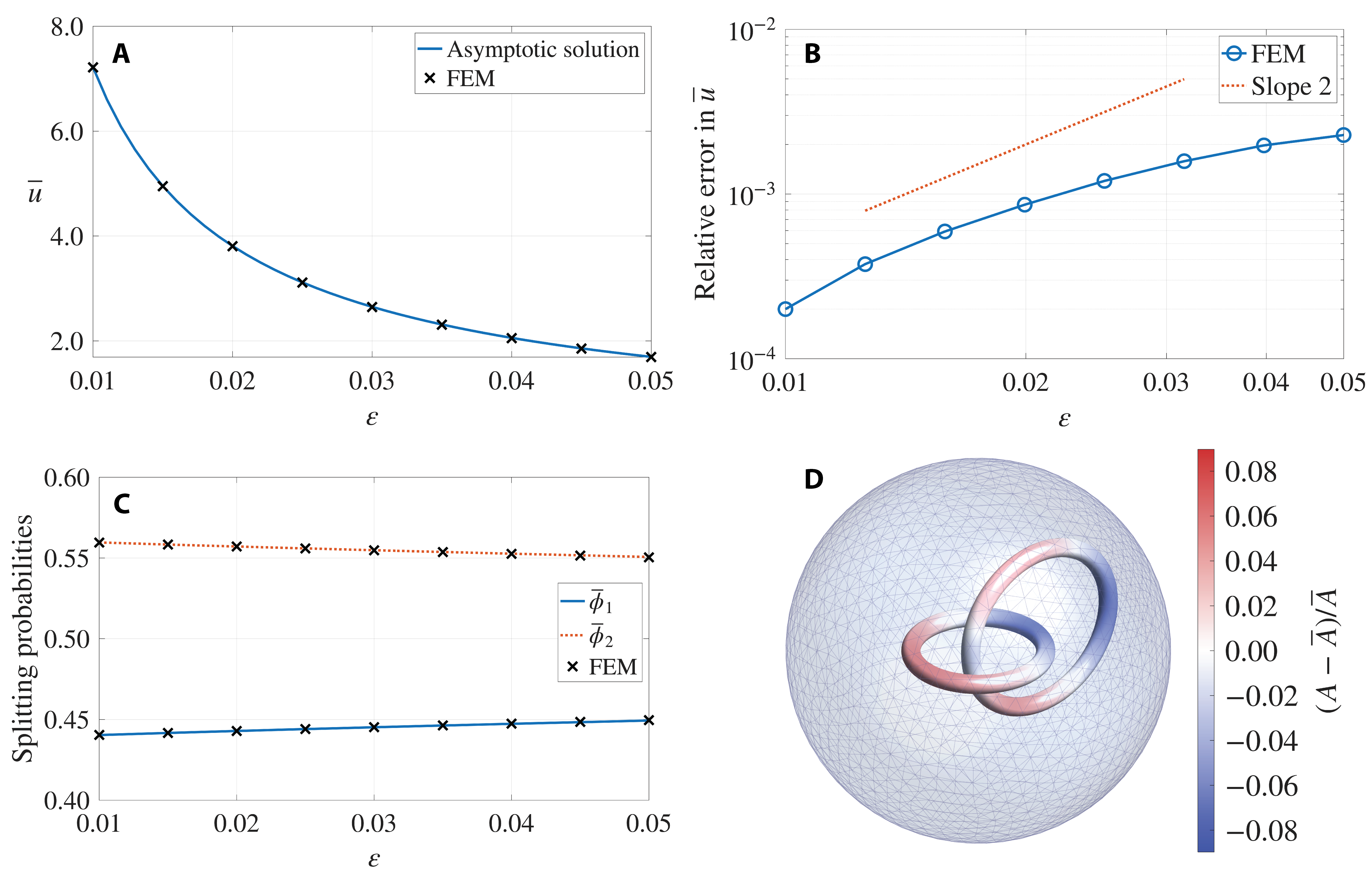}
\caption{
Validation of asymptotic reduction with full numerical
solution (FEM) for the two
filaments with common reactivity $b=2$ and $d_c(2)=e^{-\frac12}$. Panel A: Agreement between $\bar{u}$ from FEM solution and
the $\chi$ determined from the solution of the integral equation
\eqref{mprop:inteq}. Panel B: The relative error $|\bar{u}-\chi|/\bar{u}$ as $\eps\to0$ indicating the validity of the leading order approximation.  Panel C: Agreement of the splitting probabilities $\bar{\phi}_{1,2}$ with FEM and asymptotic method. Panel D: The spherical domain and filaments for $\eps=0.05$. The surfaces are shaded with relative fluxes 
$(A_i(s)-\bar{A}_i)/\bar{A}_i$ showing shielding near
the center of the loop. \label{fig:Interlacing_Filaments}}
\end{figure}

As an example, we first consider one perfectly absorbing ring filament
of radius $r_0$ of circular cross-section of radius $\eps$ in the
$z=0$ plane such as shown in Fig.~\ref{fig:sphere_tube}. The integral
system \eqref{prop:inteq} was solved explicitly for this case in
\S~\ref{sec2:conc_ring} and it was determined in \eqref{sec3:aring}
that the line density is simply $A=\equiv 1/(3\pi D r_0)$. Secondly, we
consider two identical perfectly absorbing rings with circular
cross-section centered at $(0,0,\pm z_0)$, in which symmetry implies
that $A_1(s) = A_2(s)$ for all $s$. From \eqref{sec5:cons}, this 
implies that $A = 1/(6\pi D r_0)$. We parameterize the two filament 
centerlines in terms of arclength as
$\X_{\pm}(s)=\left(r_0\cos\left({s/r_0}\right),r_0\sin\left({s/r_0}\right),\pm
z_0\right)^{T}$ with $0<s<2\pi r_0$. By scaling out $A_e$ in
\eqref{mprop:inteq}, we find that the volume-averaged MFRT is

\begin{equation}\label{mprop:tube_parallel}
\chi = A_e \chi_0 \,, \quad \mbox{where} \quad
 \chi_0 = \log\left(\frac{8r_0}{\eps}\right) +2\pi
 \int_{0}^{2\pi r_0}  R(\X_{+}(s);\X_{+}(s^{\prime})) \dd s^{\prime} +
 2 \pi  \int_{0}^{2\pi r_0}    G\big(\X_{+}(s);\X_{-}(s^{\prime})\big)\dd s^{\prime}
 \,,
\end{equation}

where the two integral terms are independent of $s$ by symmetry.  In
deriving \eqref{mprop:tube_parallel} we used $\nu={-1/\log \eps}$, and
${\mathcal E}(s)=2\log(8 r_0)$ for the common intrinsic self-energy of
the two filaments. In Figs.~\ref{fig:ie_conv}A-B, we show good
agreement between the values of $A$ and $\chi$ produced by the
integral equation solution and the closed form solutions over a range
of $r_0$ values for $\eps = 10^{-3}$ and both the single ring and the
two ring cases ($z_0 = 0.1$). In Fig.~\ref{fig:ie_conv}B, the inset
 compares the critical radii that minimize the global MFRTs for
these two examples.

As a second example, we consider two interlacing ring filaments such
as shown in Fig.~\ref{fig:Interlacing_Filaments}D whose centerlines
are given explicitly by

\begin{equation}\label{ex:two_filaments}
  \Gamma_1:\ \{\, z=0,\ x^2+y^2=0.35^2 \,\},\qquad
\Gamma_2:\ \{\, y=0,\ (x-0.4)^2+z^2=0.45^2 \,\} \,.
\end{equation}

For this configuration we solve for the MFRT and splitting
probabilities at various $\eps$ but assuming a uniform common
reactivity $b=2$, for which $d_c=e^{-1/2}$. The results are shown in
Fig.~\ref{fig:Interlacing_Filaments}. We observe in
Figs.~\ref{fig:Interlacing_Filaments}A and C that the asymptotic
solution of the MFRT $\bar{u}$ and the splitting probabilities
$\{\bar{\phi}_1,\bar{\phi}_2\}$ agree very closely with those computed
from a full finite element simulation based on quadratic elements
\cite{Geuzaine2009Gmsh}. The relative error in the approximation of
$\bar{u}$ is observed in Fig.~\ref{fig:Interlacing_Filaments}B to
exhibit an $\mathcal{O}(\eps^2)$ scaling as $\eps\to0$. In
Fig.~\ref{fig:Interlacing_Filaments}D we show the geometry together
with the filaments shaded by $(A_i(s)-\bar{A}_i)/\bar{A}_i$ for
$\bar{A}_i = |\Gamma_i|^{-1}\int_{\Gamma_i} A_i(s)\dd s$ when
$\eps=0.05$. A shielding effect is observed in which the flux to the
filaments is modulated by the competition of the neighboring geometry,
through interactions with the other filament and the exterior boundary
$\partial\Omega$.

\subsection{Examples: torus-knot filaments}

In this subsection, we consider torus-knot filaments that do not admit a
convenient arclength parameterization and for which we must
numerically compute their intrinsic self-energy.

A helical filament centerlines, parameterized by $t$, has the form
\bsub\label{eqn:filament_cases}
\begin{equation}
  \label{eqn:wiggly_curve_conv}   \X(t) =
  \Big( \big(r_0 + \rho_0\cos (p t)\big) \cos(qt),
  \big(r_0 + \rho_0\cos (p t)\big) \sin(qt),
  \rho_0 \sin (pt)\Big)\,, \qquad t\in[0,2\pi)\,.
\end{equation}

For positive integers $p$ and $q$, the curve \eqref{eqn:wiggly_curve_conv} is closed on
$0\le t<2\pi$. Provided $\rho_0<r_0$, it is a simple closed curve traversed exactly once
if and only if $\gcd(p,q)=1$. It then lies on a torus of major radius $r_0$ and minor
radius $\rho_0$, winding $q$ times around the $z$-axis and $p$ times around the centerline, and
therefore forms a $(q,p)$ torus knot. This is the \emph{unknot} when $p=1$ or $q=1$, and the
\emph{trefoil} when $\{p,q\}=\{2,3\}$.  The length of the centerline,
with Jacobian $\dd s=v(t)\dd t$ where $v(t)\equiv |\X^{\prime}(t)|$,
is

\begin{equation}
  L = \int_{0}^{2\pi} v(t) \dd t \,, \qquad \mbox{and} \qquad
  v(t) \equiv \Big[ \rho_0^2 p^2 + q^2\big( r_0 + \rho_0 \cos(p t)\big)^2
    \Big]^{\frac12} \,.
\end{equation}
\esub

\begin{figure}[htbp]
    \centering
    \includegraphics[width=0.95\textwidth]{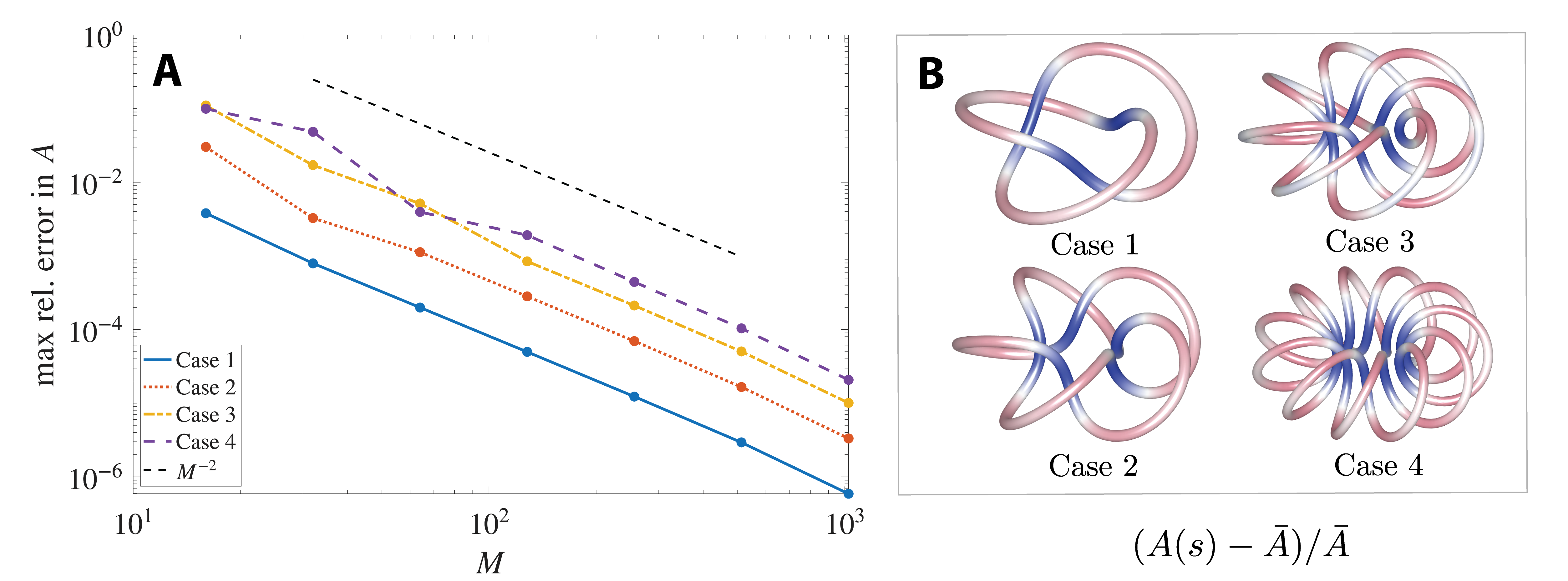}
\caption{
Convergence of the integral method. Panel A:  The maximum
relative error of $A$ computed by the integral method, as a function
of the number of discretization points $M$, for the four filaments
described in \eqref{eqn:filament_cases}, for $\eps=10^{-3}$.
The expected $M^{-2}$ convergence rate is shown by the dashed line, with
larger errors associated with more wavelengths in the
centerline. Panel B: plots of $(A(s)-\bar{A})/\bar{A}$ along
the filament surface where $\bar{A}=|\Gamma|^{-1}\int_{\Gamma}A(s)\dd
s$. Red and blue indicates regions of larger and lower relative
fluxes respectively. \label{fig:ie_conv_helix}}
\end{figure}

We consider four cases spanning a range of filament undulations parameterized on $0\leq t<2\pi$ by: 

\bsub\label{helix:ex1}
\begin{align}
  \label{eqn:filament_cases_a}    \mbox{Case 1}:& \, r_0 = 0.50 \,,
  \quad \rho_0 = 0.20\,, \quad p = 3\,,\phantom{0} \quad q = 2\,;\\
  \label{eqn:filament_cases_b}    \mbox{Case 2}:& \, r_0 = 0.50\,,
  \quad \rho_0 = 0.25\,, \quad p = 5\,,\phantom{0} \quad q = 2\,;\\
  \label{eqn:filament_cases_c}    \mbox{Case 3}:& \, r_0 = 0.55\,,
  \quad \rho_0 = 0.30\,, \quad p = 7\,, \phantom{0} \quad q = 3\,;\\
  \label{eqn:filament_cases_d}    \mbox{Case 4}:& \, r_0 = 0.55\,,
  \quad \rho_0 = 0.30\,, \quad   p = 11\,, \quad q = 2\,.
\end{align}
\esub
We choose the cross-sectional radius to be $\eps =10^{-3}$ and $d_c=
1$ consistent with a perfectly absorbing filament. By setting such a
small $\eps$, our goal for this example is to estimate the order of
the error in our numerical discretization of the integral equation
with respect to the number of sub-intervals. In
Fig.~\ref{fig:ie_conv_helix}A we demonstrate the expected
$\mathcal{O}(M^{-2})$ convergence rate of the relative errors in the
integral method. The surface of each filament is shown in
Fig.~\ref{fig:ie_conv_helix}B colored with $(A(s)-\bar{A})/\bar{A}$
for $\bar{A} = |\Gamma|^{-1}\int_{\Gamma}A(s)\dd s$ demonstrating
filament curvature effects on the density.

For our final example, we will consider two perfectly absorbing helical
filaments that have the same length, and we will determine whether the
specific type of filament centerline and its location in the unit
sphere has an appreciable affect on the global MFRT. We assume that
the filaments have a circular cross-sectional radius of either
$\eps=0.02$ or $\eps=0.05$. We will consider the following two cases:
\begin{figure}[htbp]
  \centering
     \includegraphics[width=0.9\textwidth]{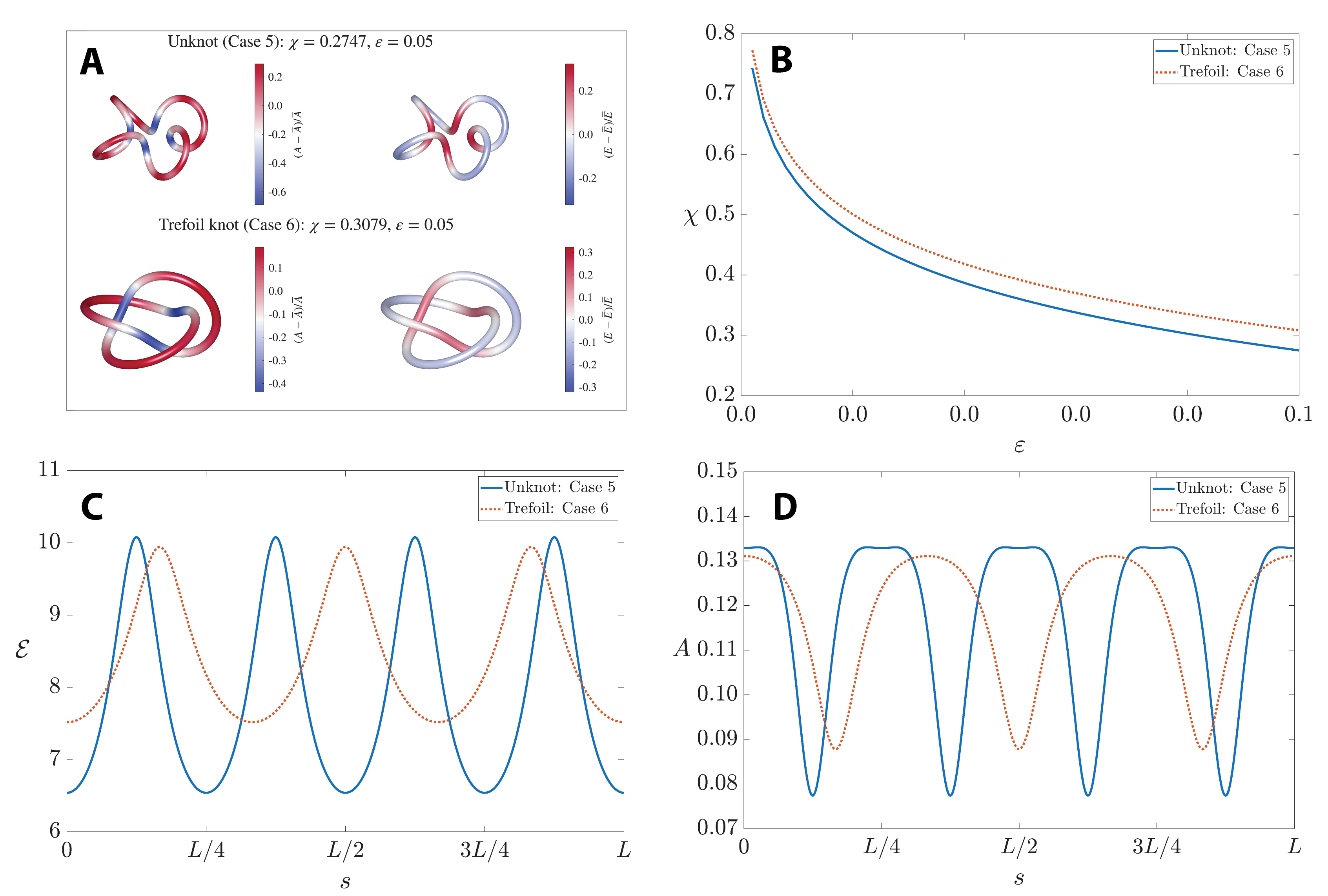}
\caption{
Solutions for the two helical filaments \eqref{eqn:wiggly_curve_conv}
and \eqref{helix:ex2} depicting the unknot \eqref{helix:ex2a} and the
trefoil knot \eqref{helix:ex2b}, respectively. The length of each
centerline is $L\approx 5.68$. Panel A: The filaments with their
relative intrinsic energies
$(\mathcal{E}(s)-\bar{\mathcal{E}})/\bar{\mathcal{E}}$ and relative
flux densities $(A(s)-\bar{A})/\bar{A}$ for $\eps
d_c=0.05$. A strong shielding effect is observed. Panel B: The MFRT
$\chi$ against $\eps$ for $d_c=1$ demonstrating that the unknot
generates a lower MFRT compared to Trefoil. Panel C: The intrinsic
energies $\mathcal{E}(s)$ against arclength for the two
filaments. Panel D: The flux densities $A(s)$ against arclength for
the two filaments when $\eps d_c=0.02$. \label{fig:helix_ex2}} 
\end{figure}

\bsub \label{helix:ex2}
  \begin{align}
  \label{helix:ex2a}  &  \mbox{Case 5}: \, r_0 = 0.40 \,,
  \quad \rho_0 = 0.2000\,, \quad p = 4\,, \quad q = 1 \,, \quad \mbox{(Unknot)}
  \,;\\
  \label{helix:ex2b}   & \mbox{Case 6}: \, r_0 = 0.40\,,
  \quad \rho_0 = 0.1366\,, \quad p = 3\,, \quad q = 2 \,, \quad
  \mbox{(Trefoil knot)} \,.
  \end{align}
\esub 

For both cases, the length of the filament centerline is $L\approx
5.68$, and so the surface area $2\pi \eps L$ of the capturing set is
the same for both helices. The filament centerlines are shown in
Fig.~\ref{fig:helix_ex2}.

By integrating \eqref{prop:inteq} over $\Gam$, we derive for each of
the two curves that

\begin{equation}\label{hel2:chi}
    \chi = \frac{|\Om|}{2\pi D \nu |\Gam|} + \frac{1}{2|\Gam|}
    \int_{\Gam} A(s) {\mathcal E}(s) \dd s + \frac{2\pi}{|\Gam|}
    \int_{\Gam}\int_{\Gam} A(s) R\left(\X(s);\X(s^{\prime})\right)
    \dd s \dd s^{\prime} \,.
\end{equation}

For each of the two examples in \eqref{helix:ex2} we have
$|\Gam|\approx 5.68$, and the average distance from the origin does
not differ substantially. As a result, we expect that the term
associated with the regular part $R$, which accounts for the filaments
interaction with the geometry, should be roughly similar for the two
curves. However, the two helices differ considerably in their
intrinsic energies owing to the different number of undulations. For
these two examples, we show in the results of Fig.~\ref{fig:helix_ex2}
that the unknot produces a lower MFRT. In Fig.~\ref{fig:helix_ex2} we
show the filaments themselves together with their intrinsic
self-energies and the flux densities for $\eps d_c = 0.02$ and $\eps
d_c = 0.05$.


\section{Extensions to related problems in diffusive capture}\label{sec5:extensions}

In this section we extend the asymptotic theory in two directions. In
\S \ref{sec5:varD} we consider the effect of a heterogeneous diffusivity
on the MFRT for a single filament, while in \S \ref{sec5:eig} we
calculate the principal eigenvalue for the Laplacian for a domain with
a single filament.

\subsection{Heterogeneous diffusivity}\label{sec5:varD}

We first extend \eqref{intro:mfrt} by allowing for a smooth heterogeneous
diffusivity $D(\x)>0$, so (with It\^{o} interpretation)

\begin{equation}\label{varD:mfrt}
\Delta_{\x} u=-\frac{1}{D(\x)} \,, \quad \x\in \Om\setminus\Gam_\eps\,; \qquad 
\eps \partial_n u + b u =0 \,, \quad \x \in \partial\Gam_\eps\,; \qquad
\partial_n u=0\, , \quad \x \in \partial\Om \,.
\end{equation}

The asymptotic analysis of \eqref{varD:mfrt} for the single filament
$\Gam_\eps$ is very similar to that for the case of uniform diffusivity.
More specifically, the construction of the solution in the inner
region is identical up to order ${\mathcal O}(\eps^2)$. As a result,
the outer solution that incorporates all logarithmic corrections has the
same singularity condition as $\x\to \X(s)$. In this way, we have
that the outer solution satisfies
\bsub \label{varD:out_1}
\begin{align}
    \Delta_{\x} u &= -\frac{1}{D(\x)} + 2\pi \int_{\Gam} A(s)
    \delta(\x-\X(s))  \dd s \,, \quad \x\in \Omega\setminus
    \Gam \,; \quad \partial_n u=0 \,, \quad \x\in
    \partial\Omega \,, \label{varD:out_1a} \\
    u &- A(s) \log|\x-\X(s)| \to \frac{A(s)}{\nu} \quad \mbox{as}
    \quad \x\to \X(s) \,, \quad s\in [0,L] \,, \label{varD:out_1b}
\end{align}
\esub where $\nu=-{1/\log[\eps d_c(b)]}$.
From an application of the divergence theorem to \eqref{sec2:out_1a} we
conclude that
\begin{equation}\label{varD:cons}
  \int_{\Gam} A(s) \dd s = \frac{1}{2\pi} \int_{\Om} \frac{1}{D(\x)} \dd \x\,.
\end{equation}

By adding and subtracting the term $|\Om|^{-1} \int_{\Om}
\left[D(\x)\right]^{-1} \dd \x$ to the right-hand side of
\eqref{varD:out_1a}, it is convenient to define $u_{p}(\x)$ to be the
unique solution to
\begin{equation}\label{varD:up}
  \Delta u_p =  -\frac{1}{D(\x)} + \frac{1}{|\Omega|} \int_{\Om}
  \frac{\dd \x^{\prime}}{ D(\x^{\prime}) } \,, \quad \x\in \Om \,; \qquad
  \partial_{n} u_{p} =0 \,, \quad \x\in \partial\Om \,; \qquad
  \int_{\Om} u_{p} \dd \x=0 \,,
\end{equation}
which is given in terms of the Neumann Green's function \eqref{intro:Gdef}
by
\begin{equation}\label{varD:up_solve}
  u_{p}(\x) = \int_{\Om} \frac{1}{D(\x^{\prime})} \GN(\x;\x^{\prime}) \,
  \dd \x^{\prime}\,.
\end{equation}
We observe that when $D(\x)$ is a constant in $\Om$, then $u_p\equiv
0$ since $\int_{\Om} \GN(\x;\x^{\prime}) \dd \x^{\prime}=0$. By
superposition, the solution to \eqref{varD:out_1} is determined in
terms of an unknown constant $\chi$ by
\begin{equation}\label{varD:urep}
  u(\x) = \chi + u_{p}(\x) - 2\pi \int_{\Gam} A(s)
  \GN(\x;\X(s))\dd s \,, \quad \mbox{for} \quad \x
  \notin \Gam \,,
\end{equation}
from which it follows that $\chi=|\Om|^{-1}\int_{\Om} u(\x) \dd \x$. As a
result, we conclude from \eqref{intro:mfrt_ave} and \eqref{varD:urep}
that the global MFRT is
\begin{equation}\label{varD:chi_ave}
   \bar{u} = \chi + o(1) \,.
\end{equation}

The remainder of the derivation of the integral equation for $A(s)$
parallels that for the uniform diffusivity case done in
\S \ref{sec2:outer}. The result is summarized as follows:

\begin{proposition}\label{propvd:int_eq} 
Let $\Gam$ be any smooth simple closed curve in $\Omega$ parameterized
by arclength as $\X(s)$, which has an ${\mathcal O}(\eps)$
cross-sectional diameter. In the limit $\eps\to 0$, the asymptotic
solution to \eqref{varD:mfrt} in the outer region at ${\mathcal O}(1)$
distances from $\Gam$ is
\begin{equation}\label{propvd:uout}
  u(\x) \sim\chi + u_p(\x) - 2\pi \int_{\Gam} A(s) \GN(\x;\X(s))\dd
  s \,, \quad \mbox{for} \quad \x \notin \Gam \,,
\end{equation}
where $u_{p}(\x)$ is the solution \eqref{varD:up_solve} to
\eqref{varD:up}, and where $A(s)$ and $\chi$ are solutions to the
integral equation with constraint, defined for all $s\in [0,L]$ where $L=|\Gam|$, by
\begin{equation}\label{propvd:inteq}
    (\mathcal{K}A)(s) - u_{p}\left(\X(s)\right) =\chi \,, \qquad \int_{\Gam} A(s) \dd s =
\frac{1}{2\pi}\int_{\Om} \frac{\dd \x}{D(\x)}.
\end{equation}
Here $\mathcal{K}$ is defined in \eqref{intro:ops}, $\GN$ is
the Neumann Green's function with regular part $R$ satisfying
\eqref{intro:Gdef}.
\end{proposition}

{\bf Example:} To illustrate the theory, we consider the unit sphere
with an equatorial ring of radius $r_0$ centered at the origin. We
will explore the effect on the global MFRT of a heterogeneous
diffusivity that is symmetric with respect to the plane $z=0$ but that
is larger at the poles than it is on the equator. However, we will fix
the harmonic mean of $D(\x)$ to unity, by the constraint $\int_{\Om}
\left[ D(\x)\right]^{-1} \dd\x=|\Omega|$ so that we can contrast our
results for the MFRT with that for the case of a uniform diffusivity
$D=1$. For a heterogeneous diffusivity with smaller values near the
equatorial filament, one would expect that the average MFRT is smaller
(i.e.~the expected time for capture is faster), than with a spatially
uniform diffusivity.

To explore this idea, we consider the choice ${1/D(\x)}=\beta_1 -\beta_2
z^2$ where $\beta_1>\beta_2>0$ in the unit sphere. Upon setting
$\int_{\Om} \left[D(\x)\right]^{-1}  \dd \x={4\pi/3}$, we obtain that
$5\beta_1-\beta_2=5$. With this choice, $u_{p}(\x)$ in \eqref{varD:up}
satisfies
\begin{equation}\label{varde:u0p}
  \Delta u_{p} = \beta_2 z^2-(\beta_1-1)=\beta_2 \left(z^2-{1/5}\right) \,,
  \quad \x\in \Om\,; \qquad
  \partial_n u_{p}=0 \,, \quad \x\in\partial\Omega\,; \qquad
  \int_{\Om} u_{p} \, \dd \x =0 \,, 
\end{equation}
which has the explicit solution in terms of $r=|\x|$ and $z$ given by
\begin{equation}\label{varde:u0pe}
  u_{p}(r,z) = \beta_2 \left( \frac{9}{700} + \frac{r^2}{70} - \frac{r^4}{140}
  + \frac{1}{14} (r^2-2) z^2 \right) \,.
\end{equation}
Evaluating \eqref{varde:u0pe} on the ring $r=r_0$ on the plane $z=0$
we get
\begin{equation}\label{varde:u0pe_ring}
  u_{p}(r_0) = \frac{\beta_2}{70}\left( \frac{7}{5} -\frac{1}{2}
  \left(1-r_0^2\right)^2\right)\,.
\end{equation}
We observe that $u_{p}(r_0)>0$ on $0\leq r_0\leq 1$ and that $u_{p}$
is monotone increasing in $r_0$.

Finally, upon substituting \eqref{varde:u0pe_ring} into
\eqref{propvd:inteq} we conclude that the global MFRT $\chi$ with this
choice of heterogeneous diffusivity is related to that with uniform
diffusivity $D=1$, labeled by $\chi_{\textrm{unif}}$, by a simple
shift $\chi = \chi_{\textrm{unif}} - u_{0p}(r_0)$, so that
\begin{equation}\label{vard;shift}
  \chi = \chi_{\textrm{unif}}- \frac{\beta_2}{70}\left( \frac{7}{5} -
  \frac{1}{2}\left(1-r_0^2\right)^2\right)\,,
\end{equation}
where $\chi_{\textrm{unif}}$ is given in \eqref{sec3:ubar} when
$D=1$. Here $\beta_2>0$ is any solution to $5\beta_1-\beta_2=5$ with
$\beta_1>\beta_2>0$. We conclude, as expected, that the global MFRT
$\chi$ decreases when the Brownian particle experiences a smaller
diffusivity near the equatorial plane where the capturing filament is
located.

In particular, if we set $\beta_2=1$ then $\beta_1={6/5}$, and the
heterogeneous diffusivity is $D(\x)=\left(\frac{6}{5}-z^2
\right)^{-1}$, which is $6$ times higher at the poles $z=\pm 1$ than
on the equatorial plane $z=0$. Setting $\eps d_c=0.02$ and evaluating
\eqref{vard;shift} at the previously optimal ring radius $r_{0}\approx
0.7680$ for the uniform diffusivity $D=1$ where
$\chi_{\textrm{unif}}\approx 0.4406$ (see second row of Table
\ref{tab:mfpt_optim}), we obtain that $\chi\approx 0.4218$, which is
approximately a $4.3\%$ reduction in the global MFRT due to this
specific heterogeneous diffusivity.

\subsection{The principal eigenvalue of the Laplacian}\label{sec5:eig}

In this subsection we consider the Laplacian eigenvalue problem in a
bounded 3-D domain with one partially reactive filament boundary
$\Gam_\eps$, formulated, apart from a normalization of the eigenfunction,
as 
\begin{equation}\label{eig:main}
\Delta u + \lambda u =0 \,, \quad \x\in
\Om\setminus\Gam_\eps\,; \qquad\eps \partial_n u + b u =0 \,, \quad \x
\in \partial\Gam_\eps\,; \qquad \partial_n u=0 \,, \quad \x\in \partial\Om\,.
\end{equation}

We seek to determine an approximation of the lowest eigenvalue of
\eqref{eig:main} that is accurate to all orders in
$\nu\equiv  {-1/\log[\eps d_c]}$; i.e. that effectively ``sums'' all
logarithmic terms in $\eps$. Motivated by the related analyses
in \cite{WHK} and \cite{KTW} for Laplacian eigenvalues in
2-D domains with spatially localized traps, we will expand $\lambda$
and the corresponding eigenfunction $u$ in the outer region as
\begin{equation}\label{eig:expan}
  \lambda = \lambda_{0}(\nu) + {\mathcal O}(\zeta)\,; \qquad
  u(\x) = u_{0}(\x;\nu) + {\mathcal O}(\zeta) \, \quad \mbox{for}
  \quad \dist(\x,\Gam) \gg {\mathcal O}(\eps) \,,
\end{equation}
where the correction $\zeta$ is transcendentally small with respect to
$\nu$ in the sense that $\zeta\ll \nu^k$ for any $k>0$.

The inner expansion near the filament centerline is identical to that
for the MFRT problem, and so in place of \eqref{sec2:out_1} we have in the
outer region that
\bsub \label{eig:out_1}
\begin{align}
    \Delta u_0 + \lambda_0 u_0 &=  2\pi \int_{\Gam} A(s)
    \delta(\x-\X(s))  \dd s \,, \quad \x\in \Omega\setminus
    \Gam \,; \qquad \partial_n u_0=0 \,, \quad \x\in
    \partial\Omega \,, \label{eig:out_1a} \\
    u_0 &- A(s) \log|\x-\X(s)| \to \frac{A(s)}{\nu} \quad \mbox{as}
    \quad \x\to \X(s) \,, \quad s\in [0,L) \,, \label{eig:out_1b}
\end{align}
\esub
for $L=|\Gam|$ and some line density $A(s)$. The solution to
\eqref{eig:out_1} is represented as
\begin{equation}\label{eig:u0_solve}
  u_{0}(\x;\lambda_0) = -2\pi \int_{\Gam} A(s)
  G_{h}(\x;\X(s);\lambda_0) \dd s\,,
\end{equation}
where for convenience below it is useful to enforce the normalization
condition $\int_{\Gam} A(s) \dd s=1$ for the eigenfunction. In
\eqref{eig:u0_solve}, $G_h(\x;\bxi;\lambda_0)$ is the Helmholtz
Green's function satisfying \bsub \label{eig:Gdefall}
\begin{equation} 
  \Delta G_{h} + \lambda_0 G_h = -\delta(\x-\bxi)\,, \quad \x\in \Omega
  \,; \qquad \partial_n G_h =0\,,\quad \x\in \partial\Omega \,,
 \label{eig:Gdef}
\end{equation}
which can be decomposed into a singular part and a smooth
regular part $R_h(\x;\bxi;\lambda_0)$ as
\begin{equation}  \label{eig:Gdef2}
  G_h(\x;\bxi;\lambda_0) =\frac{1}{4\pi|\x-\bxi|}+R_h(\x;\bxi;\lambda_0)\,.
\end{equation}
\esub
When $\Om$ is the unit sphere, an explicit infinite series representation for
$R_h$ is given in \eqref{appc:rh} of Appendix \ref{appc:eig}.

The derivation of the integral equation that determines $A(s)$ and
$\lambda_0$ is the same as in \S \ref{sec2:outer}, and yields
\begin{equation}\label{eig:inteq}
    \frac{A(s)}{\nu} + \frac{1}{2} A(s) {\mathcal E}(s) - \frac{1}{2}
    \int_{\Gam} \frac{ A(s)-A(s^{\prime})}{|\X(s)-\X(s^{\prime})|} \dd
    s^{\prime} +2\pi \int_{\Gam} A(s^{\prime}) R_h(\X(s);
    \X(s^{\prime});\lambda_0) \dd s^{\prime}=0 \,; \quad
    \int_{\Gam} A(s) \dd s=1 \,,
\end{equation}
where $\nu=-{1/\log[\eps d_c(b)]}$. To approximate the lowest
eigenvalue of $\lambda_0$ of \eqref{eig:main} that includes all
logarithmic corrections, we must determine the smallest positive value
$\lambda_0=\lambda_{0}^{\star}(\nu)$ for which the integral equation
\eqref{eig:inteq} has a line density $A(s)$ that does not vanish
identically on $\Gam$. A numerical discretization of \eqref{eig:inteq}
would lead to a nonlinear matrix eigenvalue problem to approximate
$\lambda_{0}^{\star}$. We shall refer to $\lambda_{0}^{\star}(\nu)$ as
the {\em log-summed} approximation to the lowest eigenvalue of
\eqref{eig:main}. 

\subsubsection{A two-term expansion}\label{eig:eig2t}

We first derive a two-term-in-$\nu$ expansion for the lowest root
$\lambda_0$ that is applicable to arbitrary $\Omega$, and we will
relate it to the two-term expansion of the global MFRT derived in
\eqref{prop:2term_chi} of Proposition \ref{prop:2term}. To do so, we
first readily derive from a perturbation expansion of the PDE
\eqref{eig:Gdef} (see \cite{KTW}) that
\begin{equation}\label{eig:rsmall}
  R_{h}(\x;\bxi;\lambda_0) = -\frac{1}{\lambda_0 |\Om|} + R(\x;\bxi)
    + {\mathcal O}(\lambda_0)\,, \quad \mbox{as} \quad \lambda_0 \to 0\,,
\end{equation}
where $R(\x;\bxi)$ is the regular part of the Neumann Green's function in
\eqref{intro:Gdef2}. Upon substituting \eqref{eig:rsmall} into
\eqref{eig:inteq} we obtain for $\lambda_0\ll 1$ 
\begin{equation}\label{eig:two_1}
  \frac{A(s)}{\nu} - \frac{2\pi}{\lambda_0 |\Om|} \int_{\Gam} A(s) \dd s
  + \left(\frac{1}{2} A(s) {\mathcal E}(s) - \frac{1}{2}
    \int_{\Gam} \frac{ A(s)-A(s^{\prime})}{|\X(s)-\X(s^{\prime})|} \dd
    s^{\prime}   +2\pi \int_{\Gam} A(s^{\prime}) R(\X(s); \X(s^{\prime}))
 \dd
    s^{\prime}    \right) = {\mathcal O}(\lambda_0) \,.
\end{equation}
We then substitute the expansions
\begin{equation}\label{eig:two_ex}
  A(s)=A_0(s)+ \nu A_1(s)+{\mathcal O}(\nu^2) \,, \qquad
  \lambda_0 = \nu\lambda_{00} +  \nu^2\lambda_{01} + {\mathcal O}(\nu^3)\,,
\end{equation}
into \eqref{eig:two_1} to obtain, subject to $\int_{\Gam} A(s) \dd s=1$, that
\begin{equation}\label{eig:two_mess}
  \begin{split}
  \frac{A_0}{\nu}  -\frac{2\pi}{\nu |\Om|\lambda_{00}} &
  \left( 1 + \frac{\nu \lambda_{01}}{\lambda_{00}}\right)^{-1}
  \int_{\Gam} \left(A_0 + \nu A_1\right) \dd s^{\prime}  \\
  & \qquad + A_1 + \frac{1}{2} A_0 {\mathcal E}(s) + 
  +2\pi \int_{\Gam} A_0(s^{\prime}) R(\X(s); \X(s^{\prime}))\dd
    s^{\prime}  =
  {\mathcal O}(\lambda_0)\,.
  \end{split}
\end{equation}
From the leading-order ${\mathcal O}(\nu^{-1})$ terms in
\eqref{eig:two_mess}
  we conclude that
\begin{equation}\label{eig:two_ex_a}
  A_{0}(s) = \frac{2\pi}{\lambda_{00} |\Om|} \int_{\Gam} A_0(s) \dd s \,,
  \quad \forall \,  s \in [0,L) \,.
\end{equation}
As such, $A_{0}(s)$ must be a constant, and from $\int_{\Gam} A_0(s)
\dd s=1$, we get $A_0={1/|\Gam|}$. Then, \eqref{eig:two_ex_a}
yields that
\begin{equation}\label{eig:two_ex_b}
  \lambda_{00} = \frac{2\pi |\Gam|}{|\Om|} \,.
\end{equation}

By collecting the ${\mathcal O}(1)$ terms in \eqref{eig:two_mess},
and using $\int_{\Gam} A_1(s) \dd s=0$, we conclude that
\begin{equation}\label{eig:two_ex_a1}
  -\frac{2\pi\lambda_{01}}{|\Om| \lambda_{00}^2} = A_{1}(s) +
  \frac{1}{2|\Gam|} {\mathcal E}(s) + \frac{2\pi}{|\Gam|}
  \int_{\Gam} R(\X(s);\X(s^{\prime})) \dd s^{\prime} \,,
    \quad \forall \,  s \in [0,L) \,.
\end{equation}
Instead of determining $A_{1}(s)$, we can obtain $\lambda_{01}$ by
simply integrating \eqref{eig:two_ex_a1} over $\Gam$ after using
$\int_{\Gam} A_1 \dd s=0$. We summarize our resulting two-term
expansion for the lowest eigenvalue of \eqref{eig:main} as follows:

\begin{proposition}\label{prop:eig_2term} 
In the limit $\nu={-1/\log[\eps d_c]} \to 0$, the lowest eigenvalue
for \eqref{eig:main} has the explicit two-term approximation
\begin{equation}\label{prop:eig_2term_m}
    \lambda_0 \sim \frac{2\pi\nu |\Gam|}{|\Om|} \left( 1-
    \frac{\nu}{2|\Gam|} \int_{\Gam} {\mathcal E}(s) \dd s -
      \frac{2\pi\nu }{|\Gam|} \int_{\Gam} \int_{\Gam} R\left(\X(s);
      \X(s^{\prime}) \right) \dd s \dd s^{\prime} + {\mathcal
        O}(\nu^2) \right) \,.
\end{equation}
Here $R$ is the regular part of the Neumann Green's function in
\eqref{intro:Gdef2} and ${\mathcal E}(s)$ is the intrinsic self-energy
of $\Gamma$ given by \eqref{sec2:intrinsic}.
\end{proposition}

Upon comparing \eqref{prop:eig_2term_m} with the corresponding two-term
expansion of $\bar{u}\approx \chi$ in \eqref{prop:2term_chi}, we
conclude up to two-orders in $\nu$ we have the expected correspondence
\begin{equation}\label{eig:corres}
  \chi = \frac{1}{D\lambda_0} + {\mathcal O}(\nu) \,, \quad
  \mbox{where} \quad \lambda_{0} = \nu\lambda_{00} + \nu^2\lambda_{01} +
  {\mathcal O}(\nu^3) \,.
\end{equation}
It is unclear whether this relation holds for all logarithmic orders.
As a result, for a concentric equatorial ring filament of radius $r_0$
in the unit sphere, we can use \eqref{sec3:ubar} for $\chi$ to obtain
the two-term approximation
\begin{equation}\label{eig:12term}
  \lambda_0 \sim 3\pi r_0\nu \left[
    1 +  r_0 \nu \left( \frac{\log(8r_0)}{r_0} + 2 K(r_0^2) + \pi
    \left(r_0^2 + \log{2} - \frac{14}{5} \right) + \frac{J(r_0)}{2} \right)
    \right]^{-1} \,,
\end{equation}
which is asymptotically equivalent to \eqref{prop:eig_2term_m}. In
\eqref{eig:12term}, $K(m)$ is the complete elliptic integral and $J(r_0)$
is defined in \eqref{sec3:I0J0}.

\subsubsection{Summing the log expansion for a ring filament}\label{eig:ringsum}

In an arbitrary 3-D domain with an arbitrary filament shape, a full
numerical implementation of \eqref{eig:inteq} is needed to compute the
log-summed appproximation $\lambda_{0}^{\star}$ in terms of a 
numerically computable $R_h$, e.g., such as from a fast-integral method.
For the unit sphere where $R_h$ has an infinite series expansion, a
numerical discretization of \eqref{eig:inteq} would lead to a
nonlinear matrix eigenvalue problem to approximate
$\lambda_{0}^{\star}$.

However, in the special case when $\Om$ is the unit sphere and when
$\Gam$ is an equatorial ring of radius $r_0$, we now show how to
determine $\lambda_{0}^{\star}$ from a simple root-finding
problem. For this configuration, it follows by symmetry that the
line density has the constant value $A(s)={1/L}$, with $L=2\pi
r_0$. Therefore, from \eqref{eig:inteq} with ${\mathcal E}(s)=2\log(8
r_0)$ we obtain that $\lambda_0$ satisfies
\begin{equation}\label{eig:inteq_ring}
   \int_{\Gam}  R_h(\X(s);  \X(s^{\prime});\lambda_0)
  \dd s^{\prime} = -\frac{1}{2\pi} \left[ \frac{1}{\nu} + \log(8r_0)\right] =
  -\frac{1}{2\pi} \log\left(\frac{8r_0}{\eps d_c}\right) \,.
\end{equation}
By using an infinite series representation of $R_h$, in Appendix
\ref{appc:eig} we show that the determination of $\lambda_{0}^{\star}$
can be reduced to determining the smallest positive root of
${\mathcal N}(\lambda_0)=0$, defined by
\bsub\label{eig:nfunc}
\begin{equation}\label{eig:nfunc_1}
  {\mathcal N}(\lambda_0) \equiv \frac{1}{2\pi}
  \log\left( \frac{8 r_0}{\eps d_c}\right) + \sum_{k=0}^{\infty}
  M_{k}(\lambda_0,r_0) \,; \qquad
  M_{k}(\lambda_0,r_0) \equiv \left[ \sqrt{\lambda_0} r_0 \left(2k+\frac{1}{2}
    \right) H_{2k} - \frac{1}{2} \right]
    \left( \frac{ (2k)!}{ 2^{2k} (k!)^{2}} \right)^{2}\,,
\end{equation}
where $H_{2k}=H_{2k}(\lambda_0,r_0)$ is defined by
\begin{equation}\label{eig:nfunc_h}
     H_{2k}(\lambda_0,r_0) \equiv
  \left[ y_{2k}^{\prime}(\sqrt{\lambda_0}) j_{2k}(\sqrt{\lambda_0} r_0) -
    j_{2k}^{\prime}(\sqrt{\lambda_0}) y_{2k}(\sqrt{\lambda_0} r_0) \right]
  \frac{ j_{2k}(\sqrt{\lambda_0} r_0)}{j_{2k}^{\prime}(\sqrt{\lambda_0})} \,.
\end{equation}
\esub
Here $j_{2k}(z)$ and $y_{2k}(z)$ are the spherical Bessel functions of the
first and second kind of order $2k$.

With regards to numerically computing the root to ${\mathcal
  N}(\lambda_0)=0$, we can readily estimate that $M_{k}={\mathcal
  O}\left({r^{4k+1}/k}\right)$ for large $k$ by using Stirling's
approximation and the large-order approximation of the spherical
Bessel functions. As a result, $M_k$ has a power-law decay as $k$
increases when $r$ is not too close to the boundary $r=1$. For
$r_0=0.4$, and by retaining only 10 terms in the sum in
\eqref{eig:nfunc_1}, in Fig.~\ref{fig:eig_N} we plot ${\mathcal
  N}(\lambda_0)$ versus $\lambda_0$ when $\eps d_c=0.02$ showing the
existence of a root. In fact, retaining only 3 terms in the sum
provides three decimal places of accuracy for the root. By tracking
this root as $\eps d_c$ is varied, in Fig.~\ref{fig:eig_comp} and
Table \ref{tab:eig} we compare the log-summed approximation
$\lambda_{0}^{\star}(\nu)$ with its corresponding one-term
approximation $\lambda\sim 3\pi \nu r_0$ and two-term approximation
from \eqref{eig:12term}. The full numerical FEM results are also
shown. We observe that the one-term approximation is poor unless
$\eps$ is extremely small, while the two-term appproximation decently
predicts $\lambda_{0}^{\star}$ to within $3\%$ only for $\eps d_c\leq
0.02$.  From this table, we see that the log-summed result is within
$4.7\%$ of the full numerical result even when $\eps=0.10$.

For fixed $\eps$, it was shown in Apppendix \ref{app:steklov} that
$d_{c}(b)$ is monotone increasing in the reactivity $b$ of the
filament. As a result, we observe from Fig.~\ref{fig:eig_comp} that
the time-scale for capture of a Brownian particle, determined by
${1/\lambda_{0}^{\star}}$, is smaller as $b$ increases, as expected
intuitively. Moreover, by changing the cross-sectional shape of the
filament from a disk or its reactivity, we can still use
Fig.~\ref{fig:eig_comp} after simple computing $d_{c}(b)$
numerically. For instance for an elliptical cross-section
$\partial{\mathcal B}$ of semi-axes $a_1=1$ and $a_2=2$ with $b=1$, we
calculate $d_c(b) \approx 0.7725$ \cite{Grebenkov25b}.  As a result,
we conclude that a perfectly absorbing filament with a circular
cross-sectional radius $\eps_0$ with $\eps_0\equiv \eps d_c=\eps
0.7725$ would have the same $\lambda_{0}^{\star}$ as that for the
partially reactive elliptical cross-section.

\begin{table}[ht]
\centering
\begin{tabular}{cccccc}
\hline
$\eps d_c$  & $\lambda_0\,$ (1-term) & $\lambda_0\,$ (2-term) &
$\lambda_{0}^{\star}\,$ (log-summed) & $\lambda_0\,$ (num) & $\%\,$ error \\
\hline
0.005   & 0.7115     & 0.8543 & 0.8447 & 0.8449  & 0.02 \\
0.01    & 0.8183     & 1.0135  & 0.9975 & 0.9983 & 0.08 \\
0.02    & 0.9637     & 1.2456 & 1.2161 & 1.2196  & 0.28  \\
0.05   & 1.2584        & 1.7865 & 1.7016 & 1.7263 & 1.4 \\
0.10   & 1.6373     & 2.6603  & 2.3993 & 2.5189 & 4.7 \\
\hline
\end{tabular}
\caption{Comparison of one-term $\lambda_0\sim 3\pi \nu r_0$, and the
  two-term approximation in \eqref{eig:12term} for the lowest
  eigenvalue of \eqref{eig:main} for an equatorial ring filament of
  radius $r_0=0.4$ in the unit sphere with the corresponding
  log-summed approximation $\lambda_0^{\star}$ given by the smallest
  root of \eqref{eig:nfunc} and the FEM full numerical result
  $\lambda_0$ (num). The percentage error in the last column is
  between the log-summed and the FEM numerical result. This error is
  less than $1.4\%$ on the range $\eps\le 0.05$.}
 \label{tab:eig}
\end{table}

\begin{figure}[h!tbp]
  \centering
      \begin{subfigure}[b]{0.48\textwidth}  
        \includegraphics[width=\textwidth]{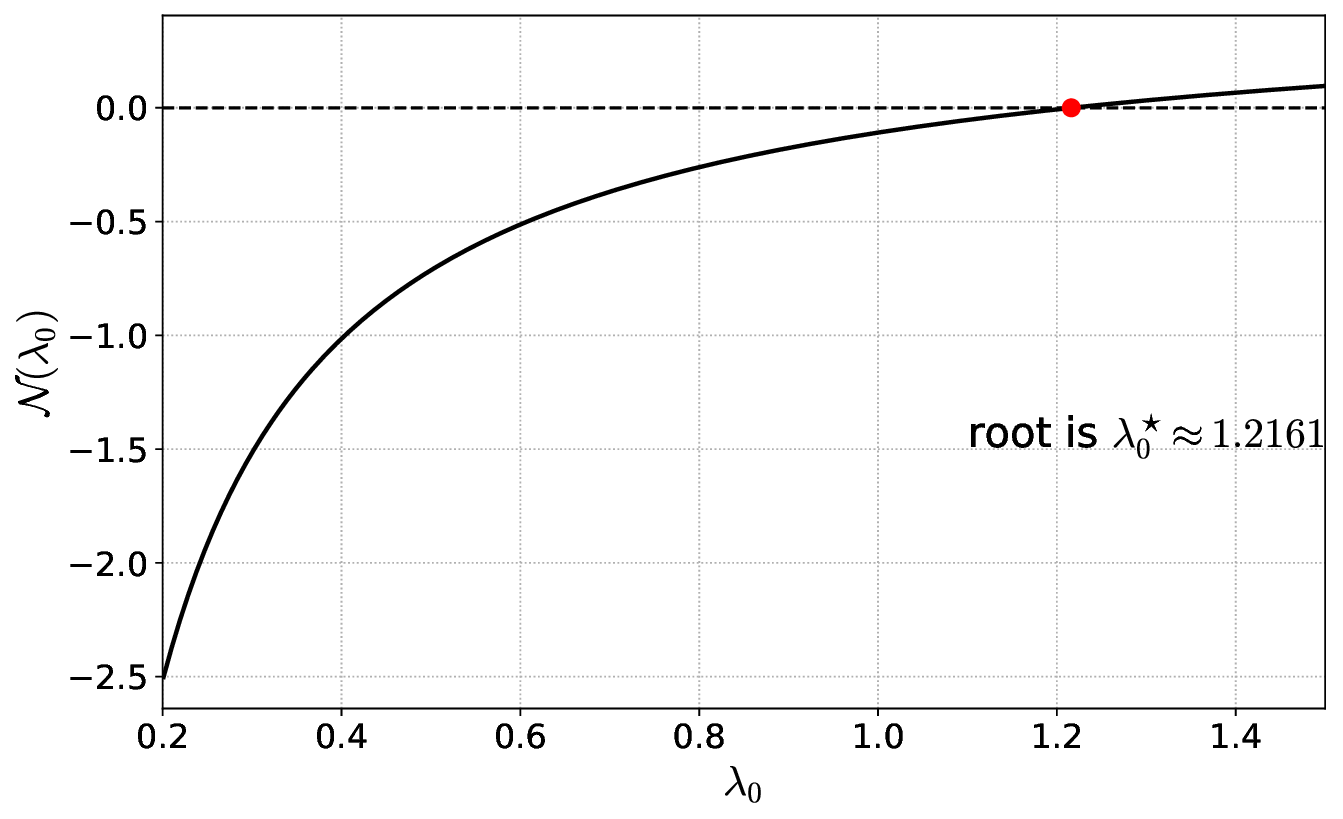}
        \caption{${\mathcal N}(\lambda_0)$ vs. $\lambda_0$}
        \label{fig:eig_N}
    \end{subfigure}
      \begin{subfigure}[b]{0.48\textwidth}
            \includegraphics[width=\textwidth]{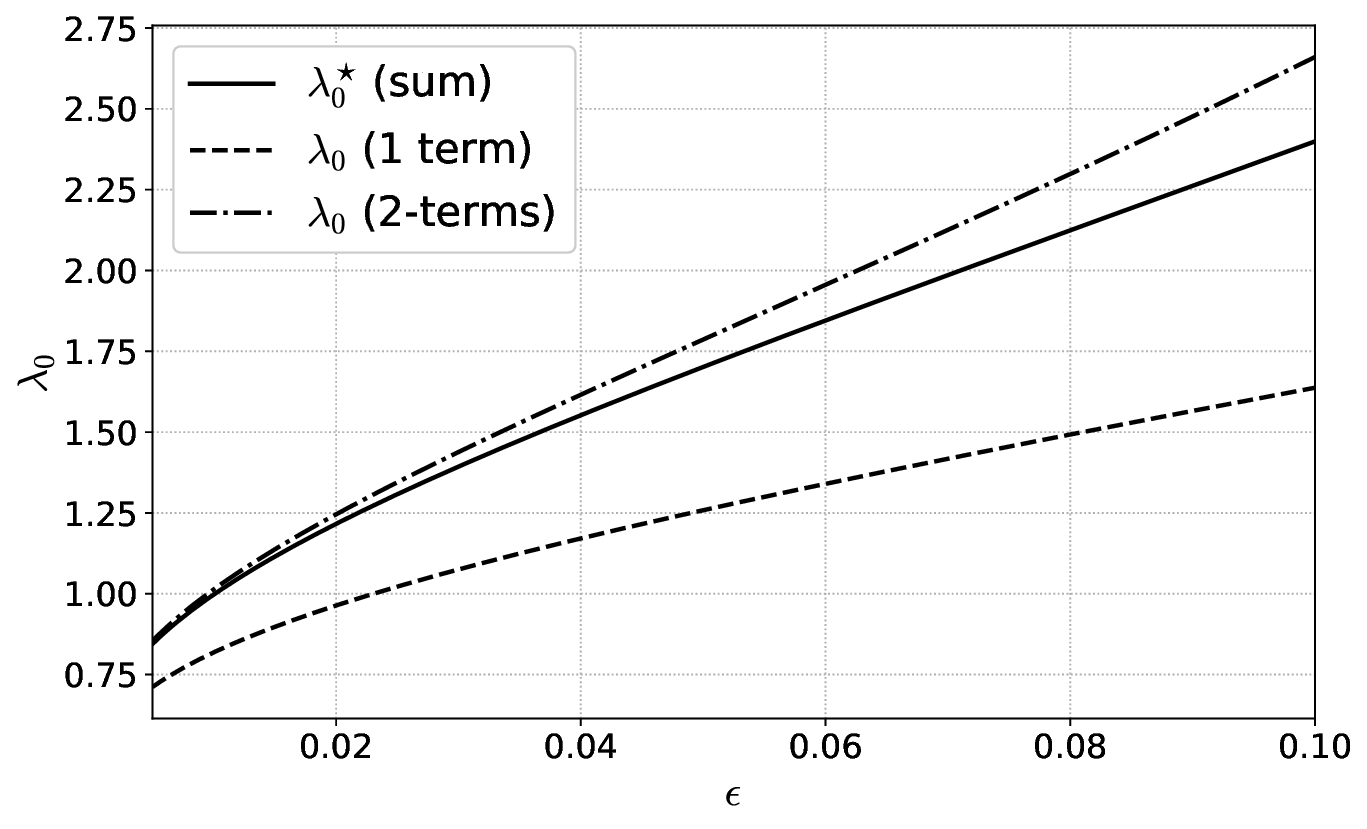}
      \caption{$\lambda_0$ vs. $\eps$}
      \label{fig:eig_comp}
    \end{subfigure}
      \caption{Approximation of the lowest eigenvalue $\lambda_0$ of
        the Laplacian in the unit sphere for a perfectly absorbing
        equatorial filament with ring radius $r_0=0.4$ and
        cross-sectional radius $\eps$.  Left: ${\mathcal
          N}(\lambda_0)$, defined in \eqref{eig:nfunc} versus
        $\lambda_0$ for $\eps=0.02$ showing the existence of a root
        $\lambda_{0}^{\star}$ for the log-summed approximation. Right:
        Comparison of the one-term $\lambda_0\sim 3\pi r_0\nu$ and
        two-term approximations from \eqref{eig:12term} with the
        log-summed approximation, computed from the zero of ${\mathcal
          N}(\lambda_0)$ in \eqref{eig:nfunc}, as $\eps$ is varied.}
    \label{fig:eig}
\end{figure}

\section{Fragmenting a filament into $N$ uniformly-spaced small
  targets}\label{sec6:ball-ring}

We revisit the MFRT problem \eqref{intro:mfrt} for a perfectly
absorbing filament $\Gam_\eps$ of circular cross-sectional radius
$\eps$ and of length $L$ satisfying
\begin{equation}\label{sec6:MFPT}
    \Delta_{\x} u=-\frac{1}{D} \,, \quad \x\in \Om\setminus\Gam_\eps\,; \qquad
     u=0 \,, \quad \x \in \partial\Gam_\eps\,; \qquad
\partial_n u=0\, \quad \x \in \partial\Om \,.
\end{equation}
Suppose that we fragment the filament $\Gamma_{\eps}$, replacing it by
a string of $N$ identical perfectly absorbing small targets 
$\Omega_{\sigma,j}$ of diameter ${\mathcal O}(\sigma)$, referred to
as {\em microtraps}, that are uniformly spaced in arclength along the
filament centerline $\Gam$.
Examples of such fragmentation for a straight filament, an
equatorial ring filament, and a Trefoil knot filament are shown in
Fig.~\ref{fig:all_frag} (recall also Fig.~\ref{fig:sphere_tube_balls}
of \S \ref{sec2:conc_ring}).  The rescaled canonical shape of a 3-D
microtrap is labeled by $\Omega_0 \equiv
\Omega_{\sigma,j}/\sigma$ (for example, the unit sphere, or a flat
disk in 3-D) having capacitance $C$.
We place $N$ microtraps $\Omega_{\sigma,j}$, with $j =
1,\ldots,N$, along the centerline $\Gam$ of the filament
$\Gam_{\varepsilon}$, uniformly spaced and with their orientation aligned with the tangent of the centerline. The absorbing set, resulting from
the union of these $N$ microtraps, is denoted by $F_{\sigma}$, which we shall refer to as the {\em fragmented
filament.}

Our goal is to establish the relation between $\eps$ and $\sigma$ for
which the global MFRT for a filament of circular cross-sectional
radius $\eps$ is asymptotically the same as that for the fragmented
filament $F_{\sigma}$. In particular, we will show that the relation
\eqref{sec3:eps_to_sigma} derived for the specific case of the
fragmentation of an equatorial filament in the unit sphere applies for
an arbitrary filament in $\Omega$, but only in the {\em transition
scaling regime} for which ${\mathcal O}(\sigma)\ll {L/N} \ll {\mathcal
O}(1)$. Since for ${L/N}\ll 1$ the fragmentation of a filament is only
felt locally, the filament centerline is nearly straight between
adjacent microtraps. The other requirement that ${L/N}\gg {\mathcal
O}(\sigma)$ implies that adjacent microtraps are still separated by
distances asymptotically larger than their ${\mathcal O}(\sigma)$
diameters. Owing to the uniform
spacing of microtraps on the filament centerline, the considerations
above suggest we need only consider a new inner problem on a ${L/N}\ll
1$ length scale that contains a single microtrap on a smaller
${\mathcal O}(\sigma)$ length-scale but that has a locally straight
centerline. This solution can then be extended periodically down the
filament. The analysis of this new inner problem shows that the effect
of fragmentation in the scaling regime ${\mathcal O}(\sigma)\lesssim {L/N}
\ll {\mathcal O}(1)$ is only to change the logarithmic capacitance of
a filament. More specifically, our analysis shows that the MFRT for
the fragmented filament is asymptotically equivalent to that for a
full unfragmented filament,  when the cross-sectional radius
$\eps$ is appropriately related to $N$, $\sigma$, and the capacitance
$C$ of the microtrap.  As a result, to determine the MFRT for a
fragmented filament we need only redefine $\eps$ in
$\nu={-1/\log\eps}$ in the Fredholm integral equation derived in
\eqref{prop:inteq} of Proposition \ref{prop:int_eq}.  Our main result,
which we derive systematically in Appendix \ref{app:slab}, is as
follows:

\begin{proposition}\label{mr:main}
Let $\bar{u}_{\Gam_{\eps}}$ and $\bar{u}_{F_{\sigma}}$ be the global
MFRTs for the perfectly absorbing filament $\Gam_{\eps}$ with a
circular cross-section of radius $\eps$ and for its fragmented version
$F_{\sigma}$, respectively.  Denote by $2a\equiv {L/(N\sigma)}$ the
spacing between adjacent microtraps of $F_{\sigma}$, scaled
relative to the common size $\sigma$ of each microtrap. In the limit
$N/L\gg 1$,  $\sigma \ll 1$ we have
$\bar{u}_{\Gam_{\eps}}\sim \bar{u}_{F_{\sigma }}$ when
\begin{equation}\label{sec6:eps-omega}
\eps =\sigma \exp \left( -\omega \left( a\right) \right) \,.
\end{equation}
Here $\omega=\omega(a)$, referred to as the infinite slab capacitance,
is defined in terms of the far-field behavior of the solution to the
new canonical inner problem
\begin{equation}\label{sec6:new_inner}
\begin{split}
  \Delta V_c & =0\,, \quad \left\vert Z\right\vert <a\,, \quad
  (X,Y,Z)\in \R^3 \setminus\Omega_0\,, \\ 
  V_c  & =0\,, \quad  (X,Y,Z)\in \partial \Omega_0\,; \qquad
 V_c(X,Y,a) =V_c(X,Y,-a)\,, \quad V_{cZ}(X,Y,a)=V_{cZ}(X,Y,-a)\,,\\
 V_c & \sim \log(\rho) +\omega \,, \quad \mbox{as} \quad \rho \equiv
 \sqrt{X^{2}+Y^{2}}\to \infty \,,
 \end{split}
\end{equation}
where $\Omega_0$ is centered at the origin $(X,Y,Z)=(0,0,0)$.  For
$a={\mathcal O}(1)$, corresponding to $N/L={\mathcal O}(\sigma^{-1})$,
we need to determine $\omega=\omega(a)$ numerically. However, in the
transition scaling regime for which $a\gg 1$, we have
\begin{equation}\label{sec6:a_large}
  \omega \sim \frac{a}{C}-\log (4a)+\gamma _{e}\,,
\end{equation}
where $\gamma_e$ is Euler's constant, so that
\begin{equation}
\eps \sim \frac{2L}{N}\exp \left( -\frac{L}{2N\sigma C}-\gamma
_{e}\right), \quad \mbox{for} \quad \frac{L}{N\sigma} \gg 1\,.  \label{sec6:eps-sig}
\end{equation}%
Here $C$ is the capacitance of the microtrap $\Omega_0$ defined by the
local problem \eqref{ball:capac}. Two special cases are $C=1$ if
$\Omega_0$ is the unit ball, and $C={2/\pi}$ when $\Omega_0$ is a flat
unit disk.
\end{proposition}

This new canonical inner problem for $V_c$ for a fragmented filament
with perfectly absorbing microtraps replaces the inner problem
\eqref{sec2:in_prob} for a full filament of cross-sectional radius
$\eps$. Recall from \eqref{prop:2term_chi} that, to leading-order, we
have $\bar{u}\sim \frac{|\Omega |}{2\pi DL}\log\left({1/\eps}\right)$
for a perfectly absorbing filament of circular cross-sectional radius
$\eps$. By using \eqref{sec6:eps-sig}, in the
limit $\sigma \rightarrow 0$ we get to leading-order that $-\log\eps
\sim {L/(2N\sigma C)}$ so that for the fragmented filament we have
$\bar{u}\sim \frac{|\Omega | }{4\pi DN\sigma C}$. This latter
estimate agrees with the leading-order result for $N$ identical
microtraps inside a 3-D domain $\Omega$ (see
\cite{ChevWard2010}). Effectively, filament fragmentation provides a
transition between the ${\mathcal O}(-\log\eps)$ behavior in 2-D and
the ${\mathcal O}(\sigma ^{-1})$ scaling of spatially localized traps
in 3-D. We now consider the three specific examples of fragmented
filaments in Fig. \ref{fig:all_frag}.

\begin{figure}[tbp]
\centering
\setlength{\tabcolsep}{0pt} 
\begin{tabular}{ccc@{\hspace{0.04\textwidth}}cc@{\hspace{0.04\textwidth}}cc}
\multicolumn{3}{c}{\textbf{(a)} straight} & \multicolumn{2}{c}{\textbf{(b)}
ring} & \multicolumn{2}{c}{\textbf{(c)} trefoil} \\ 
\includegraphics[width=\pw]{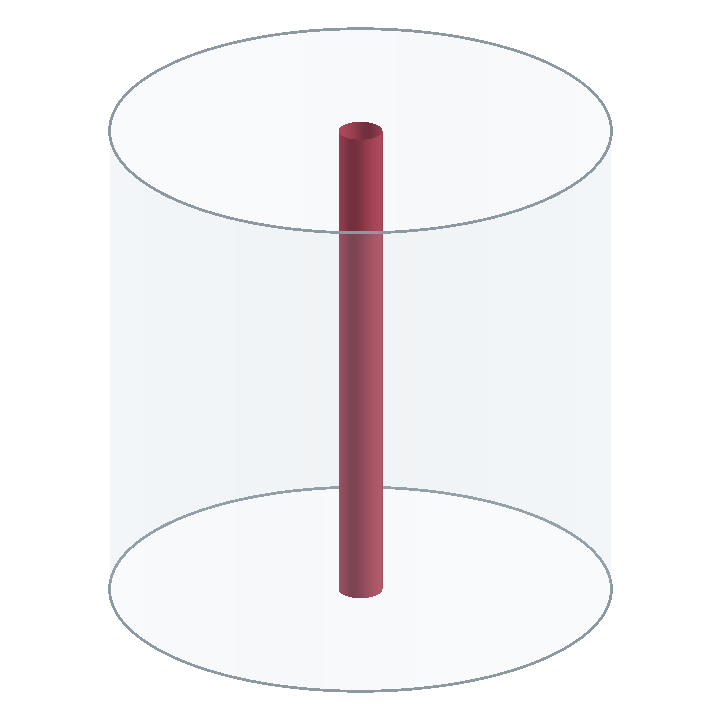} & \includegraphics[width=%
\pw]{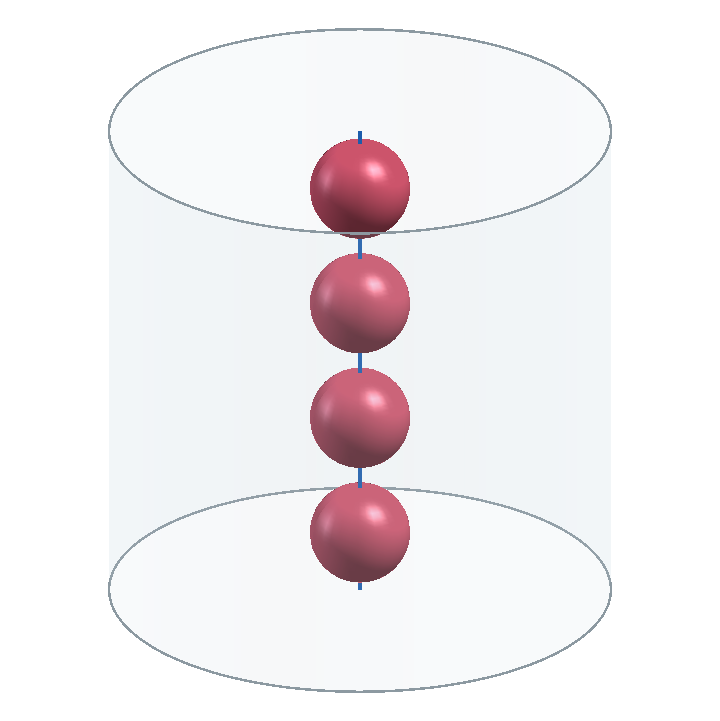} & \includegraphics[width=\pw]{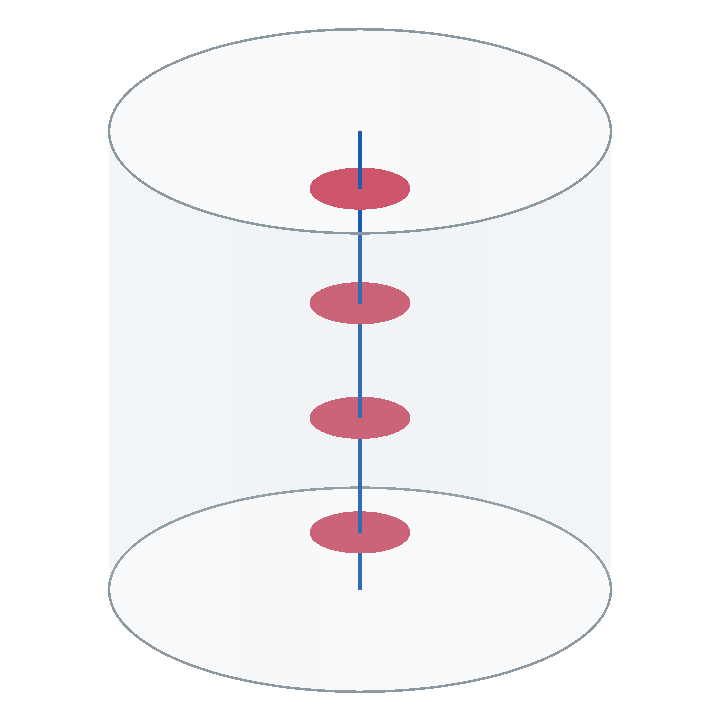} & %
\includegraphics[width=\pw]{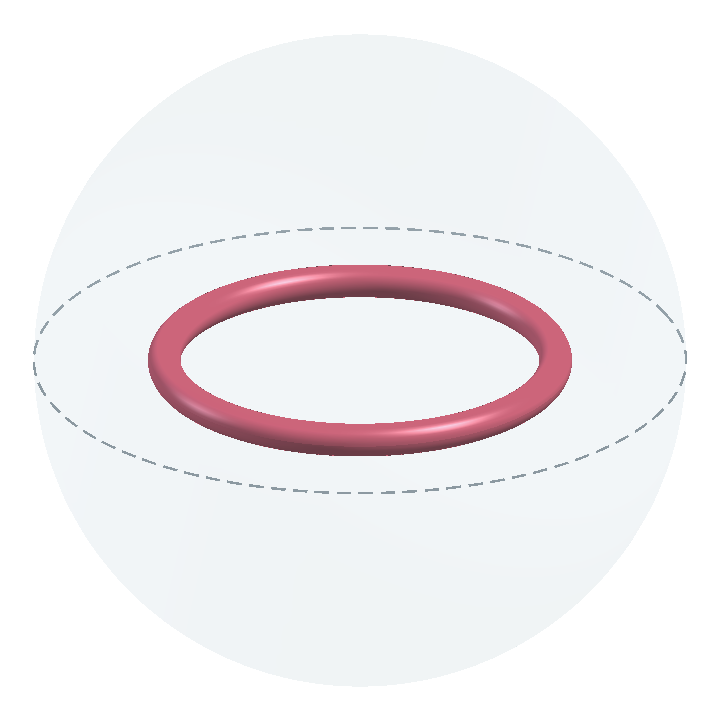} & \includegraphics[width=%
\pw]{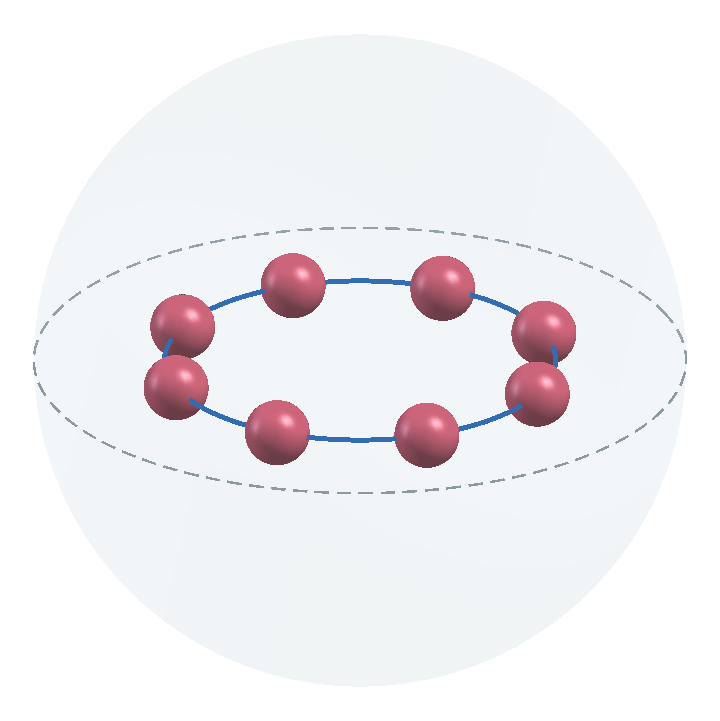} & \includegraphics[width=\pw]{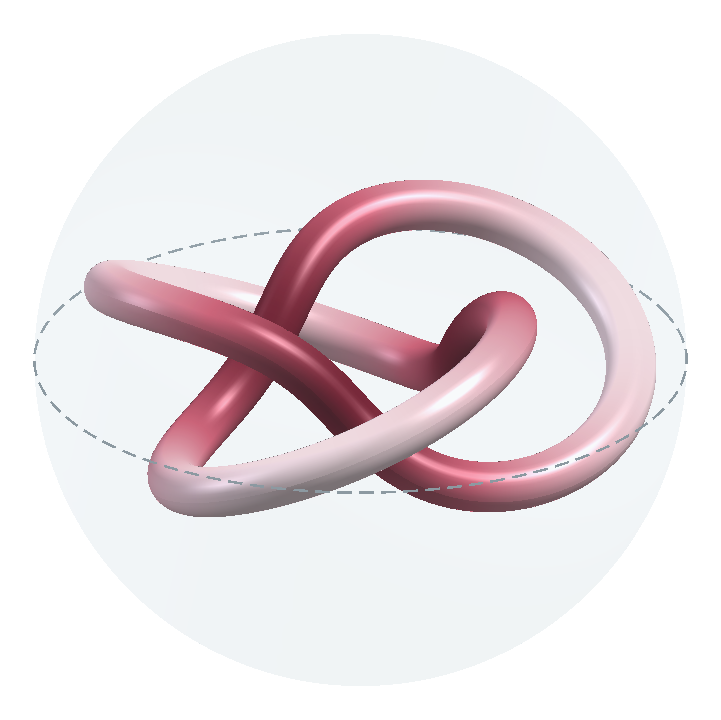} & %
\includegraphics[width=\pw]{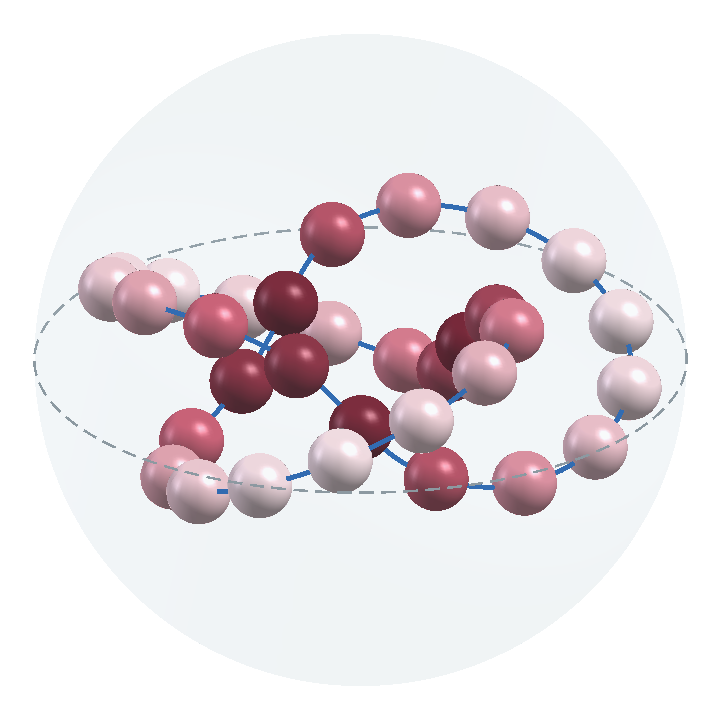} \\ 
{\footnotesize $\varepsilon =0.0878$} & {\footnotesize $N=4$ spheres} & 
{\footnotesize $N=4$ disks} & {\footnotesize $\varepsilon =0.0502$} & 
{\footnotesize $N=8$ spheres} & {\footnotesize $\varepsilon =0.0775$} & 
{\footnotesize $N=32$ spheres} \\[2pt] 
&  &  &  &  & \multicolumn{2}{c}{\includegraphics[width=0.26%
\textwidth]{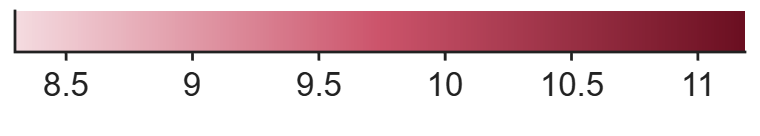}} \\ 
&  &  &  &  & \multicolumn{2}{c}{\scriptsize self-energy $E(s)$}%
\end{tabular}%
\caption{Filaments vs. fragmented filaments. (a) Straight cylinder of
   height $L=2$ and unit radius with a thin filament centered on the 
   vertical axis that is  fragmented into four small spheres or disks of 
   radius $\sigma = 0.2$. The global MFRT \eqref{frag-cyl:mfpt} for a 
   circular filament of cross-sectional radius $\eps =0.0878$ is 
   asymptotically equivalent to that of the four
  disks. (b) Equatorial filament of radius $\eps =0.0502$ and its
  equivalent fragmentation into 8 spherical microtraps of radius
  $\sigma =0.1.$ (c)\ Trefoil knot \eqref{eqn:wiggly_curve_conv} with
  $p=3$, $q=2$, $r_{0}=0.6$ and $\rho _{0}=0.25$ and its equivalent
  fragmentation into 32 spherical microtraps of radius
  $\sigma=0.1$. Color-coded by the intrinsic self-energy ${\mathcal
  E}(s)$ of the filament centerline from \eqref{sec2:intrinsic}. }
\label{fig:all_frag}
\end{figure}

\textbf{Straight vertical filament inside a cylinder.} Let $\Omega$
be a cylinder of height $L$ and radius 1 that contains a straight
vertical filament $\Gamma _{\eps}$ with circular cross-sectional
radius $\eps$ that is centered on the axis of symmetry (see
Fig. \ref{fig:all_frag}(a) (left)).  Owing to radial symmetry, an
exact result for the global MFRT for the filament is available, and we
calculate from the solution to $(r u_r)_r=-r/D$ with $u_r(1)=0$ and
$u(\eps)=0$ that
\begin{equation}\label{frag-cyl:mfpt}
\bar{u}_{\Gamma _{\varepsilon }}=\frac{1}{D} \left[\frac{\log (1/\varepsilon )}{%
2(1-\varepsilon ^{2})}-\frac{3}{8}+\frac{\varepsilon ^{2}}{8}\right] \,.
\end{equation}

To test our main result in Proposition \ref{mr:main} when $D=1$, we
take $L=2$ and fragment $\Gam_{\eps}$ into $N$ small spheres of radius
$\sigma$, uniformly spaced as shown in
Fig. \ref{fig:all_frag}(a) (middle). Since $C=1$ we predict from
\eqref{sec6:eps-sig} that $\eps \sim \frac{4}{N}\exp \left(
-\tfrac{1}{N\sigma }-\gamma _{e}\right)$ when $\sigma\ll {L/N}\ll 1$.
To compute full numerical solutions, we observe by symmetry that this
problem is equivalent to the MFRT problem for a single sphere at the
center of a cylinder of height $L/N$. The resulting 2-D problem in
$(r,z)$ is readily solved numerically to high accuracy (see Appendix
\ref{app:slab_num}). A comparison between the asymptotic predictions
and the full numerical results is shown in Table \ref{tab:cyl}.

\begin{center}
\begin{tabular}{r@{\hspace{12pt}}c@{\hspace{20pt}}c@{\hspace{20pt}}c@{\hspace{12pt}}c@{\hspace{12pt}}c@{\hspace{12pt}}c@{\hspace{12pt}}c@{\hspace{12pt}}c}
\toprule$N$ & $a$ & $\eps_{\text{asy}}$ & $\eps_{\omega}$ & $%
\bar{u}_{F_{\sigma },\, \text{num}}$ & $\bar{u}_{\Gam_{\eps}}(\eps_{\text{asy}})$ & $\bar{u}_{\Gam_{\eps}}(\eps_{\omega})$ & err$_{\text{asy}}$ & err$_{\omega}$ \\ 
\multicolumn{9}{l}{$\sigma =0.2$} \\ 
1 & 5.000 & $0.01513$ & $0.01513$ & 1.75068 & 1.72097 & 1.72095 & $+1.73\%$
& $+1.73\%$ \\ 
2 & 2.500 & $0.09217$ & $0.09226$ & 0.83906 & 0.82831 & 0.82787 & $+1.30\%$
& $+1.35\%$ \\ 
3 & 1.667 & $0.1414$ & $0.1424$ & 0.62373 & 0.62555 & 0.62237 & $-0.29\%$ & $%
+0.22\%$ \\ 
4 & 1.250 & $0.1609$ & $0.1652$ & 0.55216 & 0.56611 & 0.55402 & $-2.46\%$ & $%
-0.34\%$ \\ 
\multicolumn{9}{l}{$\sigma =0.1$} \\ 
1 & 10.000 & $1.020\times 10^{-4}$ & $1.020\times 10^{-4}$ & 4.23418 & 
4.22046 & 4.22046 & $+0.33\%$ & $+0.33\%$ \\ 
2 & 5.000 & $0.007566$ & $0.007566$ & 2.07314 & 2.06718 & 2.06717 & $+0.29\%$
& $+0.29\%$ \\ 
4 & 2.500 & $0.04609$ & $0.04613$ & 1.16966 & 1.16715 & 1.16669 & $+0.22\%$
& $+0.25\%$ \\ 
8 & 1.250 & $0.08043$ & $0.08259$ & 0.88040 & 0.89420 & 0.88135 & $-1.54\%$
& $-0.11\%$ \\ 
9 & 1.111 & $0.08215$ & $0.08582$ & 0.86186 & 0.88396 & 0.86281 & $-2.50\%$
& $-0.11\%$ \\ 
\multicolumn{9}{l}{$\sigma =0.05$} \\ 
1 & 20.000 & $4.629\times 10^{-9}$ & $4.629\times 10^{-9}$ & 9.22963 & 
9.22046 & 9.22046 & $+0.10\%$ & $+0.10\%$ \\ 
2 & 10.000 & $5.098\times 10^{-5}$ & $5.098\times 10^{-5}$ & 4.56864 & 
4.56703 & 4.56703 & $+0.04\%$ & $+0.04\%$ \\ 
4 & 5.000 & $0.003783$ & $0.003783$ & 2.41523 & 2.41365 & 2.41364 & $+0.07\%$
& $+0.07\%$ \\ 
8 & 2.500 & $0.02304$ & $0.02306$ & 1.51158 & 1.51125 & 1.51079 & $+0.02\%$
& $+0.05\%$ \\ 
16 & 1.250 & $0.04022$ & $0.04130$ & 1.22108 & 1.23456 & 1.22144 & $-1.09\%$
& $-0.03\%$ \\ 
17 & 1.176 & $0.04074$ & $0.04217$ & 1.21074 & 1.22817 & 1.21110 & $-1.42\%$
& $-0.03\%$ \\ 
18 & 1.111 & $0.04107$ & $0.04291$ & 1.20213 & 1.22411 & 1.20248 & $-1.80\%$
& $-0.03\%$ \\ 
19 & 1.053 & $0.04125$ & $0.04354$ & 1.19487 & 1.22193 & 1.19520 & $-2.21\%$
& $-0.03\%$ \\ 
\bottomrule &  &  &  &  &  &  &  & 
\end{tabular}
\captionof{table}{MFRT results when $D=1$ for the straight filament in
  the cylinder ($L=2$) versus that for the fragmented filament with
  $N$ small spheres of radius $\sigma $ on its axis of
  symmetry.} \label{tab:cyl}
\end{center}

In Table \ref{tab:cyl}, $a=1/(N\sigma )$ is half the distance between
the centers of adjacent spheres, rescaled by their common radius
$\sigma$, while $\eps_{\text{asy}}$ and $\eps_{\omega}$ are given by
\eqref{sec6:eps-sig} and \eqref{sec6:eps-omega}, respectively. To
determine $\eps_{\omega}$ we solved the inner problem
\eqref{sec6:new_inner} numerically to obtain $\omega(a)$.  In
addition, we have defined err$_{\text{asy}}\equiv\bar{u}_{F_{\sigma
}}/\bar{u}_{\Gam_{\eps},\text{asy}}-1$ and err$_{\omega}\equiv
\bar{u}_{F_{\sigma }}/\bar{u}_{\Gam_{\eps}, \omega}-1$. As expected,
we observe that err$_{\omega}$ is always very small, whereas
err$_{\text{asy}}$ is the largest when the spheres are nearly
touching, for which $a$ is not large (last row
for each $\sigma $).

We now repeat this computation replacing the small spheres by flat
disks as shown in Fig.~\ref{fig:all_frag}a (right). In this case
$C={2/\pi}$ so that from \eqref{sec6:eps-sig} we now have
$\eps_{\text{asy}}=\frac{4}{N}\exp \left( -\frac{\pi }{2N\sigma }%
-\gamma _{e}\right)$.  Unlike for the case of small spheres, the disks
never touch, so that $N$ can be taken arbitrarily large. A favorable
comparison between the asymptotic and  numerical results are shown in
Table \ref{tab:cyldisk}.

\begin{center}
\begin{tabular}{r@{\hspace{12pt}}c@{\hspace{20pt}}c@{\hspace{20pt}}c@{\hspace{12pt}}c@{\hspace{12pt}}c@{\hspace{12pt}}c@{\hspace{12pt}}c@{\hspace{12pt}}c}
\toprule$N$ & $a$ & $\varepsilon _{\text{asy}}$ & $\varepsilon _{\omega }$ & $%
\bar{u}_{F_{\sigma },\, \text{num}}$ & $\bar{u}_{\Gamma _{\varepsilon }%
\text{,asy}}$ & $\bar{u}_{\Gamma _{\varepsilon }}(\varepsilon _{\omega })$ & 
err$_{\text{asy}}$ & err$_{\omega}$ \\ 
\multicolumn{9}{l}{$\sigma =0.2$} \\ 
1 & 5.000 & $8.718\times 10^{-4}$ & $8.788\times 10^{-4}$ & 3.16545 & 3.14745
& 3.14350 & $+0.57\%$ & $+0.70\%$ \\ 
2 & 2.500 & $0.02212$ & $0.02281$ & 1.52867 & 1.53152 & 1.51635 & $-0.19\%$
& $+0.81\%$ \\ 
4 & 1.250 & $0.07881$ & $0.08776$ & 0.85510 & 0.90407 & 0.85196 & $-5.42\%$
& $+0.37\%$ \\ 
8 & 0.625 & $0.1052$ & $0.1438$ & 0.60892 & 0.76503 & 0.61763 & $-20.41\%$ & 
$-1.41\%$ \\ 
\multicolumn{9}{l}{$\sigma =0.1$} \\ 
1 & 10.000 & $3.385\times 10^{-7}$ & $3.391\times 10^{-7}$ & 7.08459 & 
7.07444 & 7.07344 & $+0.14\%$ & $+0.16\%$ \\ 
2 & 5.000 & $4.359\times 10^{-4}$ & $4.394\times 10^{-4}$ & 3.49342 & 3.49403
& 3.49007 & $-0.02\%$ & $+0.10\%$ \\ 
4 & 2.500 & $0.01106$ & $0.01140$ & 1.86521 & 1.87739 & 1.86218 & $-0.65\%$
& $+0.16\%$ \\ 
8 & 1.250 & $0.03941$ & $0.04388$ & 1.19152 & 1.24464 & 1.19139 & $-4.27\%$
& $+0.01\%$ \\ 
16 & 0.625 & $0.05259$ & $0.07191$ & 0.94474 & 1.10206 & 0.94863 & $-14.27\%$
& $-0.41\%$ \\ 
\multicolumn{9}{l}{$\sigma =0.05$} \\ 
1 & 20.000 & $5.101\times 10^{-14}$ & $5.103\times 10^{-14}$ & 14.93663 & 
14.92842 & 14.92817 & $+0.05\%$ & $+0.06\%$ \\ 
2 & 10.000 & $1.692\times 10^{-7}$ & $1.696\times 10^{-7}$ & 7.42086 & 
7.42102 & 7.42002 & $+0.00\%$ & $+0.01\%$ \\ 
4 & 5.000 & $2.180\times 10^{-4}$ & $2.197\times 10^{-4}$ & 3.83748 & 3.84060
& 3.83665 & $-0.08\%$ & $+0.02\%$ \\ 
8 & 2.500 & $0.005531$ & $0.005702$ & 2.20928 & 2.22376 & 2.20853 & $-0.65\%$
& $+0.03\%$ \\ 
16 & 1.250 & $0.01970$ & $0.02194$ & 1.53555 & 1.58931 & 1.53568 & $-3.38\%$
& $-0.01\%$ \\ 
32 & 0.625 & $0.02629$ & $0.03596$ & 1.28862 & 1.44555 & 1.29003 & $-10.86\%$
& $-0.11\%$ \\ 
\bottomrule &  &  &  &  &  &  &  & 
\end{tabular}
\captionof{table}{MFRT results when $D=1$ for the straight filament in the
  cylinder ($L=2$) versus that for the fragmented filament with $N$ flat
  disks of radius $\sigma$ perpendicular to its  axis of
  symmetry.} \label{tab:cyldisk}
\end{center}

From Table \ref{tab:cyldisk}, we observe, as expected, that
err$_{\text{asy}}$ is the largest for tightly packed disks where $a$
is not large and where the ratio $\varepsilon _{\omega}/\varepsilon
_{\text{asy}}$ is the largest.  Unlike for small spheres, the flat
disks can be taken arbitrary close, which is when that ratio grows the
most, and the effect is more pronounced than for spheres (e.g.
err$_{\text{asy}}=20\%$ for $N=8,\sigma =0.2)$. In contrast,
err$_{\omega}$, in which $\omega(a)$ must be computed numerically from
the new inner problem \eqref{sec6:new_inner}, always closely agrees
with the full numerics.

\textbf{Ring of Small Spheres.} Next, we consider an equatorial filament of
radius $r_{0}=0.6$ centered on the $z=0$ plane in the unit sphere and
we fragment it into $N$ small spheres of radius $\sigma=0.1$ as shown
in Fig.~\ref{fig:all_frag}b. To test our asymptotic theory, we take
$D=1$, $r_{0}=0.6$ and $\sigma =0.1.$ For each $N$ we use $\eps$ from
the large-$a$ asymptotics \eqref{sec6:eps-sig}, and estimate
$\bar{u}_{\Gamma _{\varepsilon}}$ both from the thin-filament
asymptotics \eqref{sec3:ubar} (with $d_c=1$) and from the numerical
solution of \eqref{sec6:MFPT}, which we label by $\bar{u}_{\Gamma
  _{\varepsilon}, \text{asy}}$ and $\bar{u}_{\Gamma _{\varepsilon
  }, \text{num}}$, respectively.  In Table \ref{tab:ring} we
compare these results with the corresponding full numerical results
$\bar{u}_{F_{\sigma}, \text{num}}$ for the fragmented filament
computed using the methodology in Appendix \ref{app:slab_num}.  In this
Table, $\omega_{\text{asy}}$ is given by \eqref{sec6:a_large} while
$\omega(a)$ is computed from the new inner problem
\eqref{sec6:new_inner} numerically using Appendix \ref{app:slab_num}.

\begin{center}
\begin{tabular}{r@{\hspace{8pt}}c@{\hspace{10pt}}c@{\hspace{10pt}}c@{\hspace{10pt}}c@{\hspace{10pt}}c@{\hspace{8pt}}c@{\hspace{8pt}}c@{\hspace{8pt}}c@{\hspace{4pt}}c@{\hspace{4pt}}c}
\toprule$N$ & $a$ & $\omega_{\text{asy}}$ & $\omega(a)$ & $\varepsilon _{%
\text{asy}}$ & $\bar{u}_{\Gamma _{\varepsilon },\, \text{asy}}$ & $\bar{u}%
_{\Gamma _{\varepsilon },\, \text{num}}$ & $\bar{u}_{F_{\sigma },\, \text{%
num}}$ & $\frac{\bar{u}_{\Gamma _{\varepsilon },\, \text{num}}}{\bar{u}%
_{\Gamma _{\varepsilon },\, \text{asy}}}-1$ & $\frac{\bar{u}_{F_{\sigma },\, 
\text{num}}}{\bar{u}_{\Gamma _{\varepsilon },\, \text{asy}}}-1$ & $\frac{%
\bar{u}_{F_{\sigma },\, \text{num}}}{\bar{u}_{\Gamma _{\varepsilon },\, 
\text{num}}}-1$ \\ 
\midrule5 & 3.770 & 1.6338 & 1.6337 & $0.01952$ & 0.51077 & 0.51115 & 0.51697
& $+0.07\%$ & $+1.21\%$ & $+1.14\%$ \\ 
10 & 1.885 & 0.4420 & 0.4382 & $0.06428$ & 0.30001 & 0.30430 & 0.30538 & $%
+1.43\%$ & $+1.79\%$ & $+0.36\%$ \\ 
15 & 1.257 & 0.2191 & 0.1932 & $0.08032$ & 0.26060 & 0.26712 & 0.26288 & $%
+2.50\%$ & $+0.87\%$ & $-1.59\%$ \\ 
16 & 1.178 & 0.2051 & 0.1708 & $0.08146$ & 0.25813 & 0.26481 & 0.25911 & $%
+2.59\%$ & $+0.38\%$ & $-2.15\%$ \\ 
17 & 1.109 & 0.1964 & 0.1524 & $0.08216$ & 0.25659 & 0.26338 & 0.25604 & $%
+2.65\%$ & $-0.21\%$ & $-2.79\%$ \\ 
18 & 1.047 & 0.1920 & 0.1369 & $0.08253$ & 0.25581 & 0.26265 & 0.25352 & $%
+2.68\%$ & $-0.89\%$ & $-3.48\%$ \\ 
\bottomrule &  &  &  &  &  &  &  &  &  & 
\end{tabular}
\captionof{table}{MFRT results when $D=1$ for the fragmentation of an
  equatorial filament of radius $r_0=0.6$ into $N$ small spheres of
  radius $\sigma =0.1$.} \label{tab:ring}
\end{center}

From Table \ref{tab:ring} we observe that the largest error of about
$3.5\%$ occurs when $a=1.047$, when the spheres almost touch each other
and the large-$a$ asymptotics \eqref{sec6:eps-sig} no longer
applies. We remark that upon replacing $\varepsilon _{\text{asy}}$ by
$\varepsilon _{\omega\left( a\right) }$ in this case reduces the error
to only $-0.05\%.$ We have also verified that $\bar{u} _{F_{\sigma
  }, \text{num}}$ can be in good agreement with the asymptotic
approximation $\bar{u}_{F_{\sigma }, \text{asy}}$ obtained by using
the Green's function for the sphere (see \eqref{ball:finalN} of
Proposition \ref{ball:prop}). For example, with $N=15$ spheres,
$\bar{u}_{F_{\sigma }, \text{asy}}=0.26050$, which is less than a
$1\%$ deviation from $\bar{u}_{F_{\sigma },\ \text{num}}$.

\textbf{Fragmented trefoil.} As a more challenging non-symmetric test
of our main result in Proposition \ref{mr:main}, we fragment the
Trefoil knot, with centerline defined by \eqref{eqn:wiggly_curve_conv}
with $p=3$, $q=2$, $r_{0}=0.6$ and $\rho _{0}=0.25$ as shown in
Fig. \ref{fig:all_frag}c (left). The filament centerline, of length
$L=8.974$, is fragmented into $N$ equally-spaced small spheres of
radius $\sigma =0.1,$ with the center of the first sphere at
$t=0$. When $D=1$, we tabulate our results for the global MFRT at
several values of $N$ in Table \ref{tab:trefoil}, following the
notation of Table \ref{tab:ring}. In this Table, $\bar{u}_{\Gamma
_{\varepsilon}, \text{asy}}$ is computed from the numerical solution
of the integral equation \eqref{prop:inteq}.

\begin{center}
\begin{tabular}{r@{\hspace{8pt}}c@{\hspace{10pt}}c@{\hspace{10pt}}c@{\hspace{10pt}}c@{\hspace{10pt}}c@{\hspace{8pt}}c@{\hspace{8pt}}c@{\hspace{8pt}}c@{\hspace{4pt}}c@{\hspace{4pt}}c}
\toprule$N$ & $a$ & $\omega_{\text{asy}}$ & $\omega(a)$ & $\varepsilon _{%
\text{asy}}$ & $\bar{u}_{\Gamma _{\varepsilon },\, \text{asy}}$ & $\bar{u}%
_{\Gamma _{\varepsilon },\, \text{num}}$ & $\bar{u}_{F_{\sigma },\, \text{%
num}}$ & $\frac{\bar{u}_{\Gamma _{\varepsilon },\, \text{num}}}{\bar{u}%
_{\Gamma _{\varepsilon },\, \text{asy}}}-1$ & $\frac{\bar{u}_{F_{\sigma },\,
\text{num}}}{\bar{u}_{\Gamma _{\varepsilon },\, \text{asy}}}-1$ & $\frac{%
\bar{u}_{F_{\sigma },\, \text{num}}}{\bar{u}_{\Gamma _{\varepsilon },\,
\text{num}}}-1$ \\
\midrule8 & 5.609 & 3.0753 & 3.0753 & $0.004617$ & 0.32252 & 0.32255 &
0.34297 & $+0.01\%$ & $+6.34\%$ & $+6.33\%$ \\
16 & 2.804 & 0.9641 & 0.9636 & $0.03813$ & 0.16472 & 0.16610 & 0.17049 & $%
+0.84\%$ & $+3.50\%$ & $+2.64\%$ \\
32 & 1.402 & 0.2551 & 0.2394 & $0.07749$ & 0.11121 & 0.11592 & 0.11493 & $%
+4.24\%$ & $+3.35\%$ & $-0.85\%$ \\
39 & 1.151 & 0.2012 & 0.1633 & $0.08177$ & 0.10711 & 0.11227 & 0.10946 & $%
+4.82\%$ & $+2.20\%$ & $-2.50\%$ \\
40 & 1.122 & 0.1978 & 0.1557 & $0.08205$ & 0.10684 & 0.11203 & 0.10894 & $%
+4.86\%$ & $+1.96\%$ & $-2.76\%$ \\
41 & 1.094 & 0.1951 & 0.1487 & $0.08227$ & 0.10664 & 0.11185 & 0.10847 & $%
+4.89\%$ & $+1.71\%$ & $-3.03\%$ \\
42 & 1.068 & 0.1932 & 0.1421 & $0.08244$ & 0.10649 & 0.11172 & 0.10803 & $%
+4.91\%$ & $+1.45\%$ & $-3.30\%$ \\
\bottomrule &  &  &  &  &  &  &  &  &  &
\end{tabular}
\captionof{table}{MFRT results when $D=1$ for the fragmentation of the Trefoil
  filament \eqref{eqn:wiggly_curve_conv} with $p=3$, $q=2$, $r_0=0.6$,
  and $\rho _{0}=0.25$ into $N$ small spheres of radius $\sigma
  =0.1$.} \label{tab:trefoil}
\end{center}

In Table \ref{tab:trefoil} we observe a close agreement
between the equivalent full filament results (both asymptotics and
numerics), and the fragmented filament, even with relatively few
spheres $(N=8).$ The error begins to grow for small trap
separation. For $N=42$ spheres, $a=1.068,$ and the asymptotic value 
$\omega \sim a-\log (4a)+\gamma _{e}=0.1932$ deviates significantly
from an exact value $\omega\left( 1.068\right) =0.1421$ computed from
a numerical solution of the new inner problem
\eqref{sec6:new_inner}. The latter value for $\omega$ changes the
equivalent filament cross-sectional radius to $\varepsilon =0.08675$,
and the error between the full computations drops from 3.30\% to only
0.25\%.

{\bf Note on numerics.} 
 We initially used the built-in MATLAB FEM solver, with the interface to it written by Claude AI (Opus version 5.5) for numerical work in this section. These numerics were then independently cross-validated by Claude's own numerical solvers that it custom-built for each of the problems, and for which it wrote its own 3D numerics without using MATLAB's FEM solver. 
 Claude's own solver was able to achieve a better convergence and accuracy than the standard MATLAB FEM solver running on a 16gig laptop; Tables \ref{tab:cyl} to \ref{tab:trefoil} therefore show the more precise values achieved by Claude's own routines. We verified that both methods (MATLAB FEM vs. Claude) agreed to within reported accuracy of MATLAB's FEM computation.
 For example, for the $N=40$ trefoil fully 3D numerical computation, Claude was able to compute up to 4 significant digits using its own 3D code (included in Table \ref{tab:trefoil}), while the MATLAB FEM package only managed up to 2 significant digits before running out of memory.

\section{Discussion}\label{sec7:discussion}

We have introduced a new asymptotic analysis for the mean
first-reaction time to thin filaments with reactive boundaries. Each
filament is assumed to have a centerline that is a simple closed
curve, and the collection of filaments are assumed to be
non-intersecting (with themselves and with each other). Our asymptotic
analysis reduces the full 3-D problem to one of determining the
line densities on each filament and the global MFRT in terms of the
solution to a coupled set of 1-D integral equations. Our theory
applies to filaments of general cross-sectional shape with the
geometry being encoded in the logarithmic reactive capacitance, which
can be computed numerically from a Steklov eigenfunction
representation. For various configurations of filaments, we validated
our asymptotic results using a simple trapezoidal approximation of the
integral operator. Our results confirm the expected accuracy of our
asymptotic theory as the cross-sectional radius decreases when
compared to full finite element simulations.  Extensions of the theory
to account for heterogeneous diffusivity, the splitting probabilities
and to determine the principal eigenvalue of the Laplacian were
discussed.  Finally, the relationship between the MFRT for filaments
and for the related problem of replacing the filament by spatially
localized targets equally-spaced on its centerline was analyzed.

We remark that in our initial formulation of the MFPT problem,
we have included the scaling factor $\eps$ to the Robin condition \eqref{intro:mfrt} to keep the constant reactivity parameter $b$ in the inner solution. Dividing by $\eps$, one recovers the conventional form of the Robin condition, $D\partial_n u + \K u = 0$, with the dimensional reactivity $\K_b \equiv b D/\eps$ that increases as $\eps\to 0$. In many applications, however, $\K_b$ remains fixed and does not scale with $\eps$. In this case, $b = \eps \K_b/D$ vanishes in the limit $\eps\to 0$ that significantly changes the asymptotic behavior of the MFRT. For instance, in the case of a single filament, substitution of the Taylor expansion \eqref{eq:dc_small} into \eqref{intro:prop:2term_chi} yields 
\begin{equation}
    \bar{u} \sim  \frac{|\Om|}{\K_b |\partial \Gamma_\eps|} + \frac{|\Om|}{D |\Gam|}\left(\frac{\log(1/\eps)}{|\partial\B|} + \frac{C_1}{2\pi} + \frac{1}{|\Gam|}  \int_{\Gam} \int_{\Gam}
    R\left(\X(s); \X(s^{\prime}) \right) \dd s \dd s^{\prime} +
    \frac{1}{4\pi|\Gam|} \int_{\Gam} {\mathcal E}(s) \dd s +
    {\mathcal O}\left(\frac{\eps}{\log(\eps)}\right)  \right). 
\end{equation}
The first term, which gives the leading-order contribution $1/\eps$, corresponds to the standard kinetics-limited regime; expectedly, the MFRT is inversely proportional to the surface area $|\partial\Gamma_\eps|$ of the reactive boundary and to the reactivity $\K_b$, whereas it does not depend on the diffusion coefficient $D$. In turn, the logarithmic term gives the subleading contribution. The qualitative change in the scaling of the MFRT, from $\log(1/\eps)$ to $1/\eps$, was discussed previously for circular targets in \cite{Grebenkov17} (see also \cite{GrebWard2D}). The logarithmic reactive capacitance $d_c(b)$ encompasses various scenarios (when $b$ is fixed, vanishing, or infinite) and quantitatively describes transitions between them in \eqref{intro:prop:2term_chi} and related expressions.

There are many avenues for future study which arise from this work. We
have demonstrated our theory on closed filaments for which the
trapezoidal rule provides accurate quadrature. However, typical
biological filaments and fibers are either open at both ends or
tethered at one end to the boundary and free at the other. We
hypothesize that the results contained in the log-summed integral
equation system still hold to leading order. However, the effects of
open ends will arise in higher-order corrections through an additional
inner problem centered at each tip, which needs to be
analyzed. Additional high-order effects will arise at ${\mathcal
O}(\eps)$ due to the centerline curvature as seen from the inner
problem \eqref{sec1:lapl2}. Related rigorous work on finite-end
filament conditions in slender-body theory has recently been
considered in \cite{MOS2019}. The details of obtaining these
higher-order corrections is left to future work, as well as
generalizations to include drift fields and directed transport
inside the filaments, among others.

\paragraph{Acknowledgments}

MJW and TK are grateful for the support of the NSERC Discovery Grant Program
(Canada). AEL acknowledges support from NSF grant DMS-2052636. This
research was supported in part by grants from the NSF (DMS-2235451)
and Simons Foundation (MPS-NITMB-00005320) to the NSF-Simons National
Institute for Theory and Mathematics in Biology (NITMB).
D.S.G. acknowledges the Simons Foundation for supporting his
sabbatical sojourn in 2024 at the CRM, University of Montr\'eal,
Canada, as well as the Alexander von Humboldt Foundation for support
within a Bessel Prize award. We gratefully acknowledge the support of the University of Notre Dame’s Kylemore Abbey Global Centre.

\begin{appendices}
\renewcommand{\theequation}{\Alph{section}.\arabic{equation}}
\setcounter{equation}{0}

\section{The inner problem}\label{app:inner}

We first derive \eqref{sec1:lapl1} in terms of the tubular
coordinates.  The position vector of a point $\x$ in the tubular
neighbourhood is written as $\x(s,\eta_1,\eta_2)=\X(s) + \eta_1 \MI(s)
+ \eta_2 \MII(s)$. Upon taking partial derivatives and using
$\MI^{\prime}(s)=-\kappa_1(s) \T(s)$ and
$\MII^{\prime}(s)=-\kappa_2(s) \T(s)$, where $\T(s)$ is the unit
tangent and $\kappa_1(s)$ and $\kappa_2(s)$ are curvatures along each
respective normal direction, we get $\x_s=\left(1-\kappa_1 \eta_1 -
\kappa_2 \eta_2\right)\T(s)$, $\x_{\eta_1}=\MI(s)$, and
$\x_{\eta_2}=\MII(s)$. In this way, we obtain the
scale factors for the transformation as
\begin{equation*}
  h_{s} = |\x_s|=1-\kappa_1\eta_1-\kappa_2\eta_2\,, \quad
  h_{\eta_1}=|\x_{\eta_1}|=1 \,, \quad  h_{\eta_2}=|\x_{\eta_2}| = 1 \,.
\end{equation*}
A sufficient condition to ensure that $h_{s}>0$ along the filament is that
$\kappa_1\eta_1+\kappa_2\eta_2<1$ for $0\leq s\leq L$. By using the Cauchy-Schwartz inequality, a sufficient condition for this inequality to hold is
that
$\sqrt{\eta_1^2 + \eta_2^2} < {1/\kappa(s)}$ for $0\le s\le L$, where 
$\kappa(s)\equiv \sqrt{\kappa_1^2+\kappa_2^2} = |\X^{\prime\prime}(s)|$ is the curvature of the filament centerline.

By using the general expression for the Laplacian of an arbitrary
function $\Phi$ in an orthogonal coordinate system, in which we set
$h_{\eta_1}=h_{\eta_2}=1$, we get
\begin{equation*}
  \Delta_{\x} \Phi = \frac{1}{h_s}
  \left[ \partial_{s} \left( \frac{1}{h_{s}}  \partial_{s} \Phi\right) +
    \partial_{\eta_2} \left( h_s \partial_{\eta_2} \Phi\right) +
   \partial_{\eta_1} \left( h_s \partial_{\eta_1} \Phi\right) \right]\,,
\end{equation*}
which readily yields  \eqref{sec1:lapl1}.

\subsection{The logarithmic reactive capacitance}\label{app:steklov}

The logarithmic reactive capacitance $d_c(b)$ is defined by the
far-field behavior of the solution $v_c$ of the boundary value problem
\eqref{sec2:in_prob}. When the cross-section $\B$ is the unit disk,
this problem admits an explicit solution, $v_c = \log|\y| + 1/b$, from
which $\log[d_c(b)] = -1/b$.  For other domains, the function $d_c(b)$
can be determined in terms of eigenfunctions of the exterior Steklov
problem.  In this Appendix, we briefly describe this spectral
representation by generalizing the approach from
\cite{Chaigneau24,Grebenkov25,Grebenkov25b,GrebWard2D}.

For a given bounded set $\B \subset \R^2$ with a connected Lipschitz
boundary $\pa$, we search for eigenvalues $\mu_k$ and eigenfunctions
$\Psi_k(\y)$ of the exterior Steklov problem
\begin{equation}
\Delta_{\y} \Psi_k = 0  \,; \quad \y \in \R^2\setminus \B\,; \qquad
\partial_n \Psi_k = \mu_k \Psi_k\,, \quad \y \in \pa\,; \qquad
\Psi_k = {\mathcal O}(1) \quad \mbox{as} \quad |\y|\to \infty\,.
\end{equation}
This eigenvalue problem has a discrete spectrum \cite{Bundrock25},
i.e, a countable set of eigenvalues $\mu_k$ which can be enumerated by
the index $k = 0,1,\ldots$ in an increasing order, $0 = \mu_0 < \mu_1
\leq \mu_2 \ldots \nearrow +\infty$.  In turn, the traces of eigenfunctions
$\Psi_k$ on $\pa$, $(\Psi_k)|_{\pa}$, form a complete orthonormal
basis in the space $L^2(\pa)$.  Since $v = v_c - v_{c,b=\infty}$ is a
harmonic function bounded at infinity, one can expand it over the
Steklov eigenfunctions
\begin{equation}  \label{eq:vc_Psi}
v_c(\y) = v_{c,b=\infty}(\y) + \sum\limits_{k=0}^\infty c_k \Psi_k(\y)\,,
\end{equation}
where the unknown coefficients $c_k$ are obtained from the boundary
condition as 
\begin{equation}
  0 = \partial_n v_c + b v_c = \partial_n v_{c,b=\infty} +
  \sum\limits_{k=0}^\infty c_k (\mu_k + b) \Psi_k(\y)\,, \quad \y\in \pa \,.
\end{equation} 
Multiplying this relation by $\Psi_j$ and integrating over $\pa$, we
get from the orthonormality of $\{(\Psi_k)|_{\pa}\}$ that
\begin{equation}\label{app:num}
  c_j = - \frac{1}{\mu_j + b} \bigl(\partial_n v_{c,b=\infty}\,,
  \Psi_j\bigr)_{L^2(\pa)} \,.
\end{equation}
Upon applying Green's second identity to the problems for $\Psi_k$ and
$v_{c,b=\infty}$, we can represent the scalar product in
the numerator of \eqref{app:num} as $-2\pi \Psi_j(\infty)$
so that
\begin{equation}
c_j = \frac{2\pi \Psi_j(\infty)}{\mu_j + b} \,.
\end{equation}
 Since $\Psi_k$ is harmonic and bounded at infinity, its gradient
decays as ${\mathcal O}\left(1/|\y|^2\right)$, which ensures the
applicability of the Green's second identity as $\lim_{R\to\infty}
R\log(R)\partial_{R} \Psi_k=0$, with $R=|\y|$.

The logarithmic reactive capacitance is obtained by
evaluating the limiting behavior of \eqref{eq:vc_Psi} as $\y\to\infty$:
\begin{equation*}
  -\log[d_c(b)] = -\log[d_c(\infty)] + \sum\limits_{k=0}^\infty c_k
  \Psi_k(\infty)\,,
\end{equation*}
from which
\begin{equation}  \label{eq:dc_spectral}
  -\log[d_c(b)] = -\log[d_c(\infty)] + 2\pi
  \sum\limits_{k=0}^\infty \frac{\Psi_k^2(\infty)}{\mu_k + b} \,.
\end{equation}
The spectral expansion \eqref{eq:dc_spectral} is the main result of
this Appendix.  For instance, when $\B$ is the unit disk, all
$\Psi_k(\y)$ decay at infinity, except for $\Psi_0(\y) =
1/\sqrt{2\pi}$ so that the above sum has the single term $k =0$.
Since $d_c(\infty) = 1$ and $\mu_0 = 0$, we retrieve the result
$\log[d_c(b)] = -1/b$.  For a general cross-section $\B$, the
eigenvalues $\mu_k$ and the coefficients $\Psi_k^2(\infty)$ can be
found numerically (see, e.g., \cite{Chaigneau24,Grebenkov25}).  In
practice, the computation of a limited number of eigenmodes can be
sufficient for an accurate computation of $d_c(b)$.  The exterior
Steklov problem also admits a natural extension to the degenerate case
in which $\B$ is an interval (see below), allowing one to describe the
limiting case of highly anisotropic filament cross-sections.

The expansion \eqref{eq:dc_spectral} provides the functional
dependence of $d_c(b)$ on the reactivity $b$.  In particular, we
conclude that $-\log[d_c(b)]$ is a monotonically decreasing function
of $b$: it decreases from $+\infty$ at $b = 0$ up to
$-\log[d_c(\infty)]$ at $b = \infty$. The small-$b$ asymptotic behavior of
$d_c(b)$ can be obtained by a Taylor expansion as
\begin{equation}  \label{eq:dc_small}
-\log[d_c(b)] = \frac{C_0}{b} + \sum\limits_{n=0}^\infty C_{n+1} (-b)^n \,,
\end{equation}
where we use $\Psi_0=|\pa|^{-1/2}$ and $\mu_0=0$ to get
\begin{equation}
C_0 = \frac{2\pi}{|\pa|} \,, \qquad 
C_1 = -\log[d_c(\infty)] + 2\pi \sum\limits_{k=1}^\infty \frac{\Psi_k^2(\infty)}
{\mu_k} \,, \qquad
C_{n} = 2\pi \sum\limits_{k=1}^\infty \frac{\Psi_k^2(\infty)}{\mu_k^n}
\quad (n=2,3,\ldots)\,.
\end{equation}

We now derive alternative representations of the spectral sums that
define $C_1$ and $C_2$. To do so, we first introduce a new function $w_b$
defined by $w_b = v_c - 2\pi/(b|\pa|)$, which satisfies
\begin{equation}  \label{eq:auxil1}
\Delta_{\y} w_b = 0\,, \quad \y\in\R^2\setminus \B\,; \qquad
\partial_n w_b + b w_b = - \frac{2\pi}{|\pa|}\,, \quad \y\in\pa\,,
\end{equation}
with $w_b(\y) = \log|\y| - \log[d_c(b)] - \tfrac{2\pi}{b|\pa|} + o(1)$
as $|\y|\to\infty$. Upon applying the divergence theorem to
\eqref{eq:auxil1}, we find that $\int\limits_{\pa} w_b \dd s= 0$ holds
$\forall~ b > 0$.

In the limit $b\to 0$, we expand $w_b$ as
\begin{equation}\label{eq:wb_exp}
  w_b = w + b w_{2} + {\mathcal O}(b^2) \,.
\end{equation}
Upon substituting \eqref{eq:wb_exp} into both \eqref{eq:auxil1} and the
condition $\int\limits_{\pa} w_b \dd s=0$, and recalling
the expansion of $-\log[d_c(b)]$ in \eqref{eq:dc_small}, we collect
powers of $b$ to obtain the following problems for $w$ and $w_2$:
\begin{equation}\label{eq:w_def}
    \Delta_{\y} w=0 \,, \quad \y\in \R^2\setminus \B \,; \qquad
    \partial_{n} w  = -\frac{2\pi}{|\pa|} \,, \quad \y\in\pa \,; \qquad
    \int_{\pa} w \dd s =0 \,,
\end{equation}
and
\begin{equation}\label{eq:w2prob}
    \Delta_{\y} w_2 =0 \,, \quad \y\in \R^2\setminus \B \,; \qquad
   \partial_{n} w_2 = -w \,, \quad \y\in\pa \,; \qquad
    \int_{\pa} w_2 \dd s = 0 \,.
\end{equation}
The two constants $C_1$ and $C_2$ are calculated from the limiting behaviors:
\begin{equation}\label{eq:C1C2}
  C_1=\lim_{|\y|\to\infty} \left( w-
\log|\y|\right) \,, \qquad C_2=-\lim_{|\y|\to\infty} w_{2}\,.
\end{equation}
Upon using Green's second identity between $w$ and $w_2$, we readily find
\begin{equation}\label{eq:c2prob}
  C_2 = - w_{2}(\infty) = \frac{1}{2\pi} \int_{\pa} w^{2} \dd s \,.
\end{equation}
We remark that $C_2$ can also be obtained from the Steklov expansion, since
from \eqref{eq:vc_Psi} we have that
\begin{equation}
  w|_{\pa} = 2\pi \sum\limits_{k=1}^\infty
  \frac{\Psi_k(\infty) \Psi_k(\y)}{\mu_k}
  \,,
\end{equation}
from which we conclude upon squaring and using orthonormality that
\begin{equation}  \label{eq:C2}
C_2 = 2\pi \sum\limits_{k=1}^\infty \frac{\Psi_k^2(\infty)}{\mu_k^2} 
= \frac{1}{2\pi} \int\limits_{\pa} w^2 \dd s \,.
\end{equation}

To determine $C_1$ and $C_2$, the key step is to solve
\eqref{eq:w_def} by using conformal maps after identifying
$\R^2\setminus \B$ as the region $\C \setminus \B$ in the complex
plane. To do so, we let $\xi=y_1+iy_2$ be any point in $\C \setminus
\B$, and $\D\equiv \lbrace{z\in \C \, | \, |z|<1\rbrace}$ be the unit
disk.  For a simply connected cross-section $\B$ with sufficiently
regular boundary, let $\xi=f(z)$ be a conformal mapping from the
exterior of the unit disk,  $\C\setminus \D$, onto $\C \setminus
\B$. As $|z|\to\infty$, it can be developed in a Laurent series as
\begin{equation}
  \xi = f(z) \sim a_0 z + b_0 + {b_1/z} + {b_2/z^2} + \ldots \,,
\end{equation}
so that $\log|\xi|\sim \log|z| + \log{a_0} + o(1)$ as $z\to
\infty$. Since the exact solution for $V_c(z)\equiv v_c(f(z))$ when $b=\infty$
is $\log|z|$, we have $v_c\sim \log|\xi|-\log(a_0)$, which identifies
that $a_0=d(\infty)$, i.e. $a_0$ is the simply the logarithmic capacity of
$\pa$.

Moreover, setting $W(z) = w(f(z))$, we obtain in terms of the Jacobian
$q(\theta) \equiv |f^{\prime}(e^{i\theta})|$ of the map on $\partial \D$,
that \eqref{eq:w_def} transforms to the following problem for
$W$:
\begin{equation}\label{eq:Wmap}
    \Delta_{z} W=0 \,, \quad z \in \C \setminus \D \,; \qquad
    \partial_{r} W = \frac{2\pi}{|\pa|} q(\theta) \,, \quad \mbox{on}
    \quad r=|z|=1 \,; \qquad
    \int_{0}^{2\pi} W(e^{i\theta}) q(\theta) \dd \theta = 0  \,.
\end{equation}
In terms of this solution, we get from \eqref{eq:C1C2} together with
the large $z$ behavior of the map that
\begin{equation}\label{eq:c1W}
    \qquad C_1=-\log [d_c(\infty)] + \lim_{|z|\to\infty} \left( W(z)-
    \log|z|\right) \,.
\end{equation}

Next, since $\int_{0}^{2\pi} q(\theta)\, d\theta=|\pa|$, we can represent the
solution to the exterior Neumann problem \eqref{eq:Wmap} in $|z|>1$ in
terms of a constant $A$ as
\begin{equation}
  W(z) = \log |z| + A + \frac{2}{|\pa|} \int\limits_0^{2\pi}
  q(\theta^{\prime})\, \log\left|1 - \frac{e^{i\theta^{\prime}}}{z}\right|
  \dd\theta^{\prime} \,.
\end{equation}
On the circle $|z|=1$ we have
\begin{equation} \label{eq:Wtheta}
  W(e^{i\theta}) = A + \frac{2}{|\pa|} \int\limits_0^{2\pi}
  q(\theta^{\prime}) \, \log|e^{i\theta} - e^{i\theta^{\prime}}|
  \dd\theta^{\prime} \,,
\end{equation}
and so the zero-integral constraint in \eqref{eq:Wmap} determines $A$ as
\begin{equation}
  A = - \frac{2}{|\pa|^2} \int\limits_0^{2\pi} \int\limits_0^{2\pi} \,
  q(\theta) \, q(\theta^{\prime}) \, \log|e^{i\theta} - e^{i\theta^{\prime}}|
  \dd \theta \dd\theta^{\prime} \,.
\end{equation}
Finally, since $W(z)\sim \log|z|+A + o(1)$ as $|z|\to\infty$, we identify
$C_1$ from \eqref{eq:c1W} as
\begin{equation}  \label{eq:C1_conformal}
  C_1 = - \log[d_c(\infty)] - \frac{2}{|\pa|^2} \int\limits_0^{2\pi}
  \int\limits_0^{2\pi} q(\theta) \, q(\theta^{\prime})
  \, \log|e^{i\theta} - e^{i\theta^{\prime}}| \dd\theta \dd\theta^{\prime}\,.
\end{equation}
Upon using the inverse conformal map $f^{-1}$, we can write $C_1$ in an
alternative form as
\begin{equation}
  C_1 = - \log[d_c(\infty)] - \frac{2}{|\pa|^2} \int\limits_{\pa}
  \int\limits_{\pa} \log\bigl|f^{-1}(\y_1) - f^{-1}(\y_2)\bigr|
  \dd \y_1 \dd\y_2 \,.
\end{equation}
 
Next, by using the identity
\begin{equation}
  \log|e^{i\theta} - e^{i\theta^{\prime}}| = -
  \sum\limits_{n=1}^\infty \frac{\cos(n(\theta - \theta^{\prime}))}{n}
  = - \sum\limits_{n=1}^\infty \frac{\cos(n\theta) \cos(n\theta^{\prime})
    + \sin(n\theta) \sin(n\theta^{\prime})}{n} \,,
\end{equation}
as readily derived from the Taylor expansion of the principal branch
of $\log(1-z)$, we can obtain a further expression for $C_1$.  If
$q(\theta)$ is even, as in the examples considered below, the sine
contributions vanish so that
\begin{equation}  \label{eq:C1_series}
  C_1 = - \log[d_c(\infty)] + \frac{2\pi^2}{|\pa|^2}
  \sum\limits_{n=1}^\infty \frac{q_n^2}{n} \,, \qquad
  \mbox{where} \qquad q_n \equiv
  \frac{1}{\pi} \int\limits_0^{2\pi}  q(\theta) \cos(n\theta) \dd\theta \,.
\end{equation}

Similarly, we can rewrite \eqref{eq:C2} by using the conformal map as
\begin{equation}  \label{eq:C2b}
  C_2 = \frac{1}{2\pi} \int\limits_{0}^{2\pi} q(\theta) \,
  [W(e^{i\theta})]^2  \dd\theta \,,
\end{equation}
where $W(e^{i\theta})$ is given in \eqref{eq:Wtheta}.  Under the
same symmetry assumption on $q(\theta)$, $C_2$ can be expressed in
terms of the  Fourier cosine coefficients as
\begin{equation}
  C_2 = \frac{\pi^2}{|\pa|^2} \sum_{m=1}^{\infty} \sum_{n=1}^{\infty}
  \frac{q_m q_n}{mn} (q_{m+n} + q_{|m-n|}) - \frac{2\pi^3}{|\pa|^3}
\biggl(\sum\limits_{n=1}^\infty \frac{q_n^2}{n}\biggr)^2 \,.
\end{equation}

As an example, let $\B$ be the interval $\B = (-1,1)$.  The Joukowski
function $f(z) = \tfrac12 (z + 1/z)$ maps $\C\setminus \D$ onto
$\C\setminus (-1,1)$.  As a consequence, we have
$q(\theta)=|f^{\prime}(e^{i\theta})|=\tfrac{1}{2}|1-e^{-2i\theta}|=|\sin\theta|$,
so that with $d_c(\infty)={1/2}$ we obtain from \eqref{eq:C1_conformal}
that
\begin{equation}
  C_1 = -\log(1/2) -\frac{2}{4^2} \int\limits_0^{2\pi}
  \int\limits_0^{2\pi} |\sin(\theta)| \, |\sin(\theta^{\prime})|\,
  \log|e^{i\theta} - e^{i\theta^{\prime}}| \dd\theta
  \dd\theta^{\prime} = \frac{3}{2} -\log 2 \,,
\end{equation}
in agreement with the result (3.9) in \cite{GrebWard2D} (note
that here we used $|\pa| = 4$ because two ``sides'' of the interval do
contribute).  A brute-force computation of the integral \eqref{eq:C2b}
yields $C_2 = (21 - 2\pi^2)/(18\pi) \approx 0.0223$, as derived in
\cite{GrebWard2D}.

For an ellipse with semi-axes $a_1$ and $a_2$ (with $a_1\leq a_2$),
we have $d_c(\infty) = (a_1+a_2)/2$.  The conformal map is $f(z) = a
(z + \beta/z)$, where $a \equiv (a_1+a_2)/2$ and $\beta \equiv
(a_2-a_1)/(a_2+a_1)$. Therefore, we get $f(e^{i\theta}) = a_2 \cos\theta +
ia_1 \sin\theta$, from which $q(\theta) = \sqrt{a_2^2 \sin^2\theta +
  a_1^2 \cos^2\theta}$.  The perimeter is $|\pa| = 4a_2 E(e)$, with $e
= \sqrt{1 - (a_1/a_2)^2}$, where $E(e)$ denotes the complete elliptic
integral of the second kind.  These expressions fully determine $C_1$,
either in terms of the double integral in \eqref{eq:C1_conformal}, or
in terms of \eqref{eq:C1_series}. Numerical results for $C_1$ and
$C_2$ for an ellipse are shown in Fig.~\ref{fig:C1_ellipse}.

\begin{figure}
\begin{center}
\includegraphics[width=80mm]{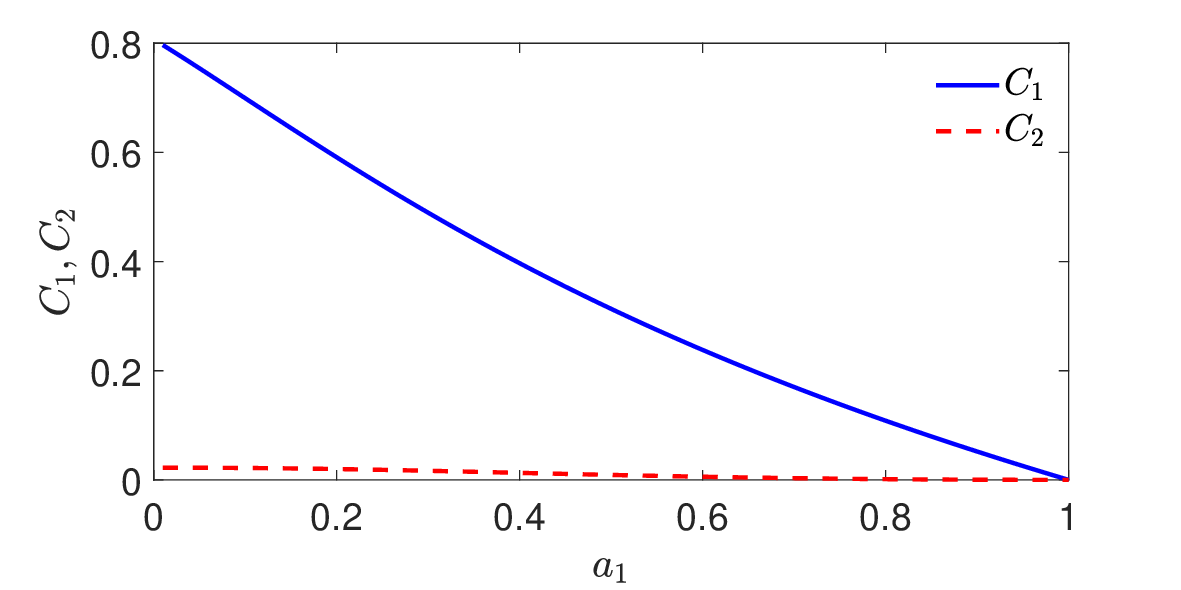}
\end{center}
\caption{
  The coefficients $C_1$ and $C_2$ as a function of $a_1$ for the
exterior of an ellipse with semi-axes $a_1$ and $a_2 = 1$.}
\label{fig:C1_ellipse}
\end{figure}

\section{Intrinsic self-energy of a simple closed curve}\label{app:curve}

Let $\Gam$ be a simple closed curve in ${\mathbb R}^3$ parameterized
by $\X(s^{\prime})$ where $s^{\prime}$ is arclength along $\Gam$. For an
arbitrary target point $\X(s)\in \Gam$, let $\X_{\eta}(s)$ denote any point
at a distance $\eta>0$ from $\X(s)$ where without loss of generality we
assume that $s$ is not an endpoint, i.e. that $0<s<L$. In this appendix we
show that the intrinsic self-energy of a curve, defined by the limiting
process
\begin{equation}\label{appa:intrinsic}
  {\mathcal E}(s) \equiv \lim_{\eta\to 0^{+}} \left( \int_{\Gam}
  \frac{\dd s^{\prime}}{
    |\X_{\eta}(s)-\X(s^{\prime})|}  + 2\log\eta \right) \,,
\end{equation}
has a finite value. For the special case where $\Gam$ is a circular ring
of radius $r$, we will calculate ${\mathcal E}(s)$ analytically.
Moreover, we will provide a numerical approach to compute ${\mathcal E}(s)$
for arbitrary simple closed curves.

We first show that the integral $I$ defined by
\begin{equation}\label{appa:int_zero}
  I \equiv \int_{\Gam} \frac{\dd s^{\prime}}{|\X_{\eta}(s)-\X(s^{\prime})|}\,,
\end{equation}
has a logarithmic singularity as $\eta\to 0^{+}$. In a neighbourhood
of the target point $\X(s)\in \Gam$, we let $s^{\prime}=s+\sigma$
where $\sigma\ll 1$.  Since $\Gam$ is smooth and $s^{\prime}$ is
arclength on $\Gam$ we have $\X(s^{\prime}) = \X(s) + \sigma
\X^{\prime}(s)+ {\mathcal O}(\sigma^2)$ as $\sigma\to 0$, where
$|\X^{\prime}(s)|=1$. As a result, by Pythagoras we conclude that
$|\X_{\eta}(s)-\X(s^{\prime})|\approx \sqrt{\sigma^2+\eta^2}$ for
$|\sigma|=|s^{\prime}-s| \ll 1$ and $\eta\ll 1$. Next, we decompose
$I$ into a local integral and a global part as
\begin{equation}\label{appa:int_zero_1}
  I \approx \underbrace{\int_{s-\sigma}^{s+\sigma} \frac{ \dd
      s^{\prime}}{\sqrt{\eta^2+ (s^{\prime}-s)^2}}}_{\mbox{local part}} +
    \int_{0}^{s-\sigma} \frac{\dd
      s^{\prime}}{|\X_{\eta}(s)-\X(s^{\prime})|} + \int_{s+\sigma}^{L}
    \frac{\dd s^{\prime}}{|\X_{\eta}(s)-\X(s^{\prime})|}\,.
\end{equation}
The local integral in \eqref{appa:int_zero_1} is calculated exactly as
\begin{equation}\label{appa:iloc}
  I_{loc} \equiv \int_{-\sigma}^{\sigma} \frac{ \dd \xi}{\sqrt{\eta^2+
    \xi^2}} = -2\log\eta + 2 \log\left[\sigma + \sqrt{\sigma^2+\eta^2}
    \right] \,.
\end{equation}
Finally, with the ordering $\eta\ll \sigma\ll 1$, we let $\eta\to
0^{+}$ to obtain that $I_{loc} \sim -2\log\eta $ as $\eta\to 0$. We
conclude that the leading-order singular behavior of $I$ is universal
and independent of the shape of $\Gam$. The remaining
$\eta$-independent constant, augmented by the contribution of $\Gam$
to $I$ away from the target point $\X(s)$ is precisely the finite part
in \eqref{appa:intrinsic} that defines the intrinsic self-energy
${\mathcal E}(s)$ of $\Gam$.

\subsection{Intrinsic self-energy of a ring}\label{appa:der_loop}

We now compute ${\mathcal E}(s)$ when $\Gam$ is a circular ring of
radius $r$ parameterized as
\begin{equation}\label{appa:loop}
  \X(s^{\prime}) = r \left( \cos\left({s^{\prime}/r}\right),
  \sin\left({s^{\prime}/r}\right), 0\right)^{T} \,, \quad 0\leq s^{\prime}
  \leq 2\pi r \,.
\end{equation}
By symmetry it follows that ${\mathcal E}(s)={\mathcal E}_0$,
independent of the target point $s$ on the ring. Therefore, without
loss of generality we can take $\X_{\eta}(s)=(r+\eta,0,0)^{T}$ in
\eqref{appa:intrinsic}, and for this point we calculate that
\begin{equation}\label{appa:dist}
  | \X_{\eta}(s)-\X(s^{\prime})|^2 = \left(r+\eta-r\cos\left({s^{\prime}/r}\right)
  \right)^{2} + r^2 \sin^{2}\left({s^{\prime}/r}\right) = \eta^{2} +
   4 r (r+\eta) \sin^{2}\left({s^{\prime}/(2r)}\right) \,.
\end{equation}
Upon introducing $\theta$ by $s^{\prime}=r\theta$, we
obtain from \eqref{appa:dist} that
\begin{equation}\label{appa:dist_1}
  I \equiv \int_{\Gam} \frac{\dd s^{\prime}}{|\X_{\eta}(s)-\X(s^{\prime})|} 
  = \int_{0}^{2\pi} \frac{ r \dd\theta}{\left[\eta^2 + 4r(r+\eta)\sin^{2}
      \left({\theta/2}\right)\right]^{1/2}} = 4r \int_{0}^{\pi/2}
    \frac{\dd\phi}{\left[\eta^2 + 4r(r+\eta)\sin^{2}\phi \right]^{1/2}}\,.
\end{equation}
By introducing $\psi={\pi/2}-\phi$ and using $\cos^{2}\psi=1-
\sin^{2}\psi$, we identify that
\begin{equation}\label{appa:I}
  I = \frac{4r}{2r+\eta} \int_{0}^{\pi/2} \frac{\dd\psi}
  {\sqrt{1-k^2\sin^{2}\psi}} = \frac{4r}{2r+\eta} K(k)
  \quad \mbox{where} \quad k \equiv \frac{2\sqrt{r(r+\eta)}}{2r+\eta}\,,
\end{equation}
and $K(k)$ is the complete elliptic integral of the first kind. In this
way, we evaluate from \eqref{appa:intrinsic} the limit
\begin{equation}\label{appa:Ilim}
  {\mathcal E}(s)= {\mathcal E}_{0} =
  \lim_{\eta\to 0} \left( \frac{4r}{2r+\eta} K(k) + 2\log\eta \right) \,,
\end{equation}
where $k$ is defined in \eqref{appa:I}.

We observe that $k\to 1^{-}$ as $\eta\to 0^{+}$. For
$k\to 1^{-}$, we use the limiting asymptotics $K(k)\sim
\log\left({4/k^{\prime}}\right)$ with $k^{\prime}\equiv
\sqrt{1-k^2}$. By using \eqref{appa:I} for $k$ we obtain that
$k^{\prime}={\eta/(2r+\eta)}$. In this way, we derive that $K(k)\sim
-\log\eta + \log(8r) + o(1)$ as $\eta\to 0^{+}$. Therefore, from
\eqref{appa:Ilim} we conclude that
\begin{equation}\label{appa:loop_res}
  {\mathcal E}_0 =\lim_{\eta\to 0^{+}}
  \left( 2 \left[-\log\eta + \log(8r) + o(1)
  \right] + 2\log\eta\right) = 2 \log(8r) \,,
\end{equation}
which is the intrinsic self-energy of a loop of radius $r$.

\subsection{Numerical method: Arclength parameterization}\label{appa:num_intrinsic}

Let us consider a disk of radius $\rho=L/(2\pi)$ and perimeter $L$
equivalent to that of the filament. We will use the solution of this
case to obtain an integral for $\mathcal{E}(s)$ with a removable
singularity. Consider two points of arclengths $s$ and $s^{\prime}$
for which the chord between them is of length
\begin{equation}\label{eq:d}
  d(s-s^{\prime})=2\rho\left|\sin\left(\frac{s-s^{\prime}}{2\rho}\right)\right|=
  \frac{L}{\pi}\left|\sin\left(\frac{\pi(s-s^{\prime})}{L}\right)\right|\,.
\end{equation}
This expression is $L$-periodic and even in $s-s^{\prime}$ and vanishes at
$s^{\prime}=s$ with local behavior
\begin{equation}\label{eq:dexp}
  d(s-s^{\prime})=|s-s^{\prime}|-\frac{\pi^2}{6L^2}|s-s^{\prime}|^3+
  \mathcal O(|s-s^{\prime}|^5)\,.
\end{equation}
Upon adding and subtracting the regularized kernel to
\eqref{appa:intrinsic} we have that
\begin{equation}\label{eq:split}
  \mathcal E(s)=\underbrace{\lim_{\eta\to 0^{+}}
    \int_\Gamma\left[\frac{1}{|\X_\eta(s)-\X(s^{\prime})|}
      -\frac{1}{\sqrt{\eta^2+\left[d(s-s^{\prime})\right]^2}}\right]
    \dd s^{\prime}}_{\text{(I)}}
  +\underbrace{\lim_{\eta\to0}\left(\int_\Gamma\frac{\dd s^{\prime}}
    {\sqrt{\eta^2+\left[d(s-s^{\prime})\right]^2}}
    +2\log\eta\right)}_{\text{(II)}} \,.
\end{equation}
Using the local behavior \eqref{eq:dexp} and \eqref{appa:dist}, we
observe that the first integral is de-singularized and we can pass to
the limit as $\eta\to0^{+}$. Hence, we obtain for the first integral that
\begin{equation}\label{eq:I}
  \text{(I)}=\int_\Gamma\left[\frac{1}{|\X(s)-\X(s^{\prime})|}-
    \frac{1}{d(s-s^{\prime})}\right]\dd s^{\prime} \,.
\end{equation}

For the second integral in \eqref{eq:split}, we additionally de-singularize
with \eqref{appa:iloc} to obtain
\begin{equation}\label{eqn:self_E_int_1}
\int_\Gamma\frac{\dd s^{\prime}}{\sqrt{\eta^2+\left[d(s-s^{\prime})\right]^2}}
=\int_\Gamma\left[\frac{1}{\sqrt{\eta^2+ \left[d(s-s^{\prime})\right]^2}}
  -\frac{1}{\sqrt{\eta^2+(s-s^{\prime})^2}}\right]\dd s^{\prime}
+\int_{s-L/2}^{s+L/2}\frac{\dd
  s^{\prime}}{\sqrt{\eta^2+(s-s^{\prime})^2}} \,.
\end{equation}
Using the periodicity of the curve, we find that as $\eta\to0^{+}$, the first
integral in \eqref{eqn:self_E_int_1} tends to
\begin{align*}
  \int_{s-L/2}^{s+L/2}\left[\frac{1}{d(s-s^{\prime})}-\frac{1}{|s-s^{\prime}|}
    \right]\dd s^{\prime}
  &=2\int_{s}^{s+L/2}\left[\frac{\pi/L}{\sin\big(\pi(s^{\prime}-s)/L\big)}
    -\frac{1}{s^{\prime}-s}\right]\dd s^{\prime} \\
& =2\int_0^{\pi/2}\left[\frac{1}{\sin\theta}-\frac1\theta\right]\dd\theta
=2\Big[\log\tan\left(\frac\theta2\right) -\log\theta\Big]_{0}^{\pi/2}
=2\log\left(\frac4\pi\right) \,,
\end{align*}
where we used the substitution $\theta=\pi(s^{\prime}-s)/L$. For the
second integral in \eqref{eqn:self_E_int_1}, we have that
\begin{equation}\label{eqn:self_E_int_2}
\int_{s-L/2}^{s+L/2}\frac{\dd
  s^{\prime}}{\sqrt{\eta^2+(s-s^{\prime})^2}}=2\,
\mathrm{arcsinh}\left(\frac{L}{2\eta}\right)=2\log\left(\frac{L}{\eta}\right)
+\mathcal O\Big(\frac{\eta^2}{L^2}\Big)\, \quad \mbox{as} \quad \eta\to0^{+} \,.
\end{equation}

Finally, by combining equations \eqref{eq:I}-\eqref{eqn:self_E_int_2} in
\eqref{eq:split}, we obtain the final result
\begin{equation}\label{eq:E}
  \mathcal E(s)=2\log\left(\frac{4L}{\pi}\right) +
  \int_\Gamma\left[\frac{1}{|\X(s)-\X(s^{\prime})|}-
    \frac{\pi/L}{\big|\sin\big(\pi(s-s^{\prime})/L\big)\big|}\right]\dd
  s^{\prime} \,.
\end{equation}
In the case where the curve is a circle of radius $\rho$, we have that
$|\X(s)-\X(s^{\prime})|=d(s-s^{\prime})$ so the integral term in
\eqref{eq:E} vanishes and we recover $\mathcal E_0=2\log(8\rho)$ as
derived in \eqref{appa:loop_res}. The reduced formulation in
\eqref{eq:E} features an integral that is well behaved and is amenable
to standard integration techniques.

\section{Optimization of the radius of an equatorial filament}\label{appr:optim}

In this appendix we prove Proposition \ref{prop:opt_ring} 
establishing that there is a unique value $r_{0min}$ of the centerline
radius $r_0$ that minimizes the global MFRT for a concentric
equatorial filament in the unit sphere. We also approximate $r_{0min}$
for $\eps\ll 1$.

To do so, we first seek properties of the integral $J(r_0)$ in
\eqref{sec3:I0J0}. Although it appears intractable to evaluate
$J(r_0)$ in closed form, we establish in the next result that
$J(r_0)$ is monotone decreasing on $0<r_0<1$. 

\begin{lemma}\label{lemma:J} On $0< r_0< 1$, we have
\begin{equation}\label{sec3:jp}
    J^{\prime}(r_0) = \frac{4\pi}{r_0} \left( 1 - \frac{2}{\pi}
    K\left(r_0^2\right)\right) <0 \,,
\end{equation}
since $K(r_0^2)>{\pi/2}$ on $0<r_0<1$.  Moreover, $J(0)=2\pi
\log{2}\approx 4.3552$ and $J(1)=8G_{c}-2\pi\log{2}\approx 2.9726$,
where $G_c$ is Catalan's constant. 
\end{lemma}

\begin{proof}
  The result for $J(0)$ is immediate. To find $J(1)$ we use
  $1-\cos\psi=2\sin^2\left({\psi/2}\right)$ to get
  \begin{align*}
    J(1) &=\int_{0}^{2\pi} \log\left[ 2\sin\left(\frac{\psi}{2}\right)
      \left( 1 + \sin\left(\frac{\psi}{2}\right) \right)\right] \dd\psi \,, \\
     & = 2\pi \log{2} + 2\int_{0}^{\pi} \log\left( \sin{x}\right) \dd x +
    2\int_{0}^{\pi} \log\left( 1 + \sin{x} \right) \, \dd x \,, \\
&= 2 \pi \log{2} + 2 \left(-\pi \log{2}\right) + 4\left(2 G_c -
    \frac{\pi}{2}\right) \,,\\
    &= 8 G_{c}-2\pi\log{2} \,,
  \end{align*}
where $G_c\equiv -\int_{0}^{\pi/4} \log\left[\tan(x)\right]\, \dd x
\approx 0.915966$ is Catalan's constant. Next, we label $g\equiv1-r^2\cos\psi + C$, where $C\equiv\left(r_0^4-
2r_0^2 \cos\psi +1\right)^{1/2}$. We calculate that
\begin{equation}\label{sec3:jp_0}
  J^{\prime}(r_0)=\int_{0}^{2\pi} \frac{\partial_{r_0}g}{g} \, \dd\psi
  \qquad \mbox{where} \qquad \frac{\partial_{r_0}g}{g} = \frac{2r_0}{C} \left(
  \frac{r_0^2-(1+C)\cos\psi} {1-r_0^2\cos\psi + C} \right) \,.
\end{equation}
Upon multiplying top and bottom by $1-r_0^2\cos\psi -C$, we calculate the
denominator $(1-r^2\cos\psi)^2-C^2$ to get
\begin{equation} \label{sec3:jp_1}
  \frac{\partial_{r_0}g}{g} = -\frac{2 r_0}{C} \frac{
    \left(r_0^2-(1+C)\cos\psi\right) \left( 1-C - r_0^2
    \cos\psi\right)}{r_0^4\sin^{2}\psi}\,.
\end{equation}
By collecting powers of $C$ in the numerator and by using $C^2=r_0^4-2r_0^2
\cos\psi+1$, we obtain after some algebra that \eqref{sec3:jp_1} reduces to
\begin{equation}\label{sec3:jp_2}
  \frac{\partial_{r_0}g}{g} = \frac{2}{r_0} - \frac{2}{r_0 C} \,.
\end{equation}
Finally, by substituting \eqref{sec3:jp_2} into 
\eqref{sec3:jp_0}, we evaluate $J^{\prime}(r_0)$ as
\begin{equation}\label{sec3:jp_f}
  J^{\prime}(r_0) = \int_{0}^{2\pi}\left( \frac{2}{r_0} - \frac{2}{r_0
    \left(r_0^4-2r_0^2\cos\psi +1\right)^{1/2}} \right)\dd\psi =
  \frac{4\pi}{r_0} - \frac{8 K(r_0^2)}{r_0} \,,
\end{equation}
which completes the proof of \eqref{sec3:jp}.
\end{proof}

Next, we use \eqref{sec3:jp} together with
\begin{equation*}
  \frac{d}{dr_0} K\left(r_0^2\right) = \frac{2}{r_0(1-r_0^4)} E(r_0^2)
  - \frac{2}{r_0} K\left(r_0^2\right)\,,
\end{equation*}
where $E(m)\equiv \int_{0}^{\pi/2}
\left(1-m^2\sin^2\theta\right)^{1/2}\, \dd\theta$ is the complete
elliptic integral of the second kind of modulus $m$. In this way, we
calculate from \eqref{sec3:ubar_2} in terms of $a_{\eps} \equiv {8/(\eps
  d_c)}$ that
\begin{equation}\label{sec3:fp}
  {\mathcal F}_{\eps}^{\prime}(r_0) = \frac{1}{r_0}
  {\mathcal C}_{\eps}(r_0) \quad \mbox{where} \quad {\mathcal C}_{\eps}(r_0)
  \equiv \frac{1-\log(a_{\eps} r_0)}{r_0} + \frac{4 E\left(r_0^2\right)}
         {1-r_0^4} + 2\pi (r_0^2 -1)\,.
\end{equation}

To show that ${\mathcal C}_{\eps}(r_0)$ has a unique root on
${\mathcal O}\left({1/a_{\eps}}\right)\ll r_0\ll 1-{\mathcal
  O}\left({1/a_{\eps}}\right)$, we first observe that ${\mathcal
  C}_{\eps}(r_0)\to -\infty$ when $r_0\to 0$ but with ${\mathcal
  O}\left({1/a_{\eps}}\right)\ll r_0$ and that ${\mathcal
  C}_{\eps}(r_0)\to +\infty$ when $r_0\to 1^{-}$ but with $r_0\ll
1-{\mathcal O}\left({1/a_{\eps}}\right)$. By the intermediate value
theorem there must be a root $r_{0\min}$ to ${\mathcal
  C}_{\eps}(r_0)=0$ on this interval.  To show that this root is
unique we now establish that ${\mathcal C}_{\eps}(r_0)$ is monotone
increasing. To do so, we use the identity
\begin{equation*}
  \frac{d}{dr_0} E(r_0^2) = \frac{2}{r_0}\left[ E(r_0^2)- K(r_0^2)\right]\,,
\end{equation*}
to obtain after some algebra that
\begin{equation}\label{sec3:cprime}
  {\mathcal C}_{\eps}^{\prime}(r_0) = 4\pi r_0 + \frac{\left(\log(a_{\eps} r_0)
    - 2\right)}{r_0^2} + \frac{8}{r_0\left(1-r_0^4\right)^2}
  \left[ r_0^4 E(r_0^2) + E(r_0^2) - (1-r_0^4)K(r_0^2)\right] \,.
\end{equation}
By using the well-known identity $E(m)-(1-m^2)K(m)>0$ on $0<m<1$ it
follows from \eqref{sec3:cprime} that ${\mathcal
  C}_{\eps}^{\prime}(r_0)>0$ on ${\mathcal
  O}\left({1/a_{\eps}}\right)\ll r_0\ll 1-{\mathcal
  O}\left({1/a_{\eps}}\right)$.

This unique root to ${\mathcal F}_{\eps}^{\prime}(r_0)=0$, given by the
zero of ${\mathcal C}_{\eps}(r_0)=0$ and labeled by $r_{0\min}$,
is a global minimum of ${\mathcal F}_{\eps}(r_0)$ since from
\eqref{sec3:fp} we have
\begin{equation*}\label{sec3:fpp}
  {\mathcal F}_{\eps}^{\prime\prime}(r_0)=-\frac{{\mathcal C}_{\eps}(r_0)}{r_0^2}
 +  \frac{{\mathcal C}_{\eps}^{\prime}(r_0)}{r_0} =
    \frac{{\mathcal C}_{\eps}^{\prime}(r_{0\min})}{r_{0\min}}  >0  
      \quad \mbox{when}   \quad r_0=r_{0\min} \,.
\end{equation*}      

Next, we approximate $r_{0\min}$ for $a_{\eps}\gg 1$ by using a
dominant balance argument. We readily observe from \eqref{sec3:jp}
that this root is close to the boundary and so we set $r_{0}=1-\delta$
for some $\delta\ll 1$. By using $1-r_0^4\sim 4\delta-6\delta^2$, we
obtain
\begin{equation*}
  -\frac{\log{a_{\eps}}}{1-\delta} + \frac{1+\delta}{1-\delta}
  +  \frac{E\left(1-2\delta+{\mathcal O}(\delta^2)\right)}
  {\delta-{3\delta^2/2}+{\mathcal O}(\delta^3)}
  -4\pi \delta +{\mathcal O}(\delta^2) =0 \,,
\end{equation*}
which yields that
\begin{equation}\label{sec3:root_1}
  -\log a_{\eps} + 1+\delta + \frac{(1-\delta)}{\delta \left(1-{3\delta/2}
    \right)} E(1-2\delta) - 4\pi \delta \approx 0\,.
\end{equation}
Then by using the well-known asymptotics $E(m)\sim
1-\tfrac{(1-m)}{2}\log(1-m) + (3\log{2}-1){(1-m)/2}$ as $m\to 1^{-}$,
we obtain that
\begin{equation}\label{sec3:root_2}
  -\log a_{\eps} + 1 + \frac{1}{\delta} + \log\left( \frac{4}{\delta}\right)
  -\frac{1}{2} + {\mathcal O}(\delta)=0\,,
\end{equation}
which has the approximate root
\begin{equation}\label{sec3:root}
  \delta \sim \frac{1}{\log a_{\eps}} +
  \frac{\left[\log(4 \log a_\eps) +{1/2}\right]}{(\log{a_{\eps}})^2}\,,
  \quad \mbox{where} \quad a_{\eps} = \frac{8}{\eps d_c} \,,
\end{equation}
which yields \eqref{prop:opt_r0}. This completes the derivation of Proposition \ref{prop:opt_ring}.

\section{Large $N$ asymptotics of ${\mathcal S}_N$}\label{appa:sn}

In this appendix we establish Lemma \ref{lemma:sn} for the discrete
sum ${\mathcal S}_N$.

Since ${\mathcal S}_{N}$ is unbounded as $N\to\infty$, we will extract
from the sum two terms with known divergent behavior. To do so we
define $g(x)$ by
\begin{equation}\label{ball:sn_d1}
    g(x)\equiv \frac{1}{\sin(\pi x)} - \frac{1}{\pi x} - \frac{1}{\pi(1-x)}\,,
\end{equation}
so that $g(x)$ is smooth on $0\le x\le 1$ with removable singularities
at $x=0$ and $x=1$. By adding and subtracting, and setting $x={j/N}$,
we then write ${\mathcal S}_{N}$ in \eqref{ball:sn} as
\begin{equation}\label{ball:sn_d2}
      {\mathcal  S}_{N} = \frac{1}{\pi} \left( N \sum_{j=1}^{N-1}
      \frac{1}{j} + \sum_{j=1}^{N-1} \frac{1}{1-{j/N}} \right) +
      \sum_{j=1}^{N-1} g\left({j/N}\right) =
       \frac{2 N}{\pi} H_{N-1} + \sum_{j=1}^{N-1} g\left({j/N}\right) \,,
\end{equation}
where $H_{N-1}$ is the $(N-1)$-th harmonic number with the large-$N$
asymptotics (see \cite{graham})
\begin{equation}\label{ball:hn}
   H_{N-1}\equiv \sum_{j=1}^{N-1} \frac{1}{j} \sim \log(N) + \gamma_e
    - \frac{1}{2N} - \frac{1}{12N^2} + {\mathcal O}\left({1/N^3}\right)
    \quad \mbox{as} \quad N\to \infty \,.
\end{equation}
Next, since $g(x)$ is smooth we apply the Euler-Maclaurin summation
expansion to obtain
\begin{equation}\label{ball:euler}
    \frac{1}{N} \sum_{j=1}^{N-1} g\left({j/N}\right) = \int_{0}^{1} g(x) \dd x -
    \frac{1}{2N} \left[ g(0)+g(1)\right] + \frac{1}{12 N^2} \left[
      g^{\prime}(1)-g^{\prime}(0)\right] + {\mathcal O}\left({1/N^3}\right)\,,
\end{equation}
as $N\to \infty$.  We readily calculate that
\begin{equation}\label{ball:gprop}
    \int_{0}^{1} g(x) \dd x = \frac{2}{\pi} \log\left( \frac{2}{\pi}\right)
    \,, \quad g(0)=g(1)=-\frac{1}{\pi} \,, \quad g^{\prime}(0)=-g^{\prime}(1)
    = \frac{\pi}{6} - \frac{1}{\pi} \,.
\end{equation}
In this way, we obtain from \eqref{ball:euler} and \eqref{ball:gprop} that
\begin{equation}\label{ball:gfinal}
    \sum_{j=1}^{N-1} g\left({j/N}\right) = \frac{2N}{\pi}
    \log\left(\frac{2}{\pi}\right) + \frac{1}{\pi} + \left( \frac{1}{6\pi}
    - \frac{\pi}{36} \right) \frac{1}{N} + {\mathcal O}\left({1/N^2}\right)
    \quad \mbox{as} \quad N\to \infty \,.
\end{equation}
Upon substituting \eqref{ball:gfinal} and \eqref{ball:hn} into
\eqref{ball:sn_d2} we obtain the limiting asymptotics
\eqref{ball:sn_asy} of Lemma \ref{lemma:sn} as we observe that the
${\mathcal O}(1)$ term vanishes, leaving only an ${\mathcal
  O}\left({1/N}\right)$ error estimate with the small
coefficient ${-\pi/36}$.

\section{Ring filament: the lowest eigenvalue of the Laplacian}\label{appc:eig}

For the unit sphere, and for $\lambda_0>0$ and
$j_l^{\prime}(\sqrt{\lambda_0})\ne 0$, so that $\lambda_0$ is not an
eigenvalue of the Laplacian for the filament free-problem in the
unit-sphere for radially symmetric eigenfunctions, the regular part of
the Helmholtz Green's function in \eqref{eig:Gdef2} can be obtained
from separation of variables as \bsub \label{appc:rh}
\begin{equation}\label{appc:rh_1}
\begin{aligned}
  R_h(\x;\bxi;\lambda_0) = \frac{1}{4\pi} \sum_{\ell=0}^{\infty}
  (2\ell + 1) \Big[ &-\sqrt{\lambda_0} \, j_{\ell}(\sqrt{\lambda}r_{<})
    y_{\ell}(\sqrt{\lambda} r_{>}) 
    +  \sqrt{\lambda_0} \frac{
      y_{\ell}^{\prime}(\sqrt{\lambda_0})}{j_{\ell}^{\prime}(\sqrt{\lambda_0})}
    j_{\ell}(\sqrt{\lambda_0} r) j_{\ell}(\sqrt{\lambda_0}r_0) \\
    & \qquad \qquad\qquad\qquad - \frac{r_{<}^{\ell}}{(2\ell+1)
      r_{>}^{\ell+1}} \Big]\,
  P_{\ell}(\cos\gamma)\,.
\end{aligned}
\end{equation}
Here $j_{\ell}(x)$ and $y_{\ell}(x)$ are spherical Bessel functions of the
first and second kind of order $\ell$, $P_{\ell}(x)$ is the Legendre polynomial
of degree $\ell$, and we  have labeled
\begin{equation}\label{appc:rh_2}
  r_0=|\bxi| \,, \quad r=|\x| \,, \quad r_{<}=\min(r,r_0) \,, \quad
  r_{>}=\max(r,r_0) \,,  \quad \cos\gamma  = \frac{ \x {\bf \cdot} \bxi}
  {|\x| |\bxi|}\,.
\end{equation}
\esub

For an equatorial ring of radius $r_0<1$, for which
$\X(s)=r_0\left(\cos({s/r_0}),\sin({s/r_0}),0\right)^{T}$, we use
\eqref{appc:rh} in \eqref{eig:inteq_ring}, after first converting to
angular variables, to obtain
\begin{equation}\label{appc:int_i}
  \begin{split}
  - \log\left(\frac{8r_0}{\eps d_c}\right)
  & = 2\pi r_0 \int_{0}^{2\pi} R_h(\X(r_0\theta);  \X(r_0\theta^{\prime});
  \lambda_0) \dd \theta^{\prime} \,, \\
  & = \sum_{\ell=0}^{\infty} \left[ \sqrt{\lambda_0} r_0 \left(\ell +\frac{1}{2}
    \right) H_{\ell}(\lambda_0,r_0) - \frac{1}{2} \right] \int_{0}^{2\pi}
  P_{\ell}(\cos\psi) \dd \psi \,,
  \end{split}
\end{equation}
where $H_{\ell}(\lambda_0,r_0)$ is given in \eqref{eig:nfunc_h}
upon replacing $l$ by $2k$. Finally, the integral in \eqref{appc:int_i}
can be evaluated using an identity in \cite{erdelyi}. If $\ell$ is odd,
$\int_{0}^{2\pi} P_{\ell}(\cos\psi)\, d\psi =0$, while if $l=2k$ and $k=0,1,2,
\ldots$, we have
\begin{equation}\label{appc:leg}
  \int_{0}^{2\pi} P_{2k}(\cos\psi) \dd \psi = 2\pi
  \left( \frac{ (2k)!}{ 2^{2k} (k!)^{2}} \right)^{2}\,.
\end{equation}
Finally, upon substituting \eqref{appc:leg} into
\eqref{appc:int_i} we conclude that $\lambda_0$ must be a root
of ${\mathcal N}(\lambda)$ as defined in \eqref{eig:nfunc}.

\section{Fragmented filaments: inner problem and numerics}\label{app:slab}

In this appendix we first derive the new inner problem
\eqref{sec6:new_inner} in Proposition \ref{mr:main} for a fragmented
filament.

Let $\X(s)$ be a smooth simple closed curve of length $L$
parameterized by arclength $s\in [0,L)$ with $\X(0)=\X(L)$, which we
  discretize into $N$ equal segments with spacing $\triangle s={L/N}\ll 1$.
  At each $s_j=(j-1)L /N$ for $j=1,\ldots,N$, corresponding to
  cartesian coordinates $\x_j=\X(s_j)$, we center a microtrap of a
  common shape $\Omega_{\sigma}$ and diameter ${\mathcal O}(\sigma)\ll
  1$ to generate a fragmented filament.  We assume that the collection
  of microtraps is disjoint.

  For any fixed $j$ with $1<j<N$ (WLOG), the microtrap centered at
  $s_j$ has two neighbours centered at arclength positions $s_{j-1}$
  and $s_{j+1}$. For this $j$, we label $s_{j-{1/2}}=s_j-{(\triangle
    s)/2}$ and $s_{j+1/2}=s_j+ {(\triangle s)/2}$. By using Taylor
  series, the chord distance $|\X(s_{j+{1/2}}) -\X(s_{j-1/2})|$ for
  $\triangle s\ll 1$ is
  \begin{equation*}
    |\X(s_{j+{1/2}})-\X(s_{j-1/2})| \sim \triangle s - \frac{\kappa_j
        (\triangle s)^3}{24} \,, \quad \mbox{for} \quad \triangle s \ll 1\,,
  \end{equation*} 
  where $\kappa_j\equiv |\X^{\prime\prime}(s_j)|$ is the curvature of $\X(s)$
  at $s=s_j$. Upon neglecting the ${\mathcal O}((\triangle s)^3)$
  error proportional to the curvature, it follows that the three points
  $\X(s_j)$, $\X(s_{j+{1/2}})$ and $\X(s_{j-1/2})$ are all asymptotically colinear
  for $\triangle s\ll 1$ and so the curve is locally straight.

As a result, we define the unit tangent ${\bf t}_{Z}$ by
  \begin{equation}\label{slab:tz}
    {\mathbf t}_{Z} = \frac{\X(s_{j+1/2})-\X(s_{j-1/2}) }{
              |\X(s_{j+1/2})-\X(s_{j-1/2})|} \,,
 \end{equation}
  and we label ${\mathcal P}_{j}$ to be the 2-D plane orthogonal to
  ${\mathbf t}_{Z}$ at $\X(s_j)$. We define two additional planes
  parallel to ${\mathcal P}_{j}$ containing $\X(s_{j+{1/2}})$ and
  $\X(s_{j-{1/2}})$ as ${\mathcal P}_{j+{1/2}}$ and ${\mathcal
    P}_{j-{1/2}}$, respectively, where $\dist( {\mathcal P}_{j-1/2},
  {\mathcal P}_{j+1/2}) = \triangle s={L/N}$. In order to account for
  the periodicity of the microtraps on the entire curve, we will
  impose that the local solution is periodic across these planes. To
  span the plane ${\mathcal P}_j$ through $\X(s_j)$, we introduce two
  mutually orthogonal unit vectors ${\mathbf t}_{X}$ and ${\mathbf t}_{Y}$.

Next, within an ${\mathcal O}(\sigma)$ neighbourhood of $\X(s_j)$, we
introduce the linear transformation
\begin{equation}\label{slab:coord}
  \x = {\mathcal L}(X,Y,Z) \equiv \X(s_j) + \sigma X {\mathbf t}_{X} +
  \sigma Y {\mathbf t}_{Y} + \sigma Z {\mathbf t}_{Z} \,.
\end{equation}
In terms of these local $(X,Y,Z)$ coordinates, the magnified microtrap
centered at $(X,Y,Z)=(0,0,0)$ is $\Omega_0$ and has diameter
${\mathcal O}(1)$. In addition, the images of the planes ${\mathcal
  P}_{j+{1/2}}$ and ${\mathcal P}_{j-{1/2}}$ correspond to $Z=a$ and
$Z=-a$ respectively, where $a\equiv {(\triangle
  s)/(2\sigma)}={L/(2N\sigma)}$.  Upon labeling the local solution by
$V(X,Y,Z)\equiv u\left( {\mathcal L}(X,Y,Z)\right)$, \eqref{sec6:MFPT}
transforms to $\Delta V=-{\sigma^2}/D$, so that to leading-order
$V\sim V_0$ where $V_0$ is harmonic in the slab geometry $|Z|\leq a$,
$-\infty<X,Y<\infty$ with $V_0$ having periodic boundary conditions on
$Z=\pm a$. For each $Z$, we impose that $V_0\sim A(s_j) \log\rho$ as
$\rho\equiv \sqrt{X^2+Y^2}\to\infty$, where $A(s_j)$ is the local
line density evaluated at $s=s_j$, which is determined by the outer
problem. In this way, by introducing $V_c$ by $V_0=A(s_j)V_{c}$, we
obtain that $V_{c}$ satisfies the new canonical inner problem
\eqref{sec6:new_inner} with unit logarithhmic growth for which
$V_c\sim \log\rho+\omega+o(1)$ as $\rho\to \infty$ for some {\em
  infinite slab capacitance} $\omega=\omega(a)$ to be determined.

In terms of the outer variable $\x$, we use \eqref{slab:coord} to estimate
for a fixed $Z$ that $\log|\x-\X(s_j)|= \log\rho + \log\sigma +
{\mathcal O}\left(Z^2/(2\rho^2)\right)$ as $\rho\to \infty$, so that
the asymptotic matching condition to the outer solution becomes
\begin{equation}\label{slab:V0ff}
  V_0 = A(s_j) V_c \sim A(s_j) \log|\x-\X(s_j)| + A(s_j)
  \left(\omega - \log\sigma \right) \,, \quad \mbox{as} \quad \x\to \X(s_j) \,.
\end{equation}
Upon equating the far-field behavior \eqref{slab:V0ff} with the
similar matching condition \eqref{sec2:inn_match} for a perfectly
absorbing filament of cross-sectional radius $\eps$, we simply enforce
${1/\nu}=\omega - \log\sigma$ where $\nu=-{1/\log[\eps d_c]}$ with
$d_c=1$. Solving this expression for $\eps$ gives
\eqref{sec6:eps-omega}. With this choice the outer problems for the
fragmented and unfragmented filaments are exactly the same, and
consequently have the same average MFRT to all logarithmic orders. To
treat fragmented filaments we need only replace $\eps=\sigma
e^{-\omega}$ in the integral equation \eqref{prop:inteq}.

\subsection{Large-$a$ asymptotics of the new inner problem}\label{slab:large-a}

Finally, although for general $a={\mathcal O}(1)$ we must compute
$\omega(a)$ numerically from the solution to \eqref{sec6:new_inner}
(see below) we now establish the asymptotics
\eqref{sec6:a_large} for $a\gg 1$. This range corresponds to the transition
region ${\mathcal O}(\sigma)\ll {L/N} \ll {\mathcal O}(1)$. 

For $a\to\infty$, the boundaries of the slab tend to infinity, and so
\eqref{sec6:new_inner} can be solved asymptotically by inner-outer
matching. In the inner region near the microtrap $\Omega_0$ the
leading-order asymptotic solution is $V_c\sim B\left( 1 - w_c\right)$
for some constant $B$, where $w_c$ solves \eqref{ball:capac}. Using
the far-field behavior of $w_c$ in \eqref{ball:capac}, we obtain that
\begin{equation}\label{slab:ff_0}
  V_c \sim B - \frac{B C}{d} \,, \quad \mbox{as} \quad
  d \equiv \sqrt{X^2+Y^2+Z^2}\to \infty \,,
\end{equation}
where $C$ is the capacitance of $\Omega_0$. For the outer region, for
$a\gg 1$ we introduce the stretched coordinate $\x=(x,y,z)^T={(X,Y,Z)^T/a}$
so that the effect of the microtrap is a point singularity
(here $\x$ is a local variable distinct from the $\x$ in the
outer region). After matching to \eqref{slab:ff_0}, the outer solution,
labeled by $v(\x)$ is
\begin{equation}\label{slab:vout}
  v(\x) \sim - 4\pi B \frac{C}{a} g(\x) + B\,,
\end{equation}
where $g(\x)$ is the unique slab Green's function satisfying
\begin{equation}\label{slab:gfunc}
\begin{split}
  \Delta g & =0\,, \quad \left\vert z\right\vert <1\,, \quad
  -\infty<x,y<\infty \,, \\
  g(x,y,1)&=g(x,y,-1)\,, \quad g_{z}(x,y,1)=g_{z}(x,y,-1)\,,\\
  g(\x) &= \frac{1}{4\pi |\x|} + o(1) \,, \quad \mbox{as} \quad |\x|\to 0\,.
 \end{split}
\end{equation}
The uniqueness follows from the $o(1)$ condition as $|\x|\to 0$.

Defining $r\equiv (x^2+y^2)^{1/2}$, the solution to \eqref{slab:gfunc} has
the form $g=g(r,z)$. The far-field asymptotics of $g(r,z)$,
established below in Lemma \ref{lem:gslab}, is
\begin{equation}\label{slab:gfar}
g(r,z)\sim -\frac{1}{4\pi }\log r+\frac{1}{4\pi }\left( \log \left( 4\right)
-\gamma _{e}\right) \,, \quad \mbox{as} \quad r=\sqrt{x^{2}+y^{2}} \to
\infty \,,
\end{equation}
for each $z$ in $|z|\leq 1$.  Substituting \eqref{slab:gfar} into
\eqref{slab:vout}, and replacing $r={\rho/a}$, we have the following
far-field behavior for any fixed $Z={z/a}$:
\begin{equation}\label{slab:V_ff}
  V_c \sim \frac{BC}{a} \log\rho + \frac{BC}{a} \left(-\log(4a) +
  \gamma_e\right) + B \,, \quad \mbox{as} \quad \rho=\sqrt{X^2+Y^2}\to\infty
  \,.
\end{equation}
Upon comparing \eqref{slab:V_ff} with the target far-field behavior
in \eqref{sec6:new_inner} we must have $B={a/C}$ and
$\omega\sim -\log(4a)+\gamma_e + {a/C}$ as $a\to \infty$. This confirms
\eqref{sec6:a_large} and completes the derivation of Proposition
\ref{mr:main}.

Our final step is to derive \eqref{slab:gfar} for the slab Green's
function satisfying \eqref{slab:gfunc}. This Green's function provides
the canonical connection between 2-D and 3-D. 

\begin{lemma}
  \label{lem:gslab} For $r=\sqrt{x^2+y^2}\to \infty$, the asymptotics
  \eqref{slab:gfar} for the solution to \eqref{slab:gfunc} holds for
  any fixed $|z|\le 1$.
\end{lemma}

\begin{proof}
  Define ${\mathbf e}_z=(0,0,1)^{T}$. By using the method of images,
  $g$ can be represented as
\begin{equation}  \label{gimages}
  g(\x)=\frac{1}{4\pi }\frac{1}{|\x|}+
  \frac{1}{4\pi }\sum_{|k|>0}^{\infty}
  \left( \frac{1}{|\x-2 k {\mathbf e}_z|} -\frac{1}{2|k|} \right)\,.
\end{equation}

By substracting the constant $\frac{1}{2|k|}$ off of each image term,
the resulting function automatically satisfies the $o(1)$ condition at
the origin in \eqref{slab:gfunc}. Moreover this sum is well defined,
since $\tfrac{1}{|\x-2 k {\mathbf e}_z|} - \tfrac{1}{2|k|} = {\mathcal
  O}(k^{-2})$ as $k\to\infty$ for any fixed $\x$.

We now determine the far-field behaviour. Suppose first that $z=0$. We then
write \eqref{gimages} as
\begin{equation}\label{slab:gshit}
  4\pi g(\x)= \frac{1}{r} + 2 S\,,
\end{equation}
where $S$ is defined by
\begin{equation}\label{slab:sk}
  S = \lim_{K\to \infty} \left(\sum_{k=1}^{K}\frac{1}{\sqrt{r^2+ 4k^2}} -
      \sum_{k=1}^{K} \frac{1}{2k} \right) \,.
\end{equation}
By using
\begin{equation*}
\sum_{1}^{K}\frac{1}{k}=\log K+\gamma _{e}+{\mathcal O}(K^{-1}) \,,
\end{equation*}
we rewrite \eqref{slab:sk} as
\begin{equation}\label{slab:sk1}
  S = \lim_{K\to \infty} \left(\sum_{k=1}^{K}\frac{1}{\sqrt{r^2+ 4k^2}} -
      \frac{1}{2} \log K \right) - \frac{\gamma_e}{2}\,.
\end{equation}
For large $r$ on the range ${\mathcal O}(1)\ll r \ll {\mathcal O}(K)$,
we introduce $\triangle \xi={2/r}\ll 1$ and $\xi_k={2k/r}$,
to obtain 
\begin{equation}
  \sum_{k=1}^{K} \frac{1}{\sqrt{r^2+4k^2}} = \frac{1}{2} \sum_{k=1}^{K}
  \frac{\triangle \xi}{\sqrt{1+\xi_k^2}} \,.
\end{equation}
In terms of the right Riemann sum we identify that
\begin{equation}\label{slab:rr}
  \lim_{r\to\infty} \left(\frac{1}{2} \sum_{k=1}^{K}
  \frac{\triangle \xi}{\sqrt{1+\xi_k^2}} \right) = \frac{1}{2}
  \int_{0}^{2K/r} \frac{\dd \xi}{\sqrt{1+\xi^2}} =
  \frac{1}{2} \log\left(\xi+ \sqrt{1+\xi^2}\right) \Big{\vert}_{0}^{2K/r}
   \sim \frac{1}{2}\log\left(\frac{4K}{r}\right) + o(1) \,,
\end{equation}
when ${\mathcal O}(1)\ll r \ll {\mathcal O}(K)$.  Upon substituting
\eqref{slab:rr} into \eqref{slab:sk1} we obtain for $r\gg 1$ that
\begin{equation}\label{slab:sfinal}
  S \sim \lim_{K\to \infty} \left(\frac{1}{2}\log K + \log{2} - \frac{1}{2}
  \log{r} - \frac{1}{2} \log{K} + o(1) \right)- \frac{\gamma_e}{2} \sim
  -\frac{1}{2}\log{r} + \log{2} -\frac{\gamma_e}{2} \,.
\end{equation}
Finally, by using \eqref{slab:sfinal} in \eqref{slab:gshit}, we
obtain the limiting behavior for $g$ in \eqref{slab:gfar}.

We remark that since we are taking the limit of large $r$, this
derivation remains unaltered even if $z$ is non-zero since $\left\vert
z\right\vert <1\ll r$.
\end{proof}

\subsection{Numerics of $\omega(a)$.}\label{app:slab_num}

We used the MATLAB PDE Toolbox (quadratic finite elements) to solve (\ref{sec6:new_inner}).
For both the disk and the ball, axial symmetry then reduces  (\ref{sec6:new_inner}) to a
two-dimensional problem in the cylindrical coordinates $(\rho,Z)$, with
$\rho=\sqrt{X^2+Y^2}$, on the truncated half-strip $0<\rho<\rho_{\max}$,
$0<Z<a$, where we solve $\nabla\cdot(\rho\nabla V_c)=0$. For the disk,
$V_c=0$ on $0<\rho<1$, $Z=0$; for the ball, the quarter disk $\rho^2+Z^2<1$
is removed and $V_c=0$ on the arc $\rho^2+Z^2=1$. On the far boundary $\rho=\rho_{\max}$ we
impose the unit flux $\rho\,\partial_\rho V_c=1$ and
homogeneous Neumann conditions on the remaining boundaries. 
Note that the $Z$-average
$\bar V(\rho)\equiv a^{-1}\int_0^a V_c\,dZ$ satisfies
$(\rho\bar V')'=0$ exactly and therefore matches the far-field 
behaviour
$\bar V\sim \log\rho+\omega$. 
We therefore read off
$\omega = \bar V(\rho_{\max})-\log\rho_{\max}$. We used $\rho_{\max}=1+4a$.
 Halving the
mesh size changed $\omega$ by at most $2\times10^{-6}$ (ball) and
$6\times10^{-5}$ (disk).  Doubling $\rho_{\max}$ changed $\omega$ by less than $10^{-7}.$

Table 
~\ref{tab:omega} compares $\omega(a)$, computed numerically from (\ref{sec6:new_inner}) as described
above, with the large-$a$ asymptotics (\ref{sec6:a_large}).

\begin{table}[htbp]
\centering
\begin{tabular}{ccc|ccc}
 \multicolumn{3}{c|}{ball ($C=1$)} & \multicolumn{3}{c}{disk ($C=2/\pi$)}\\
 $a$ & $\omega(a)$ & $\omega_{\rm asym}$ & $a$ & $\omega(a)$ & $\omega_{\rm asym}$\\ \hline
 1.05 & 0.1376 & 0.1921 & 0.25 & 0.1171 & 0.9699\\
 1.10 & 0.1501 & 0.1956 & 0.50 & 0.2524 & 0.6695\\
 1.25 & 0.1913 & 0.2178 & 1.00 & 0.6059 & 0.7617\\
 1.50 & 0.2741 & 0.2855 & 1.50 & 1.0633 & 1.1417\\
 2.00 & 0.4950 & 0.4978 & 2.00 & 1.5930 & 1.6394\\
 3.00 & 1.0919 & 1.0923 & 3.00 & 2.7832 & 2.8047\\
 5.00 & 2.5815 & 2.5815 & 5.00 & 5.4275 & 5.4355\\
\end{tabular}
\caption{The infinite slab capacitance $\omega(a)$ of (\ref{sec6:a_large}), computed by
finite elements, compared with the large-$a$ asymptotics
$\omega_{\rm asym}=a/C-\log(4a)+\gamma_e$ from  (\ref{sec6:new_inner}), for a unit ball and a
flat unit disk. The numerical values are accurate to all digits shown.}
\label{tab:omega}
\end{table}

The agreement is excellent for
large $a$. The asymptotic formula converges very fast for a ball, with a
relative error of only $0.56\%$ at $a=2$. As expected, the convergence is
slower for a disk, with a relative error of $2.9\%$ at $a=2$.
\end{appendices}

\bibliographystyle{plain} \bibliography{references}
\end{document}